\documentclass{amsart}

\usepackage{amssymb}
\usepackage{amsfonts}
\usepackage{amscd}
\usepackage{lscape}
\usepackage{verbatim}
\usepackage{epsfig}
\usepackage{xr}

\input xy
\xyoption{all}

\DeclareFontFamily{OT1}{rsfs}{}
\DeclareFontShape{OT1}{rsfs}{n}{it}{<-> rsfs10}{}
\DeclareMathAlphabet{\mathscr}{OT1}{rsfs}{n}{it}

\newtheorem{thm}{Theorem}[subsection]
\newtheorem{lem}[thm]{Lemma}
\newtheorem{prop}[thm]{Proposition}
\newtheorem{propdef}[thm]{Proposition-Definition}
\newtheorem{cor}[thm]{Corollary}
\newtheorem{rems}[thm]{Remarks}

\theoremstyle{definition}

  \newtheorem{defi}[thm]{Definition}
  \newtheorem{rem}[thm]{Remark}

  \newtheorem{ack}{Acknowledgements}   

\numberwithin{equation}{thm}

\newcommand{\Cref}[1]{Corollary~\textup{\ref{#1}}}
\newcommand{\Dref}[1]{Definition~\textup{\ref{#1}}}

\newcommand{\Lref}[1]{Lemma~\textup{\ref{#1}}}
\newcommand{\Pref}[1]{Proposition~\textup{\ref{#1}}}
\newcommand{\Rref}[1]{Remark~\textup{\ref{#1}}}
\newcommand{\Sref}[1]{Section~\textup{\ref{#1}}}
\newcommand{\Ssref}[1]{Subsection~\textup{\ref{#1}}}
\newcommand{\Tref}[1]{Theorem~\textup{\ref{#1}}}

\def\bilap#1{\hbox to 0pt{\hss#1\hss}}
  \def\Rarrow#1{\bilap{\hbox to#1{\rightarrowfill}}}
  \def\Larrow#1{\bilap{\hbox to#1{\leftarrowfill}}}
  \def\Equals#1{\bilap
                   {\raise 4pt\hbox
                     {\vrule width#1 height.5pt}%
                    \kern-#1\raise 1pt\hbox
                     {\vrule width#1 height.5pt}%
                   }}

\newcommand{\EQAL}[1]%
{\,\begin{picture}(#1,0)%
\put(0,3){\line(1,0){#1}}%
\put(0,1){\line(1,0){#1}}%
\end{picture}\,}%

\newcommand{\vlto}[1]%
{\,\begin{picture}(#1,3)%
\put(0,2){\vector(1,0){#1}}%
\end{picture}\,}%

\newcommand{\vllarrow}[1]%
{\,\begin{picture}(#1,3)%
\put(#1,2){\vector(-1,0){#1}}%
\end{picture}\,}%

\newcommand{\dirlm}[1]%
   {
      {\lim\hskip-1.58em\lower.65ex
        \hbox{$
                 {}_{\stackrel{\lower1ex\hbox
                                         {$\scriptstyle -\!\!\!\longrightarrow$}
                                       }{\vbox to0pt{\vss\vskip.6ex
                                             \hbox{$\scriptstyle{}^{#1}$}\vss}}
                    }
             $}
      }
\:}

\newcommand{\subdirlm}[1]%
   {
      {\lim\hskip-1.5em\lower.6ex
        \hbox{$
                    {}_{\stackrel{\lower1ex\hbox
                                            {$\scriptstyle\longrightarrow$}
                                 }{ ^{#1} }
                       }
              $}
      }
\:}

\newcommand{\inlm}[1]%
    {
       {\lim\hskip-1.58em\lower.65ex
         \hbox{$
                  {}_{\stackrel{\lower1ex\hbox
                                         {$\scriptstyle \longleftarrow\!\!\<-$}
                               }{\vbox to0pt{\vss\vskip.6ex
                                             \hbox{$\scriptstyle{}^{#1}$}\vss}}
                     }
              $}
       }
\:}

\def\hz#1{{\hbox to 0pt{#1}}}
\def\Iso{\vbox to 0pt{\vss\hbox{$\widetilde{\phantom{nn}}$}\vskip-7pt}}

\def\>{\mspace {1mu}}
\def\<{\mspace{-1mu}}
\def\({{\textup(}}
\def\){{\textup)}}
\def\bigl#1{{\textup{\begin{large}#1\end{large}}}}
\def\bigr#1{{\textup{\begin{large}#1\end{large}}}}

\newcommand{\btrg}{\blacktriangle}
\newcommand{\fm}{{\mathfrak{m}}}

\newcommand{\fp}{{\mathfrak{p}}}

\newcommand{\X}{{\mathscr X}}

\newcommand{\Y}{{\mathscr Y}}
\newcommand{\Z}{{\mathscr Z}}
\newcommand{\V}{{\mathscr V}}
\newcommand{\U}{{\mathscr U}}
\newcommand{\W}{{\mathscr W}}
\newcommand{\I}{{\mathscr I}}
\newcommand{\J}{{\mathscr J}}

\newcommand{\eC}{{\mathscr C}}

\newcommand{\eE}{{\mathscr E}}
\newcommand{\eH}{{\mathscr H}}

\newcommand{\eM}{{\mathscr M}}
\newcommand{\eN}{{\mathscr N}}
\newcommand{\co}{{\mathscr O}}
\newcommand{\eF}{{\mathscr F}}
\newcommand{\eG}{{\mathscr G}}

\newcommand{\Spec}{{\mathrm {Spec}}}
\newcommand{\Spf}{{\mathrm {Spf}}}

\newcommand{\ares}[1]{\ush{\boldsymbol {\mathrm{res}}}_{\<\<\<\<\<\<\<\<\<{}_#1}}

\newcommand{\A}{{\mathcal A}}

\newcommand{\F}{{\mathcal F}}

\newcommand{\Hr}{{\mathrm H}}

\newcommand{\Rr}{{\mathrm R}}
\newcommand{\M}{{\mathcal M}}
\newcommand{\N}{{\mathcal N}}

\newcommand{\cO}{{\mathcal O}}

\newcommand{\eK}{{\mathscr K}}

\newcommand{\D}{{\mathbf D}}
\newcommand{\K}{{\mathbf K}}
\newcommand{\C}{{\mathbf C}}

\newcommand{\bbG}{{\mathbb G}}

\newcommand{\vc}{{\vec{\mathrm{c}}}}

\newcommand{\Dqc}{\D_{\mkern-1.5mu\mathrm {qc}}}
\newcommand{\wDqc}{ \widetilde
          {\vbox to6.5pt{\vss\hbox{$\mathbf D$}}}
    _{\mkern-1.5mu\mathrm {qc}} }

\newcommand{\wDqcp}{\wDqc^{\lower.5ex\hbox{$\scriptstyle+$}}}

\newcommand{\Dvc}{\D_{\<\vc}}
\newcommand{\Dqct}{\D_{\mkern-1.5mu\mathrm{qct}}}

\newcommand{\Dc}{\D_{\mkern-1.5mu\mathrm c}}

\newcommand{\qc}{{\mathrm{qc}}}

\newcommand{\R}{{\mathbf R}}
\newcommand{\Rp}[1]{{{\Rr}'}_{\<\<\<{{#1}}}}
\newcommand{\Rfs}{{\mathbf R f_{\!*}}}
\newcommand{\bL}{{\mathbf L}}
\newcommand{\Hom}{{\mathrm {Hom}}}

\newcommand{\Homb}{{\mathrm {Hom}}^{\bullet}}
\newcommand{\Ac}{\A_{\mathrm c}}
\newcommand{\Aqc}{\A_{\qc}}
\newcommand{\Avc}{\A_{\vec {\mathrm c}}}
\newcommand{\Aqct}{\A_{\mathrm {qct}}\<}

\newcommand{\ft}{f_{\mathrm t}^\times}

\newcommand{\shr}[1]{{#1}^{\boldsymbol{!}}}

\newcommand{\ush}[1]{{#1^{\textup{\texttt\#}}}}
\newcommand{\Tr}[1]{{\mathrm {Tr}}_{#1}}

\newcommand{\fTr}[1]{{\mathrm {Tr}}^{\btrg}_{{#1}}}
\newcommand{\ttr}[1]{{\boldsymbol{\ush{\tau}_{\!\!\!\!{}_{#1}}}}}

\newcommand{\tin}[1]{{\ush{\mathrm{tr}}_{\!\!{{#1}}}}}
\newcommand{\omgs}[1]{{\ush{\omega_{\<\<{#1}}}}}

\newcommand{\wnor}[1]{{{\eN}}^r_{{#1}}}
\newcommand{\wI}[2]{({\wedge_{#1}^r{#2}/{#2}^2})^*}

\newcommand{\Ext}{\operatorname{\eE{\mathit{xt}}}}
\newcommand{\ext}{\operatorname{Ext}}
\newcommand{\sHom}{\eH{om}}
\newcommand{\sHomb}{\eH{om}^{\bullet}}
\newcommand{\iGp}[1]{{\varGamma_{\<\!#1}'}}
\newcommand{\iG}[1]{{\varGamma_{\<\!#1}^{\phantom\prime}}}

\newcommand{\wid}[1]{\widehat{#1}}
\newcommand{\wit}[1]{\widetilde{#1}}
\newcommand{\set}{\!:=}
\newcommand{\lra}{\longrightarrow}
\newcommand{\iso}%
{{\mkern8mu\longrightarrow \mkern-25.5mu{}^\sim\mkern17mu}}
\newcommand{\osi}%
{{\mkern8mu\longleftarrow \mkern-24.5mu{}^\sim\mkern16mu}}
\newcommand{\Otimes}{\underset
   {\vbox to 0pt {\vskip-1ex\hbox{$\scriptscriptstyle=$}\vss}}
     {\otimes}\vadjust{\kern.4pt}}

\newcommand{\BL}{{\boldsymbol\Lambda}}

\newcommand{\smcirc}%
   {{\raise.15ex\hbox to.7em{$\hss \scriptstyle\circ\hss$}}}

\newcommand{\bLambda}{{\mathbf \Lambda}}

\title[Transitivity for formal schemes I]%
{Grothendieck duality and transitivity I: Formal schemes}

\author[S.\,Nayak]{Suresh Nayak}%
{\smash{}}
\address{Indian Statistical Institute \\
8th Mile, Mysore Road, Bangalore \\
Karnataka-560059, INDIA}
\email{snayak@isibang.ac.in}

\author[P.\,Sastry]{Pramathanath Sastry}%
{\smash{}}
\address{Chennai Mathematical Institute\\
                 Sipcot IT park, Siruseri\\
                 Kanchipuram Dist TN, 603103, INDIA}
\email{pramath@cmi.ac.in}
\makeindex

\date{\today}

\begin{document}
%{\bf\hfill \today}

\begin{abstract}{For a proper map $f\colon X\to Y$ of noetherian 
ordinary schemes, one has a well-known natural transformation, 
$\bL f^*({\boldsymbol{-}})\overset{\bL}{\otimes}f^!\co_Y\to f^!$,
 obtained via the projection formula, which extends, using Nagata's 
compactification, to the case where $f$ is separated and of finite type. In this paper we extend this
transformation to the situation where $f$ is a pseudo-finite-type map of noetherian
formal schemes which is a composite of compactifiable maps, and
show it is compatible with the pseudofunctorial structures involved. This natural transformation 
has implications for the abstract theory of residues and traces, giving
Fubini type results for iterated maps. These abstractions are rendered concrete 
in a sequel to this paper.}
\end{abstract}

\maketitle

%\setcounter{tocdepth}{2}
%\tableofcontents
%\newpage
%\input sec1.tex

%\newpage

All schemes are assumed to be noetherian. They could be formal or ordinary. For any
formal scheme $\X$, $\A(\X)$ is the category of $\co_\X$-modules and $\D(\X)$ its
derived category. The torsion functor $\iGp{\X}$ on $\co_\X$-modules
is defined by the formula
\[\iGp{\X} \set \dirlm{n}\sHom_{\co_\X}(\co_\X/\I^n, -)\] 
where $\I$ is any ideal of definition of the formal scheme $\X$. A {\emph{torsion module}}
$\eF$ is an object in $\A(\X)$ such that $\iGp{\X}\eF=\eF$.
The category $\Ac(\X)$ (resp.~$\Avc(\X)$, resp.~$\Aqc(\X)$, resp.~$\Aqct(\X)$) 
is the category of coherent (resp.~direct limit of coherent, resp.~quasi-coherent, resp.~quasi-coherent
and torsion) $\co_\X$-modules. $\Dc(\X)$ (resp.~$\Dvc(\X)$, resp.~$\Dqc(\X)$, resp.~$\Dqct(\X)$)
is the subcategory of $\D(\X)$ of complexes having homology in $\Ac(\X)$
(resp.~$\Avc(\X)$, resp.~$\Aqc(\X)$, resp.~$\Aqct(\X)$),
while  $\Dc^*(\X)$, (resp.~$\Dvc^*(\X)$, resp.~$\Dqc^*(\X)$, resp.~$\Dqct^*(\X)$) 
for $*$ in $\{ b, +, - \}$ denotes the 
corresponding full subcategory whose homology is additionally, bounded, or bounded below,
or bounded above, accordingly. 

We will be using the notions of pseudo-proper, pseudo-finite-type, pseudo-finite maps used
in \cite{dfs}. For example if $f\colon \X\to \Y$ is a map of formal schemes, it is {\emph{pseudo-proper}}
if the map of ordinary schemes $f_{\<\<{}_0}\colon X\to Y$ obtained by quotienting
the structure sheaves $\co_\X$ and $\co_\Y$ by ideals of definition for $\X$ 
and $\Y$, is proper. If this is true
for one pair of defining ideals $(\I,\,\J)$ with $\J\subset \co_\Y$ and $\J\co_\X\subset \I\subset\co_\X$, 
then it is true for all such pairs. The definitions of pseudo-finite-type, pseudo-finite are analogous.

We assume familiarity with the viewpoint of duality that was initiated by Deligne in \cite{del},
especially as laid out  by Lipman in \cite[Chapter 4]{notes}. In particular, we assume
familiarity with the main properties of the 
{\emph{inverse-image pseudofunctor}}~$\shr{\boldsymbol{(-)}}$. 
\pagebreak
This is a $\Dqc^+$-valued contravariant pseudofunctor on the category of schemes and separated essentially finite-type maps.

\marginpar{\phantom{X}}

\section{\bf Introduction}\label{s:intro}
We begin by recalling the notion of a (contravariant) pseudofunctor.  A
{\emph{pseudofunctor}} $\boldsymbol{(-)^\triangle}$ on a category $\eC$
is an assignment of categories $X^\triangle$, one for each $X\in\eC$, such that for each map
$f\colon X\to Y$ in $\eC$, we have a functor $f^\triangle\colon Y^\triangle \to X^\triangle$,
for each $X\in\eC$, an isomorphism
$\eta^\triangle_X \colon {\bf 1}_X^\triangle\iso {\bf 1}_{X^\triangle}$, for
each pair of composable maps $X\xrightarrow{f} Y \xrightarrow{g} Z$ in $\eC$,
an isomorphism of functors $C^\triangle_{f,g}\colon (gf)^\triangle \iso f^\triangle g^\triangle$, such that 
for a third map $h\colon Z\to S$ in $\eC$, ``associativity" holds, i.e., the following
diagram commutes,
%$C^\triangle_{f,hg}\<\smcirc\< f^\triangle(C^\triangle_{g,h}) = (C^\triangle_{f,g})(h^\triangle)\<\smcirc\<C^\triangle_{gf,h}$
%holds.
\[
{\xymatrix{
 f^\triangle g^\triangle h^\triangle \ar[d]^{\,\rotatebox{-90}{\makebox[-0.1cm]{\Iso}}}_{C^\triangle_{g,h}}
  \ar[rr]^\Iso_{C^\triangle_{f,g}} & &
(gf)^\triangle h^\triangle \ar[d]_{\,\rotatebox{-90}{\makebox[-0.1cm]{\Iso}}}^{C^\triangle_{gf, h}} \\
f^\triangle(hg)^\triangle  \ar[rr]^{\Iso}_{C^\triangle_{f, hg}} & & (hgf)^\triangle .
}}
\]
and for the compositions ${\bf 1}_Yf = f = f{\bf 1}_X$, the isomorphisms $\eta^\triangle_{-}$ and 
${C^\triangle_{-,-}}$ are compatible in the obvious way.
However, as we shall see, we may need to relax the condition
that $\eta^\triangle_X \colon {\bf 1}_X^\triangle\to {\bf 1}_{X^\triangle}$ be an isomorphism for
objects $X\in \eC$. In such a case, if other conditions are satisfied, $\boldsymbol{(-)^\triangle}$ is
called a {\emph{pre-pseudofunctor}}. For simplicity, we may often call these naturally occurring
pre-pseudofunctors as pseudofunctors.

\subsection{}\label{ss:lip-trans} The principal aim of this paper is to extend results in 
\cite[\S\,4.9]{notes}
to the situation of formal schemes. An example of the kind of result in {\it{loc.cit.}}~that would
interest us is the existence of a map 
\stepcounter{thm}
\begin{equation*}\label{map:lip-trans}\tag{\thethm}
\bL f^* g^!\co_Z\overset{\bL}{\otimes}_{\co_X} f^!\co_Y \lra (gf)^!\co_Z
\end{equation*}
for a pair of finite-type separated maps $X\xrightarrow{f} Y \xrightarrow{g} Z$. This is closely
related to the results in \cite{jag}.  Let us discuss \eqref{map:lip-trans} in some detail to
orient the reader to the kind of questions  we are interested in as well as the difficulties involved in answering them for formal schemes. 

Recall that if $h\colon V\to W$ is a proper map of schemes, then the {\emph{twisted inverse 
image functor}} $h^!\colon \Dqc^+(W)\to \Dqc^+(V)$ is a right adjoint to 
$\R h_*\colon \Dqc^+(V)\to  \Dqc^+(W)$.
We therefore have the co-adjoint unit 
\stepcounter{thm}
\begin{equation*}\label{map:intro-trace}\tag{\thethm}
\Tr{h}\colon \R h_*h^!\to {\bf 1}_{\Dqc^+(W)},
\end{equation*}
 the so-called 
{\emph{trace map}} for $h$, such that the map 
\[\Hom_{\Dqc^+(V)}(\eF,\,h^!\eG) \to \Hom_{\Dqc^+(W)}(\R h_*\eF, \eG)\]
given by $\varphi\mapsto \Tr{h}(\eG)\<\smcirc\<\R h_*(\varphi)$
is an isomorphism \cite[p.204, Theorem\,4.8.1\,(i)]{notes}. Next, if $h$ is \'etale, by part\,(ii) of {\textit{loc.cit.}, we have $h^!=h^*$. Note that if %$h$ is proper (and far more generally),
$\eF\in \Dqc(X)$ and  $\eG\in\Dqc(Y)$, 
we have the bifunctorial {\emph{projection isomorphism}} of \cite[p.139, Proposition 3.9.4]{notes}
\stepcounter{thm}
\begin{equation*}\label{iso:intro-proj-formula}\tag{\thethm}
\eG\overset{\bL}{\otimes}_{\co_Y}\R h_*\eF\iso \R h_*(\bL h^*\eG\otimes_{\co_X}\eF).
\end{equation*}
Finally, recall that if $h$ is separated and of finite type, by a famous theorem of Nagata \cite{nagata} one can find a compactification of $h$,
i.e., a factorization $h=\bar{h}\<\smcirc\< i$ with $i$ an {\emph{open immersion}} and
$\bar{h}$ a {\emph{proper}} map. Thus, after choosing such a compactification, one could
define $h^!$ for such maps, by the formula $h^!=i^*{\bar{h}}^!$. That this is independent of
the compactification chosen is proven in \cite{del}. The technical difficulties encountered
carrying out this program are formidable, and form the content of \cite{del}, \cite{del-sga}, 
\cite{del-sga'}, and \cite[Chapter 4]{notes}.

The map \eqref{map:lip-trans} is described as follows. First suppose $f$ and $g$
are proper. One has a natural map 
\stepcounter{thm}
\begin{equation*}\label{intro-star}\tag{\thethm}
\R(gf)_*(\bL f^* g^!\co_Z\overset{\bL}{\otimes}_{\co_X} f^!\co_Y) \lra \co_Z
\end{equation*}
given by the composite
\begin{align*}
\R(gf)_*(\bL f^* g^!\co_Z\overset{\bL}{\otimes}_{\co_X} f^!\co_Y)
& \xrightarrow{\phantom{XX}\Iso\phantom{XXX}} \R g_*\Rfs(\bL f^* g^!\co_Z\overset{\bL}{\otimes}_{\co_X} f^!\co_Y)\\
& \xrightarrow{\phantom{X}{\eqref{iso:intro-proj-formula}}^{-1}\phantom{X}\,} 
\R g_*(g^!\co_Z\overset{\bL}{\otimes}_{\co_Y}\Rfs f^!\co_Y) \\
& \xrightarrow{\,\R g_*({\bf 1}\otimes \Tr{f})\,\,} \R g_*g^!\co_Z \\
& \xrightarrow{\phantom{XXX}\Tr{g}\phantom{XX}} \co_Z .
\end{align*}
Since $(gf)^!$ is right adjoint to $\R (gf)_*$, the map \eqref{intro-star} gives rise to \eqref{map:lip-trans}
as the unique map such that 
$\Tr{gf}(\co_Z)\<\smcirc\<\R (gf)_*\eqref{map:lip-trans}=\eqref{intro-star}$.

If $f$ or $g$ is not proper, one can find a compactification of $gf$, say $gf=F\smcirc j$, such 
that $F= \bar{g}\<\smcirc\<\bar{f}$, with $\bar{f}$, $\bar{g}$ proper maps, 
and where these maps embed into a commutative diagram 
\stepcounter{thm}
\[
\begin{aligned}\label{intro-**}
{\xymatrix{
X \ar[d]_f \ar[r] & \overline{X} \ar[r] \ar[dl] \ar[r] & \overline{\overline{X}} \ar[dl]^{\bar f} \\
Y \ar[d]_g \ar[r] & \overline{Y} \ar[dl]^{\bar g} & \\
Z && 
}}
\end{aligned}\tag{\thethm}
\]
with all horizontal arrows open immersions and all south-west pointing arrows proper, with the
composite of the two horizontal arrows on the top row being $j$. The map \eqref{map:lip-trans} for the pair $(f,g)$ can be defined by ``restricting" the
corresponding map for $(\bar{f}, \bar{g})$ to $X$. The map \eqref{map:lip-trans} is independent
the choice of such diagrams --- this is the essential content of 
\cite[pp.\,231--232, Lemma\,4.9.2]{notes}. We provide a proof in this paper for formal schemes
in \Pref{prop:chi-welldef} below.

The map
 \eqref{map:lip-trans} is to be regarded as an abstract form of certain
 transitivity results for differential
forms which are important for duality, e.g.,  property (R4) of residues stated in \cite[p.198]{RD}.
Briefly, if $h\colon V\to W$ is a smooth
map of relative dimension $n$, then one can show that there is an isomorphism
\stepcounter{thm}
\begin{equation*}\label{intro-dag}\tag{\thethm}
\omega_h[n] \iso h^!\co_W, 
\end{equation*}
where $\omega_h=\Omega^n_{V/W}$ is the $n^{\rm th}$
exterior power of the $\co_V$-module of relative differential forms $\Omega^1_{V/W}$
 for the map $h$. There are many descriptions of such isomorphisms (see
 \cite[p.\,397, Theorem 3]{verdier}, \cite[p.\,84, Duality theorem]{hk2} for projective maps with the 
 base which is free of embedded points, 
 \cite[p.\,750, Duality Theorem]{ajm} for maps where the base is free of embedded points).
Thus if $f$ and $g$ are smooth of relative dimensions, say, $m$ and $n$ respectively, then, 
upon taking homology in degree $m+n$, \eqref{map:lip-trans} induces a map of coherent 
$\co_X$-modules,
\stepcounter{thm}
\begin{equation*}\label{intro-ddag}\tag{\thethm}
f^*\omega_g\otimes_{\co_X}\omega_f \to \omega_{gf}. 
\end{equation*}
The above map is an isomorphism since \eqref{map:lip-trans} is in this special case, $f$ being
a perfect map.
How this compares with the usual isomorphism between $f^*\omega_g\otimes_{\co_X}\omega_f$ 
and $\omega_{gf}$ depends on
the choice of \eqref{intro-dag} which
in turn depends on the choice of 
concrete trace maps of the form $\R h_*\omega_h[n]\to \co_W$. 
For the one implicit in \cite{ajm}, the problem is studied in \cite{jag} (see also the correction). 
In a sequel to this paper we will show that when \eqref{intro-dag} 
is chosen to be  the isomorphism given
by Verdier in \cite[p.\,397, Thm.\,3]{verdier}, the map \eqref{intro-ddag} is the map given locally by 
$f^*(\mu)\otimes\nu\mapsto \nu\wedge f^*(\mu)$ where the notation is self-explanatory.

The problem of finding the concrete expression for \eqref{intro-ddag} given \eqref{intro-dag} is best attacked by expanding the scope of our study from ordinary schemes 
to formal schemes. One reason for this is that maps such as \eqref{map:lip-trans} are compatible 
(in a precise sense) with completions
along closed subschemes in $X$, $Y$, and $Z$.  This allows for far greater flexibility than the
method of compactifying and restricting. From this larger point of view, property (R4) of
residues in \cite[p.\,198]{RD} is a concrete manifestation of \eqref{map:lip-trans} for maps
between formal schemes.

 In this paper we lay the foundations for all of this. The generalization of parts of \S 4.9 of
 \cite{notes} (in particular of  the map \eqref{map:lip-trans} above) is carried out in section
 \Sref{s:fubini-1}, especially in \Pref{prop:chi-welldef}.  As for property
(R4) for residues, an abstract form of it for Cohen-Macaulay maps is proved by us
in \Pref{prop:leray-chi} below. We draw the reader's attention
especially to formulas \eqref{eq:res-res-1} and \eqref{eq:res-res-2} which follow from \textit{loc.cit.}

%once isomorphisms 
%$\wedge^e_{\co_Y}\Omega^1_{Y/Z}[e]\cong g^!\co_Z$, 
%$\wedge^d_{\co_X}\Omega^1_{X/Y}[d]\cong f^!\co_Y$, and 
%$\wedge^{d+e}_{\co_X}\Omega^1_{X/Z}[d+e]\cong (gf)^!\co_Z$ are fixed. One of the tests for 
%a consistent choice of such isomorphisms being the ``correct" one is that the map
%$f^*\wedge^e_{\co_Y}\Omega^1_{Y/Z}\otimes_{\co_X}\wedge^d_{\co_X}\Omega^1_{X/Y}
%\to \wedge^{d+e}_{\co_X}\Omega^1_{X/Z}$ be the natural one between differential forms.

%While general theorems exist for the existence of right adjoints to direct image functors
%$\Rfs$ associated with a map of schemes $f$, they are of little interest for 
%duality theory unless the map $f$ is pseudo-proper (which, if the schemes involved are
%ordinary, is the same as saying $f$ is proper). For example, if $f$ is a separated finite type
%map between ordinary schemes, which is not proper, then the correct duality functor $f^!$ is
%not the right adjoint to $\Rfs$ (which right adjoint exists), but $i^*p^!$ where $i$ is an open
%immersion, $p$ a proper map, $f=p\<\smcirc\< i$, and $p^!$ is right adjoint to $\R p_*$.
%Staying with ordinary schemes, the factorization of $f$ into an open immersion followed by a 
%proper map is called a {\emph{compactification}} of $f$, and a famous theorem of Nagata 
%\cite{nagata} says that when $f$ is separated and of finite type, compactifications of $f$ exist (the 
%conditions are clearly necessary for compactifiations to exist).

If we move to formal schemes which are not necessarily ordinary, 
\cite[Theorem\,6.1]{dfs} assures us of the existence
of a right adjoint to $\Rfs\colon\Dqct(\X)\to \Dqct(\Y)$ for a pseudo-proper map (and more)
$f\colon \X\to \Y$  and this right adjoint is denoted $f^!$.
However, given a general separated pseudo-finite-type map $f$, we
are no longer assured that $f$ has a compactification, i.e., we are no longer assured that we
have a factorization $f=\bar{f}\<\smcirc\< i$ with $i$ an open immersion, and ${\bar f}$ 
pseudo-proper.\footnote{On the other hand, we do not have an example of a
separated pseudo-finite-type map which does not have a compactification.} We are therefore forced to work in the  category $\bbG$
whose objects are formal schemes and whose morphisms are composites of ``compactifiable"
maps. In \cite{pasting} the first author shows that we have a pre-pseudofunctor 
$\shr{\boldsymbol{(-)}}$
on $\bbG$, which generalises what we have for the category of finite-type separated maps
on ordinary schemes (see \Sref{s:shriek-sharp}).

If $k$ is a field and $A=k[|X_1,\dots, X_d|]$ the ring of power series over $k$ in $d$ 
analytically independent variables, $\fm$ the maximal ideal of $A$,
$\X$ the formal spectrum of $A$ with its $\fm$-adic topology, and $\Y$ the spectrum of $k$,
then the natural map $f\colon \X\to \Y$ is pseudo-proper and 
$f^!\co_\Y$ is the torsion sheaf obtained by sheafifying the $A$-module 
$\Hr^d_\fm(\omega_A)$ where $\omega_A$ is ``the" canonical module of $A$.
From here to recovering local duality for the complete local ring $A$ requires a more
careful examination of the relationship between $f^!\co_\Y$ and the canonical module
$\omega_A$. As it turns out, $\omega_A$,  a finitely generated $A$-module,
can be recovered from the torsion module associated to $f^!\co_\Y$. More generally,
for a pseudo-proper map $f\colon \X\to \Y$ between formal schemes and an object
$\eG\in\Dc^+(\Y)$, there is a deep
relationship between $f^!\eG \in\Dqct(\X)$ and an associated object in $\Dc^+(\X)$, of which the
above mentioned relationship between $\Hr^d_\fm(\omega_A)$ and $\omega_A$ is an example. This
necessitates the development of a second twisted inverse image functor $\ush{f}$ related to
$f^!$. The twisted inverse image $\ush{f}$  was introduced by Alonso, J\'erem\'{\i}as, and Lipman
in \cite{dfs} and the relationship between $f^!$ and $\ush{f}$ is one of the many important
portions of that work. 

This paper is mainly concerned with a pre-pseudofunctor $\ush{\boldsymbol{(-)}}$ on 
the category $\bbG$ such that for $f$ pseudo-proper, $\ush{f}$ is the functor mentioned
in the last paragraph. The pseudofunctor $\ush{\boldsymbol{(-)}}$ is 
one of two ways that the twisted inverse image pseudofunctor $\shr{\boldsymbol{(-)}}$ on
the category of ordinary schemes and separated finite-type maps generalizes to $\bbG$,
the other being the pre-pseudofunctor $\shr{\boldsymbol{(-)}}$ on $\bbG$ discussed
above. The map \eqref{map:lip-trans} can 
be defined for formal schemes with $f^!$ and $g^!$ replaced by $\ush{f}$ and $\ush{g}$. 
The principal technical issue which creates complications is the lack
of diagrams like \eqref{intro-**} into which a pair of maps $\X\xrightarrow{f} \Y \xrightarrow{g} \Z$
embed. In the next sub-section we give brief introduction to $\shr{\boldsymbol{(-)}}$
and $\ush{\boldsymbol{(-)}}$ on $\bbG$.

\subsection{Two twisted inverse images}\label{ss:2-twisted}
The duality pseudofunctors $\ush{\boldsymbol{(-)}}$ and $\shr{\boldsymbol{(-)}}$
on $\bbG$ are explained in \Sref{s:shriek-sharp} of this paper, but for the purposes of this 
introduction we say a few quick words. For a pseudo-proper map $f\colon\X\to\Y$, the
functor $f^!\colon \Dqct^+(\Y)\to \Dqct^+(\X)$ is right adjoint to $\Rfs\colon \Dqct^+(\X)\to
\Dqct^+(\Y)$.  In fact $f^!$ extends to a larger category $\wDqcp(\Y)$ which contains
$\Dqct^+(\Y)$, namely the full subcategory of $\D(\Y)$ of objects $\eG$ such that
$\R\iGp{\Y}(\eG)\in\Dqc^+(\Y)$, the extended functor
being $f^!\<\<\smcirc\<\R\iGp{\Y}$. Note that this extended $f^!$ continues to take values
in $\Dqct^+(\X)$.  
There is another duality functor associated with the pseudo-proper map $f$, namely 
$\ush{f}\colon\wDqcp(\Y)\to \wDqcp(\X)$, which is right adjoint to 
 $\Rfs\R\iGp{\X}\colon \wDqcp(\X)\to \wDqcp(\Y)$. 
 
  The functors $\ush{f}$ and $\shr{f}$ are related via the formulas $\shr{f}\cong\R\iGp{\X}\ush{f}$
 and $\ush{f}\cong \BL_{\X}f^!$, where $\BL_{\X}(-) = \R\sHom(\R\iGp{\X}\co_{\X}, -)$.
  
 {\scriptsize{As an example, if $k$ is a field, $\X$ the formal spectrum of the power series ring
 $A=k[|X_1,\dots, X_d|]$ (given the $\fm$-adic topology, where $\fm$ is the maximal
 ideal of $A$), $\Y=\Spec{\,k}$, and $f\colon \X\to \Y$ the map of formal schemes corresponding
 to the obvious $k$-algebra map $k\to k[|X_1,\dots, X_d|]$, then $f$ is pseudo-proper. 
 Identifying $A$-modules with their associated sheaves on $\X$, and writing
 $\widehat{\Omega}^d_{A/k}$ for the universally
 finite module of $d$-forms for the algebra $A/k$, we have 
 $f^!(k)=\Hr^d_\fm(\widehat{\Omega}^d_{A/k})[0]$ and
 $\ush{f}(k)=\widehat{\Omega}^d_{A/k}[d]$.}}
 
 This is for pseudo-proper maps, the original setting for defining $f^!$ in \cite{dfs}. However,
 in \cite{pasting}, the first author was able to show that
 $f^!\colon \wDqcp(\Y)\to \wDqcp(\X)$ can be defined when $f\colon \X\to \Y$ is in $\bbG$,
 even when it is not pseudo-proper, in such a way that (a) when $f$ is an open immersion,
 $f^!\cong f^*\R\iGp{\Y}=\R\iGp{\X}f^*$, and (b) such that the resulting variance theory
 $\shr{\boldsymbol{(-)}}$ is a pre-pseudofunctor. In fact, ${\bf 1}_{\X}^! \iso \R\iGp{\X}$,
 and the latter functor is not isomorphic to the identity functor.
 
 For a map $f\colon \X\to \Y$ in $\bbG$, we set $\ush{f}=\BL_\X(f^!)$. The source of $\ush{f}$
is  $\wDqcp(\Y)$ and its target is $\wDqcp(\X)$, so that $\ush{f}\colon\wDqcp(\Y)\to \wDqcp(\X)$. 
If $f$ is pseudo-proper this definition of $\ush{f}$ agrees 
 with the earlier one (as the functor which is right adjoint to $\Rfs\R\iGp{\X}$). The variance theory
 $\ush{\boldsymbol{(-)}}$ is a pre-pseudofunctor with $\ush{\bf 1_\X} \iso \BL_\X$. Our $\ush{f}$
 agrees with the one in \cite{dfs} only when $f$ is pseudo-proper.
 
 The paper is organized as follows. The definitions of $\shr{\boldsymbol{(-)}}$ and 
 $\ush{\boldsymbol{(-)}}$ and their first properties are given in \Sref{s:shriek-sharp}. 
 Sections \ref{s:CM} and \ref{s:CM-tr} deal with Cohen-Macaulay maps. The meat of the
 paper is in Sections \ref{s:fubini-1} and \ref{s:fubini-res}. There is an appendix which
 contains a number of results useful in the main body of the text. We have placed these results
 in the appendix so that the main narrative is not broken into disconnected bits.
 
%\subsection{The smooth case} 
%For a smooth map $f\colon X\to Y$ of relative dimension $d$, let 
%$\omega_f=\wedge^d_{\co_X}\Omega^1_{X/Y}$, the
%$\co_X$-modules of relative $d$-forms for $f$.
%
%
%Transitivity in duality theory revolves around the following related set of memes. At
%its very basic, if $X\xrightarrow{f} Y\xrightarrow{g} Z$ is a pair of {\emph{smooth, proper}} 
%maps of relative
%dimensions $d$ and $e$ respectively, $\omega_f\set \wedge^d_{\co_X}\Omega^1_{X/Y}$ (the
%$\co_X$-modules of relative $1$-forms for $f$), 
%$\omega_g=\wedge^e_{\co_Y}\Omega^1_{Y/Z}$,
%$\omega_{gf}=\wedge^{d+e}_{\co_X }\Omega^1_{X/Z}$, 
%\[\bar{\varphi}= \bar{\varphi}_{g,f}\colon 
%f^*\omega_g\otimes_{\co_X} \omega_f
%\iso \omega_{gf}\]
%the isomorphism between differential forms given locally by 
%$f^*(\wdd{t_1}{t_e})\otimes\wdd{s_1}{s_d}\mapsto
%\wdd{s_1}{s_d}\wedge\wdd{t_1}{t_e}$, then one would like to have a theory of traces
%such that (with $\vin{f}\colon \Rr^df_*\omega_f\to \co_Z$, $\vin{g}\colon 
%\Rr^eg_*\omega_g\to\co_Z$,
%and $\vin{gf}\colon \Rr^{d+e}(gf)_*\omega_{gf}$, the traces) such that the composite
%\[\Rr^{d+e}(gf)_*(f^*(\omega_g)\otimes \omega_f)
%\xrightarrow{\Rr^{d+e}f_*({\bar{\varphi}})} \Rr^{d+e}(gf)_*\omega_{gf} \xrightarrow{\vin{gf}} \co_Z\]
%is the composite
%\begin{align*}
%\Rr^{d+e}(gf)_*(f^*(\omega_g)\otimes \omega_f)
%& \xrightarrow[{\phantom{XX}\text{Leray}}\phantom{XX}]{\Iso} \Rr^e g_*(\omega_g\otimes_{\co_Y}
%\Rr^df_*(\omega_f))\\
%& \xrightarrow{\Rr^eg_*({\bf 1}\otimes \vin{f})}  \Rr^eg_*(\omega_g) \\
%& \xrightarrow{\phantom{XXX}\vin{g}\phantom{XX}} \co_Z
%\end{align*}

\section{\bf Notations and basics on formal schemes}\label{s:notations}

\subsection{} 
We discuss some basic matters on formal schemes and the derived categories of complexes on them. 
Most of what we say here can be found with more details in \cite{dfs}. Also, for the basic conventions
on derived functors we refer to \cite{notes}.

For any formal scheme $\X$ and any coherent ideal $\I$ in $\co_{\X}$, $\iG{\I}$ denotes the functor that 
assigns to any $\co_{\X}$-module $\F$, the submodule of sections of $\F$ annihilated locally by some power
of $\I$. The torsion functor $\iGp{\X}$ is the one that assigns to any~$\F$
the submodule of sections of $\F$ annihilated locally by some open ideal in~$\co_{\X}$. Thus for any 
defining ideal $\I$ for $\X$, $\iGp{\X} = \iG{\I}$. A torsion module is a module~$\F$ satisfying
$\iGp{\X}\F  = \F$. 

For any formal scheme $\X$, the abelian categories $\A_{?}(\X)$ for $?$ in 
$\{\text{c, } \vec{\text{c}}, \text{qc, qct}\}$ are defined as in the beginning of this paper
and the same applies to the definition of derived categories $\D_?^*(\X)$.

\begin{comment} %%%%%%%%%%%%%%%%%
For any
formal scheme $\X$, the category $\Ac(\X)$ (resp.~$\Avc(\X)$, resp.~$\Aqc(\X)$, resp.~$\Aqct(\X)$) 
is the category of coherent (resp.~direct limit of coherent, resp.~quasi-coherent, resp.~quasi-coherent
and torsion) $\co_\X$-modules. $\Dc(\X)$ (resp.~$\Dvc(\X)$, resp.~$\Dqc(\X)$, resp.~$\Dqct(\X)$)
is the subcategory of $\D(\X)$ of complexes having homology in $\Ac(\X)$
(resp.~$\Avc(\X)$, resp.~$\Aqc(\X)$, resp.~$\Aqct(\X)$),
while  $\Dc^*(\X)$, (resp.~$\Dvc(^*(\X)$, resp.~$\Dqc^*(\X)$, resp.~$\Dqct^*(\X)$) 
for $*$ in $\{ b, +, - \}$ denotes the 
corresponding full subcategory whose homology is additionally, bounded, or bounded below,
or bounded above, accordingly. 
\end{comment} %%%%%%%%%%%%%%%%%

We use the notation $\R F$ (resp.~$\bL F$) to denote the right (resp.~left) derived functor associated 
to any (triangulated) functor $F$ between derived categories, and for the derived tensor product
we use $\overset{\bL}{\otimes}$.

Let $\wDqc(\X)$ denote the triangulated full subcategory of  $\D(\X)$ whose objects
consist of complexes $\F$ such that $\R\iGp{\X}\F \in \Dqc(\X)$ (and hence $\R\iGp{\X} \F\in \Dqct(\X)$).
Similarly, $\wDqcp(\X)$ denotes the subcategory consisting of complexes $\F$ such that 
$\R\iGp{\X}\F \in \Dqc^+(\X)$. Thus there are full subcategories  
$\Dqct^+(\X) \subset \Dqc^+(\X) \subset \wDqcp(\X)$
which are all equal when $\X$ is an ordinary scheme. 

The functor $\R\iGp{\X} \colon \D(\X) \to \D(\X)$ has a right adjoint given by
\[
\BL_{\X}(-) = \R\sHom(\R\iGp{\X}\co_{\X}, -).
\]
Via canonical maps $\R\iGp{\X} \to 1 \to \BL_{\X}$, both $\R\iGp{\X}$ and $\BL_{\X}$ are idempotent
functors and in fact there are natural isomorphisms
\[
\R\iGp{\X}\R\iGp{\X} \iso \R\iGp{\X} \iso \R\iGp{\X}\BL_{\X}, \qquad \qquad
\BL_{\X}\R\iGp{\X} \iso \BL_{\X} \iso \BL_{\X}\BL_{\X}.
\]
In particular, $\BL_{\X}(\wDqcp(\X)) \subset \wDqcp(\X)$ and therefore, the restriction of
$\R\iGp{\X}$ to $\wDqcp(\X)$ and of $\BL_{\X}$ to $\Dqct^+(\X)$ also constitute an 
adjoint pair.

For $\eF \in \Dc(\X)$, the canonical map $\eF \to \BL_{\X}\eF$ is an isomorphism 
by Greenlees May duality \cite[Prop 6.2.1]{dfs}.
More generally, if $\Dvc(\X)$ is the subcategory of $\Dqc(\X)$ consisting of 
complexes whose homology sheaves are direct limits of coherent ones, 
then the restriction of $\BL_{\X}$ to $\Dvc(\X)$ is isomorphic to
the left-derived functor of the completion functor $\Lambda_{\X}$ which assigns to any
sheaf~$\F$, the inverse limit $\inlm{n} \F/\I^n\F$ where $\I$ is any defining ideal in $\co_{\X}$.
In contrast, note that $\R\iGp{\X}$ does not preserve coherence of homology in general.

Let $f \colon \X \to \Y$ be a map of noetherian formal schemes. 
%Let $f \colon \X \to \Y$ be a map in $\bbG$. 
Then there are natural isomorphisms (see \cite[Proposition 5.2.8]{dfs})
\stepcounter{thm}
\begin{gather*}\label{iso*gamma-lambda}\tag{\thethm}
\R\iGp{\X}\bL f^* \R\iGp{\Y} \iso \R\iGp{\X}\bL f^* \iso \R\iGp{\X}\bL f^* \BL_{\Y}, \\
\BL_{\X}\bL f^* \R\iGp{\Y} \iso \BL_{\X}\bL f^* \iso \BL_{\X}\bL f^* \BL_{\Y}. 
\end{gather*}
Here the first isomorphism in the first line follows easily from the fact that for any coherent ideal 
$\I \subset \co_{\Y}$, we have $\bL f^*\R\iG{\I} \cong \R\iG{\I\co_{\X}}\bL f^*$, a fact which 
can be checked locally using stable Koszul complexes, see \eqref{iso:k-infty-gam} in Appendix below for instance. 
The remaining isomorphisms in~\eqref{iso*gamma-lambda} result from the first one by pre-composing 
with~$\BL_{\X}$ or post-composing with~$\BL_{\Y}$.

For $f \colon \X \to \Y$ as above, $f^*$ sends torsion $\co_\Y$-modules to torsion $\co_\X$-modules
and hence we have $\bL f^*(\Dqct(\Y)) \subset \Dqct(\X)$ (in addition to the usual inclusions
$\bL f^*(\Dqc(\Y)) \subset \Dqc(\X), \bL f^*(\Dc(\Y)) \subset \Dc(\X)$). By  \eqref{iso*gamma-lambda}
we also deduce that $\bL f^*(\wDqc(\Y)) \subset \wDqc(\X)$.

Unlike the case of ordinary (noetherian) schemes, $\R f_*$ does not map $\Dqc^+(\X)$ (or even $\Dc^+(\X)$)
inside $\Dqc^+(\Y)$ in general. Under additional torsion conditions we do get the desired 
behaviour. Thus we have $\R f_*(\Dqct^+(\X)) \subset \Dqct^+(\Y)$ and therefore 
$\R f_*\R\iGp{\X}(\wDqcp(\X)) \subset \Dqct^+(\Y)$. 

For any morphism $f \colon \X \to \Y$ as above and for any closed subset $Z \subset \X$,
we set $\R_Zf_* \set \R f_*\R\iG{Z}$. Likewise we set
$\Rp{\X}f_* \set \R f_*\R\iGp{\X}$. For any integer $r$ we use 
$\R_Z^rf_* \set H^r\R_Zf_*$ and $\Rp{\X}^rf_* \set H^r\R\iGp{\X}f_*$.

\marginpar{}

\section{\bf The duality pseudofunctors over formal schemes} \label{s:shriek-sharp}

We mainly work with the category $\bbG$ of composites of open immersions 
and pseudo-proper maps between noetherian formal schemes. 
By Nagata's compactification theorem, every 
separated finite-type map of \emph{ordinary} schemes lies in $\bbG$.

\subsection{} 

The results in \cite{dfs} and \cite{pasting} extend the
theory of $\shr{(-)}$ over ordinary schemes to that
over~$\bbG$. Thus, there is a contravariant pseudofunctor
$(-)^!$ on~$\bbG$ with values in~$\Dqct^+(\X)$ for any formal scheme $\X$,
such that if $f \colon \X \to \Y$ is pseudo-proper, there exists a functorial map
$t_f\colon \Rfs f^! \to {\bf 1}_{\Dqct^+(\Y)}$ such that $(f^!,\,t_f)$ is a right adjoint
to $\R f_* \colon \Dqct^+(\X) \to \Dqct^+(\Y)$ while if~$f$ is an open immersion or more
generally, if~$f$ is adic \'etale, then $f^! = f^*$ pseudofunctorially. 
For a formally \'etale map~$f$ in~$\bbG$, 
(e.g., a completion map $\X \to X$ where~$X$ is an ordinary scheme
and~$\X$ its completion along some coherent ideal in $\cO_X$), we have $f^! \iso \R\iGp{\X}f^*$
(again pseudofunctorially), see \cite[Theorem 7.1.6]{pasting}.
Note that~$f^!$ does not preserve coherence of homology in general.

There is an extension of $(-)^!$ that we find convenient to use. 
For any $f \colon \X \to \Y$ in~$\bbG$ the composite $f^!\R\iGp{\Y}$ sends $\wDqcp(\Y)$ 
to $\Dqct^+(\X) \subset \wDqcp(\X)$ and is isomorphic to $f^!$ when restricted to $\Dqct^+(\Y)$.
The extended functor is also denoted as~$f^!$. 
However, this extension $(-)^!$ only forms a pre-pseudofunctor (see beginning of~\S\ref{s:intro}).
Thus, for $\X \xrightarrow{f} \Y \xrightarrow{g} \Z$ in~$\bbG$, the comparison isomorphism
$C_{f,g}^! \colon (gf)^! \iso f^!g^!$ is induced by the usual one over~$\Dqct^+(-)$ 
and by the isomorphisms
\[
(gf)^!\R\iGp{\Z} \iso {f^!}{g^!}\R\iGp{\Z}  \iso {f^!}\R\iGp{\Y}{g^!}\R\iGp{\Z}. 
\]
The associativity condition for $C_{-,-}^!$ vis-\'{a}-vis composition of 3 maps easily 
results from the corresponding one over $\Dqct^+(-)$.   
For the identity map~$1_{\X}$ on~$\X$ 
%however, instead of~$(1_{\X})^!$ being the identity functor, now we 
we have a natural map $1_{\X}^! = \R\iGp{\X} \to {\bf 1}$ 
and the comparison isomorphisms
corresponding to composing $f$ on the left or right by identity are the canonical
ones $f^! \iso f^!\R\iGp{\X}$ and $f^! \iso \R\iGp{\X}f^!$.

This extended pre-pseudofunctorial version of $(-)^!$ is what we will use
from now on. It has the following properties.
For $f \colon \X \to \Y$ in $\bbG$, if~$f$ is pseudoproper, then
$f^! \colon \wDqcp(\Y) \to \Dqct^+(\X)$ is right adjoint  
to $\R f_* \colon \Dqct^+(\X) \to \wDqcp(\Y)$,
%If $f$ is  adic \'etale, then $f^! = f^*\R\iGp{\Y}$ and more generally 
while if~$f$ is formally \'etale, then~$f^!$ 
is isomorphic to $\R\iGp{\X}f^*\R\iGp{\Y} \iso \R\iGp{\X}f^*$,
(see Theorem~7.1.6 and~\S 7.2 of~\cite{pasting}).

For $\Dc^+$-related questions, it is useful to work with another 
generalization of~$(-)^!$ from ordinary schemes. For $f \colon \X \to \Y$ in $\bbG$ 
we define $\ush{f} \colon \wDqcp(\Y) \to \wDqcp(\X)$ by the formula
\stepcounter{thm}
\begin{equation*}\label{def:sharp}\tag{\thethm}
\ush{f} \set \BL_{\X}f^!.
\end{equation*}
Since $f^! \iso \R\iGp{\X}\ush{f}$, therefore $(-)^!$ and $\ush{(-)}$ determine each other
upto isomorphism.

%It follows that if $f$ is pseudoproper, then $\ush{f}$ is right adjoint to 
%$\R f_*\R\iGp{\X}$ so that we have a trace map 
%\[ \Tr{f} \colon \R f_*\R\iGp{\X}\ush{f} \to {\bf 1}, \]
%while if $f$ is open, (or more generally, if 
%$f$ is formally \'etale) we have $\ush{f} \iso \BL_{\X}f^*$. 

Note that $\ush{(-)}$ is also a pre-pseudofunctor. 
For maps $\X \xrightarrow{f} \Y \xrightarrow{g} \Z$ in $\bbG$, the comparison isomorphism
$\ush{C_{f,g}} \colon \ush{(gf)} \iso \ush{f}\ush{g} $ is given by the composite 
\[
\ush{(gf)} = \BL_{\X}{(gf)^!} \iso \BL_{\X}{f^!}{g^!}
\iso \BL_{\X}{f^!}\BL_{\Y}{g^!} = \ush{f}\ush{g}
\]
where the last isomorphism is obtained from composite of the following sequence where 
we use $f^!\R\iGp{\Y} \iso f^!$ in the first and the last step:
\[
f^!g^! \osi f^!\R\iGp{\Y}g^! \iso  f^!\R\iGp{\Y}\BL_{\Y}g^! \iso f^!\BL_{\Y}g^!. 
\]
The associativity condition for $\ush{C_{-,-}}$ vis-\'{a}-vis composition of 3 maps easily 
results from the corresponding one for~$(-)^!$.  
For the identity map~$1_{\X}$ on~$\X$, 
there is a map ${\bf 1} \to \ush{(1_{\X})} = \BL_{\X}$ 
and the comparison isomorphisms
corresponding to composing $f$ on the left or right by identity are the canonical
ones $\ush{f} \iso \BL_{\X}\ush{f}$ and $\ush{f} \iso \ush{f}\BL_{\Y}$. 

If $f \colon \X \to \Y$ is a pseudo-proper map (whence a map in $\bbG$), then 
$\ush{f}$ is right adjoint to $\R f_*\R\iGp{\X}$ 
so that we have a co-adjoint unit, the so-called {\emph{trace map}}
\stepcounter{thm}
\begin{equation*}\label{map:Tr-f}\tag{\thethm}
\Tr{f} \colon \R f_*\R\iGp{\X}\ush{f} \to {\bf 1}, 
\end{equation*}
while if $f$ is an open immersion (or more generally, if $f$ is formally \'etale)
%while if $f$ is formally \'etale, 
then there is a natural isomorphism 
$\ush{f} \iso \BL_{\X}f^*$. Moreover these isomorphisms are 
pre-pseudofunctorial over the corresponding full subcategories of $\bbG$. 

A cautionary remark. In \cite[Prop.~6.1.4]{dfs},
$\ush{f}$ is defined as $\BL_{\X}\ft$ where $\ft$ is the right adjoint to the
restriction of $\Rfs$ to $\Dqct(\X)$. The functor $\BL_{\X}\ft$ is shown to be
be a right adjoint to $\R f_*\R\iGp{\X}$
for \emph{every} $f$. Our definition of $\ush{f}$ agrees with that of 
\cite{dfs} when $f$ is pseudo-proper, but not in general.

For any $f \colon \X \to \Y$ in $\bbG$, 
we have $\ush{f}(\Dc^+(\Y)) \subset \Dc^+(\X)$. This can be
seen by reducing to the special cases when~$f$ is pseudoproper or~$f$ is an open immersion;
in the latter case one uses that $\BL_{\X}\big|_{\Dc(\X)}$ 
%$\BL_{\X}\R\iGp{\X} \iso \BL_{\X}$ 
is isomorphic to the identity functor, so that $\ush{f} \iso \BL_{\X}f^* \iso f^*$ 
on~$\Dc(\X)$, while the former case
is dealt with in \cite[p.\,89, Proposition 8.3.2]{dfs}. Thus $\ush{(-)}$ gives 
a $\Dc^+$-valued pseudofunctor 
on~$\bbG$. It also follows that if~$f$ is formally \'etale and $\F \in \Dc^+(\Y)$, then we have
an isomorphism
\stepcounter{thm}
\begin{equation*}\label{eq:gm}\tag{\thethm}
f^*\F \iso \ush{f}\F
\end{equation*}
which is pseudofunctorial over the category of formally \'etale maps. If~$\Y$ is a formal scheme,
$\I \subset \co_\Y$ an open coherent ideal, $\W \set \widehat{\Y}$ the completion of~$\Y$ by~$\I$ and 
$\kappa \colon \W \to \Y$ the corresponding completion map,
then~$\kappa$ is both pseudoproper and \'etale. For any $\eF \in \Dc^+(\Y)$,
the isomorphism of~\eqref{eq:gm} is the same (see Lemma~\ref{lem:kappa} in Appendix below) 
as the map adjoint to the natural composite
\[
\R\kappa_*\R\iGp{\W}\kappa^* \iso \R\iG{\I} \to \mathbf{1}.
\]

Nevertheless, while working with $\ush{(-)}$, it is also convenient to work with the larger category 
$\wDqcp$ in addition to $\Dc^+$ since, some of the functors, 
such as $\R f_*\R\iGp{\X}$, the left adjoint to~$\ush{f}$ 
when $f \colon \X \to \Y$ is pseudoproper,
do not preserve coherence of homology in general.

Over ordinary schemes $\ush{(-)}$ and $(-)^!$ are canonically identified %with $\shr{(-)}$ 
and so we use both interchangeably.

\subsection{}
\label{subsec:base-ch}
Suppose we have a cartesian square $\mathfrak s$ of noetherian formal schemes
\[
\xymatrix{
\V \ar[d]_{g}\ar[r]^v \ar@{}[dr]|\square& \X \ar[d]^{f} \\
\W \ar[r]_u & \Y
}
\]
with $f$ in $\bbG$ and $u$ flat. The \textit{flat-base-change theorem} 
for $(-)^!$  gives an isomorphism for $\F \in \wDqcp(\Y)$:
\[
\beta^!_{\mathfrak s}(\eF) \colon \R\iGp{\V}v^*f^! \F \iso g^!u^* \eF
\]
(see \cite[p.\,77, Theorem 7.4]{dfs} for the case when $f$ is pseudoproper 
and \cite[p.\,261, Theorem 7.14]{pasting} for the general case and also \cite[\S 7.2]{pasting}).
For $\F \in \wDqcp(\Y)$ the corresponding flat-base-change isomorphism for $\ush{(-)}$ 
%(cf. \cite[p.\,86, Theorem 8.1]{dfs} if $f$ is pseudoproper)
\stepcounter{thm}
\begin{equation*}\label{iso:qc-basech}\tag{\thethm}
\ush{\beta_{\mathfrak s}}(\eF) \colon \BL_{\V}v^*\ush{f}(\eF) \iso \ush{g}u^*(\eF)
\end{equation*}
is induced by the following sequence of natural isomorphisms of functors
\[
\BL_{\V}v^*\ush{f} = \BL_{\V}v^*\BL_{\X}f^! \osi \BL_{\V}v^*f^! \osi \BL_{\V}\R\iGp{\V}v^*f^! \iso \BL_{\V}g^!u^*
=\ush{g}u^*
\]
where the last isomorphism is induced by $\beta^!_{\mathfrak s}$ (cf. \cite[p.\,86, Theorem 8.1]{dfs} for $f,g$ pseudoproper).

If $\F \in \Dc^+(\Y)$, or if $u$ is open or if $\V$ is an ordinary scheme,
we have an isomorphism 
\stepcounter{thm}
\begin{equation*}\label{iso:bc-sharp}\tag{\thethm}
v^*\ush{f}\F\iso \ush{g}u^*\F,
\end{equation*}
(see \cite[Theorem 8.1, Corollary 8.3.3]{dfs}).
%since, on $\Dc(\V)$, $\bLambda_{\V}$ is isomorphic to the identity map.

Further properties of the base-change map are explored in Appendix \ref{ss:more-base-ch}.

\subsection{}\label{ss:trZ}
Let $f\colon X\to Y$ be a separated map of finite-type between ordinary schemes, and $Z$  a closed subscheme
of $X$ which is proper over $Y$.  The completion map $\kappa\colon \X\to X$ of $X$ along $Z$,
is formally \'etale and affine and the composition $\wid{f} \set f\kappa$ is pseudoproper. 
We define the trace map for $f$ along $Z$ 
\stepcounter{thm}
\begin{equation*}\label{map:tr-fZ-intro}\tag{\thethm}
\Tr{f,Z}\colon \R_Zf_*\ush{f} \to {\bf 1}
\end{equation*}
to be the composite
\[\R f_*\R\iG{Z}\ush{f} \iso \Rfs \kappa_*\R\iGp{\X}\kappa^*\ush{f} \iso
 \R{\wid{f}}_*\R\iGp{\X}\ush{\kappa}\ush{f} \iso \R{\wid{f}}_*\R\iGp{\X}\ush{\wid{f}}
\xrightarrow{\Tr{\wid{f}}} {\bf 1} \]
%$\kappa\colon \X\to X$ is the completion of $X$ along $Z$, $\wid{f}=f\kappa$ and 
where the first isomorphism is from the canonical isomorphism 
$\R{\iG{Z}} \iso \kappa_*\R\iGp{\X}\kappa^*$ while the remaining maps are the obvious natural ones.

There is an alternate description of $\Tr{f,Z}$ involving compactifications. Let $(u, \bar{f})$ be
a compactification of $f$, i.e., $u\colon X\to \overline{X}$ is an open immersion, 
$f\colon \overline{X}\to Y$ a {\emph{proper}} map such that $f={\bar{f}}\<\smcirc\< u$. A theorem of 
Nagata assures us that compactifications always exist (see \cite{nagata}, \cite{nagata4},
\cite{nagata2}, and \cite{nagata3}). According to
\Lref{lem:indep1}, $\Tr{f,Z}$ can also be described as the composite 
\stepcounter{thm}
\begin{equation*}\label{map:tr-fZ-alt}\tag{\thethm}
\Rfs\R\iG{Z}\ush{f} \iso \R{\bar{f}}_*\R\iG{u(Z)}\ush{\bar{f}} \to 
\R{\bar{f}}_*\ush{\bar{f}} \xrightarrow{\Tr{\bar{f}}} {\bf 1}.
\end{equation*}

\section{\bf Cohen-Macaulay Maps}\label{s:CM}
\subsection{} Recall that a locally finite type map $f\colon X\to Y$
between ordinary schemes is said to be {\it Cohen-Macaulay of relative dimension $r$} if it is
a flat map and all the non-empty fibres are Cohen-Macaulay and of pure dimension $r$. 
This is equivalent to saying that $f$ is flat, $f^!\co_Y$ (which is defined locally if $f$ is
not separated) has homology concentrated in only degree $-r$, and the resulting $\co_X$-module
$\omgs{f}$ obtained by gluing the various local $\Hr^{-r}(f^!\co_Y)$ is coherent and flat
over $Y$. We make the obvious generalization to formal schemes. First we need the following
definition.

\begin{defi} A map of formal schemes $f\colon \X\to \Y$ is said to be {\em locally in $\bbG$} if
for every point $x\in \X$, there exists an  open neighbourhood $U$ of $x$ %and $V$ of $f(x)$
such that the restriction $f_U$ of $f$ to $U$ is in $\bbG$.
%$f(U)\subset V$ and the map $f_{U,V}\colon U\to V$ obtained by restricting $f$ to $U$ is in $\bbG$.
\end{defi}

Note that if $f\colon\X\to \Y$ is locally in $\bbG$ and $\eF \in \Dc^+(\Y)$, then locally on $\X$, 
$\ush{f}\eF$ is defined, but it need not be defined globally, 
even though by pseudofunctoriality of~$\ush{(-)}$ over~$\Dc^+$, these local twisted
inverse images are isomorphic on overlaps and these isomorphisms form a descent datum for
the Zariski topology. However, in this case, for every integer $n$, and every
$\eG\in \Dc^+(\Y)$ the sheaves $\Hr^n(\ush{f_U}(\eG))$ do glue to give a coherent
sheaf on $\X$ which we denote as $\Hr^n(\ush{f}\eG)$ (even though there might well be no $\ush{f}\eG$). We are not aware of any example where $\ush{f}$ is defined locally but not globally. 
%\marginpar{Where should $\eG$ live?}

\begin{defi}\label{def:cmrd}
A map of formal schemes $f\colon \X \to \Y$ is called {\em Cohen-Macaulay \textup{(CM)} of
relative dimension $r$} if it is flat, locally in $\bbG$ with $\Hr^i(\ush{f}\co_\Y)=0$ for
$i\neq -r$ and $\omgs{f}\set \Hr^{-r}(\ush{f}\co_\Y)$ is flat over~$\Y$. The coherent
$\co_\X$-module $\omgs{f}$ is called the \emph{relative dualizing sheaf} for the CM map
$f$. If such a map~$f$ is already in~$\bbG$, we shall make the identification 
$\ush{f}\co_\Y=\omgs{f}[r]$.
\end{defi} 

It is tempting to give alternate definitions for a map to be CM that are more local in nature. 
An extension of the theory of $(-)^!$ and $\ush{(-)}$ to a larger category containing maps that 
are essentially of pseudo-finite type (see \cite[\S 2.1]{lns}) would provide a natural setting for proving 
equivalence between various possible alternate definitions. Since there is no such extension in literature 
and since we do not need such a result here, we do not pursue this matter further. 

A map $f \colon \X \to \Y$ is said to be \emph{smooth} if it is in~$\bbG$ and is formally smooth.
In this case the universally finite module of relative differential forms~$\widehat{\Omega}^1_{\X/\Y}$
is a locally free module of finite rank and if it has constant rank $r$ 
we say that $f$ has relative dimension~$r$, see \cite[\S 2.6]{lns}. We use $\omega_f$ to denote the 
top exterior power of the universally finite module of differentials.

%LATER We end by pointing out that if $f$ is smooth of relative dimension $r$, then it 
%will follow from \eqref{iso:pre-verdier} below that $f$ is Cohen-Macaulay
%of relative dimension~$r$. For ordinary schemes this is well-known. 

%\subsection{Duality for regular immersions} HAVE TO SAY SOMETHING HERE ABOUT 
%REGULAR IMMERSIONS AND  DUALITY VIA KOSZUL COMPLEXES. ESPECIALLY THE
%MAP $\eta'_i$.

\section{\bf Traces and Residues for Cohen-Macaulay maps}\label{s:CM-tr}

\subsection{Abstract Trace for Cohen-Macaulay maps}\label{ss:tr-cm} 
Suppose $g\colon \X\to \Y$ is Cohen-Macaulay of relative dimension $r$ and
is {\em pseudo-proper}. Since $\omgs{g}[r]=\ush{g}\co_\Y$, therefore, as in \eqref{map:Tr-f},
we have a trace map 
\[\Tr{g}(\co_\Y)\colon \R g_*\R\iGp{\X}\omgs{g}[r] \to \co_\Y.\] 
\begin{defi}
Let $g\colon \X\to \Y$ be as above (i.e., $g$ is pseudo-proper and Cohen-Macaulay of relative
dimension $r$).
The {\em abstract trace map on $\Rp{\X}^rg_*\omgs{g}$} (or simply {\em the trace map on
$\Rp{\X}^rg_*\omgs{g}$}) is the map
\stepcounter{thm}
\begin{equation*}\label{map:tin}\tag{\thethm}
\tin{g}\colon \Rp{\X}^r g_*\omgs{g} \to \co_\Y.
\end{equation*}
given by $\tin{g}=\Hr^0(\Tr{g}(\co_\Y))$.
\end{defi}
%
%For the rest of this section, unless otherwise stated,
%$f\colon X\to Y$ is a separated {\em Cohen-Macaulay map} of relative dimension
%$r$ between {\em ordinary schemes}. 

\begin{thm}\label{thm:local-duality} 
%\marginpar{Proof has to be given, along the lines of email on  March 9 to Suresh}
Let $g\colon \X\to \Y$ be pseudo-proper and Cohen-Macaulay of relative
dimension $r$. Then for any quasi-coherent $\co_\X$-module $\eF$
satisfying $\Rp{\X}^j g_*\eF =0$ for $j> r$, 
%and for every coherent $\co_\X$-module $\eF$. Then for  each  coherent $\co_\X$-module $\eF$ 
we have a functorial isomorphism %(functorial in $\eF\in \Ac(\X)$)
\[\Hom_\X(\eF,\,\omgs{g}) \iso \Hom_\Y(\Rp{\X}^r g_*\eF,\,\co_\Y), \]
which is given by sending $\theta\in \Hom_\X(\eF,\,\omgs{g})$ to the composite
\[\Rp{\X}^r g_*\eF \xrightarrow{\Rp{\X}^r g_*(\theta)} \Rp{\X}^r g_* \omgs{g} 
\xrightarrow{\tin{g}} \co_\Y.\]
\end{thm} 
\begin{proof}
By adjointness we have a natural isomorphism
\[
\R\Hom_\X(\eF,\,\omgs{g}[r]) \iso \R\Hom_\Y(\Rp{\X} g_*\eF,\,\co_\Y),
\]
hence there are natural isomorphisms
\begin{align*}
\Hom_\X(\eF,\,\omgs{g}) &= \Hom_{\D(\X)}(\eF,\,\omgs{g}) \\
&= H^{-r}\R\Hom_\X(\eF,\,\omgs{g}[r]) \\
&\cong H^{-r}\R\Hom_\Y(\Rp{\X} g_*\eF,\,\co_\Y) \\
&= \Hom_{\D(\Y)}(\Rp{\X} g_*\eF[r],\,\co_\Y) \\
&\cong \Hom_{\D(\Y)}(\Rp{\X}^r g_*\eF,\,\co_\Y) \qquad (\text{see \cite[p.\,37, Prop.\,1.10.1]{notes}})\\
& = \Hom_\Y(\Rp{\X}^r g_*\eF,\,\co_\Y).
\end{align*}
\end{proof}

In the above proof, we don't know if $\omgs{g}$ satisfies the hypotheses required of $\eF$. We 
are interested in special cases when this is true. This leads to the following.

\begin{cor}\label{cor:local-duality} If  $\Rp{\X}^jg_*(\eF)=0$ for every $j>r$ and every $\eF\in\Ac(\X)$ (resp.~$\eF\in\Avc(\X)$, resp.~$\eF\in\Aqc(\X)$), then $(\omgs{g},\,\tin{g})$ represents the 
functor $\Hom_\Y({\Rp{\X}^rg_*(\boldsymbol{-}}),\, \co_\Y)$
on $\Ac(\X)$ (resp.~$\Avc(\X)$, resp.~$\Aqc(\X)$). 
In particular if $\Y=Y$ is an ordinary scheme, $f\colon X\to Y$ 
a Cohen-Macaulay map of ordinary schemes of relative dimension $r$, 
$Z$ a closed subscheme of $X$ such that
the resulting map $Z\to Y$ is finite and flat, $\X=X_{/Z}$ the completion of~$X$ along~$Z$,
and $g\colon \X\to \Y$ the map induced by $f$, then  $(\omgs{g},\,\tin{g})$ represents the functor 
$\Hom_\Y({\Rp{\X}^rg_*(\boldsymbol{-}}),\, \co_\Y)$ on $\Avc(\X)$.
\end{cor}
\proof 
The first assertion holds since, if $g$ is Cohen-Macaulay, then $\omgs{g} \in \Ac(\X)$.
The second assertion follows from the first since, if $\I$ is a coherent ideal defining $Z$ in $X$,
$\eF\in\Avc(\X)$ and $\kappa$ denotes the canonical map $\X \to X$, 
then $\eF \iso \kappa^*\eG$ for some $\eG \in \Aqc(X)$ (see \cite[p.\,31, Prop.\,3.1.1]{dfs}), and hence
$\R g_*\R\iGp{\X}\eF \iso \R f_*\kappa_*\R\iGp{\X}\kappa^*\eG \iso \R f_*\R\iG{\I}\eG$
(see \cite[\S5]{dfs}).
%The second half of the statement clearly follows from the first half, for which
%we only need to observe that since $g$ is Cohen-Macaulay, $\omgs{g} \in \Avc(\X)$. 
\qed

\medskip

\begin{rem}\label{rem:joe}
In a slightly different direction, Lipman observed the following (private communication). First note that
according to \cite[p.\,39, Prop.\,3.4.3]{dfs}, since all our schemes are noetherian, if $f\colon\X\to\Y$
is a map of schemes (possibly formal), the functor $\Rfs$ is bounded above on $\Dvc(\X)$. In
other words, there is an integer $e\ge 0$ such that if $\eH\in\Dvc(\X)$ and $H^i(\eH)=0$
for $i\ge i_0$, then $H^i(\Rfs\eH)=0$ for all $i\ge i_0+e$. Next, by computing local cohomologies
using stable Koszul complexes (see \eqref{iso:k-infty-gam}) on affine open subschemes of $\X$ and
using quasi-compactness of the noetherian scheme $\X$, we see that there is an integer
$t$ such that $H^j(\R\iGp{\X}\eF)=0$ for $\eF\in\Avc(\X)$ and $j>t$. 
It is then not hard to see that if $r=e+t$,
and if  $\eH\in\Dvc(\X)$ is such that $H^i(\eH)=0$ for $i > i_0$, then $H^j(\Rfs\R\iGp{\X}(\eH))=0$
for $j>i_0+r$. Now suppose $f$ is {\emph{pseudo-proper}}.
By the argument given in
\cite[p.\,165, Lemma\,4.1.8]{notes}, we see that if $\eG\in\Avc(\Y)$ is such that
$\ush{f}\eG\in\Dvc(\X)$ then $H^j\ush{f}\eG=0$ for every any $j< -r$. 
Let $\omgs{f}=H^{-r}(\ush{f}\co_\Y)$. Then as we argued earlier,
$\omgs{f}\in\Ac(\X) \subset \Dvc(\X)$. The proof of \Tref{thm:local-duality}
applies and we have a functorial isomorphism (without any Cohen-Macaulay hypotheses)
\[\Hom_\X(\eF,\,\omgs{f})\iso \Hom_\Y(\Rp{\X}^r\eF,\,\co_\Y)\]
for every $\eF\in\Avc(\X)$. We point out that $\Rp{\X}^j\vert_{\Avc(\X)} =0$ for $j>r$. In
particular, in this argument, $\Rp{\X}^j\omgs{f}=0$ for $j>r$. We could not guarantee this
for the $r$ used in the Theorem.
\end{rem}

\subsection{Abstract Residue for Cohen-Macaulay maps} Throughout this subsection
\[f\colon X\to Y\]
is a {\em finite-type Cohen-Macaulay} map between ordinary schemes of relative dimension $r$.
Suppose $Z\hookrightarrow X$ is a closed subscheme of $X$, {\em proper over $Y$}. 
Let $\X=X_{/Z}$ be the formal completion of $X$ along $Z$, and  $\kappa\colon\X\to X$ the 
completion map. Let $\wid{f}\colon \X\to Y$ be the composite $f\smcirc\kappa$. We have
$\ush{\wid{f}}\iso \ush{\kappa}\ush{f} \iso \kappa^*\ush{f}$, whence $\Hr^j(\ush{\wid{f}})\iso
\kappa^*\Hr^j(\ush{f})$.
Note that $\wid{f}$ is 
pseudo-proper and Cohen-Macaulay of relative dimension $r$ and therefore ${\tin{\wid{f}}}$ is defined.
In \cite[p.\,742, (3.2)]{cm} a residue map along $Z$ is defined using local compactifications.
Here is a reformulation of that definition in terms of $\kappa$.
\begin{defi}\label{def:a-res} Let $Z$ and $f$ be as above. The {\em abstract residue along $Z$} 
\stepcounter{thm}
\begin{equation*}\label{def:ares}\tag{\thethm}
\ares{Z}\colon \Rr^r_Zf_*\omgs{f} \to \co_Y
\end{equation*}
is the composite 
\[\Rr^r_Zf_*\omgs{f} \xrightarrow[\eqref{iso:iGZ-iGpX}]{\Iso} \Rp\X^r{\wid{f}}_*\omgs{\wid{f}} 
\xrightarrow{\tin{\wid{f}}} \co_Y.\]
\end{defi}
It is worth unravelling the first isomorphism in the above composite a little more.
The isomorphism $\kappa_*\R\iGp{\X}\kappa^*\iso \R\iG{Z}$ gives rise to 
isomorphisms (one for every $j$)
\stepcounter{thm}
\begin{equation*}\label{iso:loc-coh}\tag{\thethm}
\Rr^j_Z f_*\eF \iso \Hr^j({\R\wid{f}}_*\R\iGp{\X}\kappa^*\eF)=\Rp{\X}^j{\wid{f}}_*\kappa^*\eF
\end{equation*}
which are functorial in $\eF$ varying over quasi-coherent $\co_X$-modules. 
In affine terms, if $X=\Spec{\,R}$, $M$ an $R$-module, and $Z$ is given
by the ideal $I$, then writing $\wid{R}$ for the $I$-adic completion of $R$, and $J=I\wid{R}$,
the above isomorphism is the well-known one
\[\Hr^j_I(M) \iso\Hr^j_J(M\otimes_R\wid{R}).\]
%\marginpar{Should have dropped coherent and finitely generated. Too sleepy to make the change.}
The isomorphism 
$\Rr^r_Zf_*\omgs{f} \iso \Rp\X^r{\wid{f}}_*\omgs{\wid{f}}$
induced by \eqref{iso:iGZ-iGpX}
is the composite of the map \eqref{iso:loc-coh}, i.e.,
 $\Rr^r_Zf_*\omgs{f} \iso \Rp\X^r{\wid{f}}_*\kappa^*\omgs{f}$, and the
 isomorphism induced by \eqref{eq:gm}, i.e., 
$\Rp\X^r{\wid{f}}_*\kappa^*\omgs{f}\iso \Rp\X^r{\wid{f}}_*\ush{\kappa}\omgs{f}\iso \Rp\X^r{\wid{f}}_*\omgs{\wid{f}}$ .

In \cite[p.742, (3.2)]{cm} a different, but equivalent, definition is given of 
$\ares{Z}$. In that situation $f$
is {\em separated,} and therefore has a compactification by a result of Nagata, say 
 $u\colon X\hookrightarrow \bar{X}$ of $X$ over $Y$. Let $\bar{f}\colon \bar{X}\to Y$ be the 
 structure morphism (by definition of a compactification, a proper 
map) of $\bar{X}$.  In {\em loc.\:cit.}, the residue along $Z$
is defined as the composite
\begin{align*}
\Rr^r_Z f_*\omgs{f} = \Hr^0(\Rfs\R\iG{Z}\omgs{f}[r])=\Hr^0(\Rfs\R\iG{Z}f^!\co_Y)
&\xrightarrow[\phantom{\Hr^0(\Tr{f})}]{\Iso} \Hr^0(\R{\bar f}_*\R\iG{u(Z)}\bar{f}^!\co_Y)\\
&\xrightarrow{\phantom{\Hr^0(\Tr{f})}} \Hr^0(\R{\bar f}_*\bar{f}^!\co_Y)\\
& \xrightarrow{\Hr^0(\Tr{\bar{f}})} \co_Y.
\end{align*}
By \Lref{lem:indep1} the two definitions coincide in the situation considered in \cite{cm}
and therefore the definition in \cite[p.742, (3.2)]{cm} is independent of the compactification 
$(u,\,\bar{f})$. This gives another proof of \cite[p.742, Proposition 3.1.1]{cm}. 

If $f$ is {\em proper}, it follows that there is a commutative diagram:
\stepcounter{thm}
%\begin{equation*}\label{diag:res-int1}\tag{\thethm}
\[
\begin{aligned}\label{diag:res-int1}
\xymatrix{
\Rr^r_Z f_*\omgs{f} \ar[dr]_{\ares{Z}} \ar[r] & \Rr^rf_*\omgs{f} \ar[d]^{\tin{f}} \\
& \co_Y 
}
\end{aligned}\tag{\thethm}
\]
%\end{equation*}

\begin{rem}\label{rem:ILN} In \cite[p.\,746, Remark 2.3.4]{iln}, Iyengar, Lipman, and Neeman
give a generalization of the residue
map in \cite{cm}. Suppose $f\colon X\to Y$ is a separated essentially finite type 
map of ordinary schemes, $W$ {\emph{a union of closed subsets}} of $X$ to each of which the 
restriction of $f$ is proper. (Note that $W$ need not be closed in $X$.)
Then one has an integer $d$ such that $\Hr^{-e}(f^!\co_Y)=0$ for all $e>d$,
while $\omega_f\set \Hr^{-d}(f^!\co_Y)\neq 0$. Iyengar, Lipman, and Neeman then define
a natural map
\stepcounter{thm}
\begin{equation*}\label{map:ILN-1}\tag{\thethm}
\Hr^d\Rfs\R\iG{W}(\omega_f) \lra \co_Y
\end{equation*}
denote by them as $\int_W$, which generalizes the map denoted $\rm{res}_{W}$ in \cite[\S 3.1]{cm}.
In greater detail, if $\Dqc(X)_W$ denotes the essential image of $\R\iG{W}$ in $\Dqc(X)$,
then in \cite[p.\,746, Corollary 2.3.3]{iln} it is shown that for $E$ in $\Dqc(X)_W$ and
$G$ in $\Dqc^+(Y)$, we have a functorial isomorphism
\[
\Hom_{\D(Y)}(\Rfs E,\,G) \iso \Hom_{\D(X)}(E,\,\R\iG{W}f^!G).
\]
In particular one has a counit
\stepcounter{thm}
\begin{equation*}\label{map:ILN-2}\tag{\thethm}
\Rfs\R\iG{W}f^!\co_Y \lra \co_Y.
\end{equation*}
The map \eqref{map:ILN-1} is defined as the composite
\begin{align*}
\Hr^d\Rfs \R\iG{W}(\omega_f) & = \Hr^0\Rfs \R\iG{W}(\omega_f[d])\\
& \lra \Hr^0\Rfs \R\iG{W}(f^!\co_Y) \xrightarrow{\Hr^0\eqref{map:ILN-2}} \Hr^0\co_Y=\co_Y.
\end{align*}
\end{rem}

\subsection{Traces for finite Cohen-Macaulay maps.}
We begin with a global construction. Suppose we have a commutative diagram of ordinary
schemes
\stepcounter{thm}
\[
\begin{aligned}\label{diag:i-f-h}
\xymatrix{
Z\, \ar@{^(->}[r]^i \ar[dr]_h & X \ar[d]^f \\
& Y
}
\end{aligned}\tag{\thethm}
\]
with $f$ Cohen-Macaulay of relative dimension $r$, $h$ a finite surjective map, $i$ a
closed immersion, the $\co_X$-ideal $\I$ of $Z$ generated by 
$t_1,\dots, t_r\in\Gamma(X,\co_X)$ such that ${\bf t}=(t_1,\dots,\,t_r)$ is $\co_{X,z}$-regular
for every $z\in Z\subset X$. Note that $h$ is necessarily flat and is Cohen-Macaulay of 
relative dimension $0$. Define
\stepcounter{thm}
 \begin{equation*}\label{map:finflat-h}\tag{\thethm}
 \ttr{h}(=\ttr{h,f,i})\colon h_*(i^*\omgs{f}\otimes_{\co_Z}\wI{\co_Z}{\I}) \xrightarrow{\phantom{XXX}} \co_Y
 \end{equation*}
as the unique map which fills the dotted arrow to make the diagram below commute 
%(where the ``equality" on the top right corner stems from the fact that $h_*$ is exact):
where $\eta'_i$ is induced by \eqref{iso:eta'-i}. 
\[ \xymatrix{
 h_*(i^*\omgs{f}\otimes_{\co_Z}\wI{\co_Z}{\I})\ar[r]^-{\Iso}_-{\eta'_i}
 \ar@{.>}[dd]_{\ttr{h}} & h_*\Hr^0(i^!f^!\co_Y) \ar[r]^-{\Iso}  & h_*\Hr^0(h^!\co_Y)
 \ar@{=}[dd] \\
 &&\\
 \co_Y &  & \Hr^0(h_*h^!\co_Y) \ar[ll]^{\tin{h}}
} \]
\begin{comment} %%%%%%%%%%
\[ \xymatrix{
 h_*(i^*\omgs{f}\otimes_{\co_Z}\wI{\co_Z}{\I})\ar[r]^-{\Iso}_-{\eta'_i}
 \ar@{.>}[dd]_{\ttr{h}} & h_*\Hr^0(i^!f^!\co_Y) \ar@{=}[r]  & \Hr^0h_*(i^!f^!\co_Y)
 \ar[dd]^{\,\rotatebox{-90}{\makebox[-0.1cm]{\Iso}}}\\
 &&\\
 \co_Y &  & \Hr^0(h_*h^!\co_Y) \ar[ll]^{\tin{h}}
}
\]
\end{comment} %%%%%%%%%%
%Here $\eta'_i$ is induced by \eqref{iso:eta'-i}. 

We would like to show that $\ttr{h}$ factors through $\ares{Z}\colon \Rr^r_Zf_*\omgs{f}\to \co_Y$.
To that end, we make the following definition. First, as in \eqref{def:btrg}, let 
$i^\btrg \set \bL i^*(\boldsymbol{-})\overset{\bL}{\otimes}_{\co_Z} (\wnor{i}[-r])$. Next,
for a quasi-coherent $\co_X$-module $\eF$, let 
\stepcounter{thm}
\begin{equation*}\label{map:sh-foo}\tag{\thethm}
\psi=\psi(\eF) \colon h_*(i^*\eF\otimes_{\co_Z}\wI{\co_Z}{\I}) \xrightarrow{\phantom{XXXX}} 
\Rr^r_Zf_*\eF
\end{equation*}
be defined by applying $\Hr^0$ to the composite
\stepcounter{thm}
\begin{equation*}\label{map:sh-foo2}\tag{\thethm}
h_*i^\btrg\eF[r] \iso \Rfs i_*i^\btrg\eF[r] \xrightarrow[\eqref{thm:eta-i}]{\Iso} \Rfs i_*i^\flat\eF[r] \lra
\Rfs \R\iG{Z}\eF[r].
\end{equation*}
The map \eqref{map:sh-foo} is a sheafied version of \eqref{foo}, 
as \eqref{diag:foo-M} shows. Moreover it is functorial in $\eF$.

The formal-scheme version is as follows. Let $\kappa\colon \X\to X$ be the completion of $X$
along $Z$, $j\colon Z\to\X$ the natural closed immersion (so that $\kappa\smcirc j=i$) and
$\wid{f}=f\smcirc\kappa$. As before, let $j^\btrg$ be as in \eqref{def:btrg}.
For $\eG\in \Avc(\X)$ we have a map
\stepcounter{thm}\stepcounter{subsubsection}
\begin{equation*}\label{map:formal-foo}\tag{\thethm}
\wid{\psi}\colon h_*(j^*\eG\otimes_{\co_Z}\wnor{j})
\xrightarrow{\phantom{XXX}}  \Rp\X^r{\wid{f}}_*\eG
\end{equation*}
defined by applying $\Hr^0$ to the composite
\stepcounter{thm}\stepcounter{subsubsection}
\begin{equation*}\label{map:sh-foo3}\tag{\thethm}
h_*j^\btrg\eG[r] \iso \R{\wid{f}}_*j_*j^\btrg\eG[r] 
\iso\R{\wid{f}}_*j_*\ush{j}\eG[r]\xrightarrow{\Tr{j}}\R{\wid{f}}_*\R\iGp{\X}\eG[r].
\end{equation*}
Since the composite $i_*i^\flat \to \R\iG{Z} \to {\bf 1}$ is ``evaluation at one", i.e., it is the
trace map (if one identifies $i^\flat$ with $i^!$), 
it is easy to see that the diagram
\stepcounter{thm}
\[
\begin{aligned}\label{diag:psi-psi}
{\xymatrix{
h_*i^\btrg\eF[r] \ar[rrr]^-{\eqref{map:sh-foo2}} \ar@{=}[d]
&&& \Rfs\R\iG{Z}\eF[r] \\
h_*j^\btrg\kappa^*\eF[r] \ar[rr]^-{\eqref{map:sh-foo3}}
&& \R{\wid{f}}_*\R\iGp{\X}\kappa^*\eF[r]
\ar[r]^-\Iso & \Rfs \kappa_*\R\iGp{\X}\kappa^*\eF[r]
\ar[u]_{\>\rotatebox{90}{\makebox[-0.1cm]{\Iso}}}
}}
\end{aligned}\tag{\thethm}
\]
commutes where the upward arrow on the right is induced by the isomorphism
$\kappa_*\R\iGp{\X} \kappa^*\iso \R\iG{Z}$. 

\begin{thm}\label{thm:sh-foo} In the situation of \eqref{diag:i-f-h},
the following diagram commutes.
\[
{\xymatrix{
 h_*(i^*\omgs{f}\otimes_{\co_Z}\wI{\co_Z}{\I}) \ar[d]_{\psi(\omgs{f})} \ar[rr]^-{\ttr{h}}
 & & \co_Y\ar@{=}[d]\\
\Rr^r_Zf_*\omgs{f} \ar[rr]^{\ares{Z}} & & \co_Y 
}}
\]
\end{thm}

\proof 
The diagram in the statement of the theorem
can be realized as the transpose of the border of the following one.
\[
\xymatrix{
 h_*(i^*\omgs{f}\otimes\eN^r_i) \ar[dd] \ar[rd] \ar[rrrr]^{\psi} 
&&&& \Rr^r_Zf_*\omgs{f} \ar[ld] \ar[dd]^{\ares{Z}}\\
& \hspace{-3em} h_*(j^*\omgs{\wid{f}}\otimes \eN^r_j) \ar[ld] \ar[rr]^{\wid{\psi}}  && 
\Rp\X^r{\wid{f}}_*\omgs{\wid{f}} \ar[rd]^{\tin{\wid{f}}} & \\
h_*\omgs{h} \ar[rrrr]_{\tin{h}}  &&&& \co_Y}
\]
We have to show the above diagram commutes. Applying $\Hr^0$ to \eqref{diag:psi-psi},
with $\eF=\omgs{f}$,
and using the isomorphism $\kappa^*\omgs{f}\iso \omgs{\wid{f}}$\,, we see
that the upper trapezium commutes. The triangle on the right commutes by definition of $\ares{Z}$
(see \Dref{def:a-res}), while the one on the 
left corresponds to the natural isomorphisms 
$\ush{i}\ush{f} \iso \ush{j}\ush{\wid{f}} \iso \ush{h}$
(after applying $\Hr^0$ and $h_*$). Finally, the lower trapezium corresponds to $\Hr^0$ of the 
outer border of the following diagram. 
\[
\xymatrix{
h_*j^\btrg\ush{\wid{f}}\co_Y \ar[rr]^{\Iso} \ar[d]^{\eqref{iso:eta'-i}}_{\,\rotatebox{-90}{\makebox[-0.1cm]{\Iso}}}
 && 
\wid{f}_*j_*j^\btrg\ush{\wid{f}}\co_Y \ar[d]^{\eqref{iso:eta'-i}}_{\,\rotatebox{-90}{\makebox[-0.1cm]{\Iso}}}
  \\
h_*\ush{j}\ush{\wid{f}}\co_Y \ar[rr]^{\Iso} \ar[d]_{\,\rotatebox{-90}{\makebox[-0.1cm]{\Iso}}}
 && \wid{f}_*j_*\ush{j}\ush{\wid{f}}\co_Y  \ar[r]_{\Tr{j}} & \wid{f}_*\R\iGp{\X}\ush{\wid{f}}\co_Y \ar[d]^{\Tr{\wid{f}}} \\
h_*\ush{h}\co_Y \ar[rrr]_{\Tr{h}}  &&& \co_Y
}
\]
The upper rectangle commutes trivially while the lower one results from the
identification of the adjoint $\ush{h}$ with the composition of the adjoints $\ush{j}\ush{\wid{f}}$. 
%This is an immediate consequence of \Lref{lem:gen-frac} and \Lref{lem:eta-kappa}. 
\qed

\begin{rem}\label{map:psi-local} The map $\psi$ in \eqref{map:sh-foo} is compatible with
open immersions in $X$ containing $Z$. In greater detail, suppose $i$ factors as
\[Z\xrightarrow{u} U \xrightarrow{x} X\]
with $x\colon U\to X$ an open immersion, and $u$ (necessarily) a closed immersion. 
Then
\[
{\xymatrix{
h_*i^*\eF \otimes \wnor{i} \ar[rr]^{\eqref{map:sh-foo}} 
\ar@{=}[d] && \Rr^r_Zf_*\eF  
\ar[d]^{\rotatebox{90}{\makebox[0.1cm]{\Iso}}\!} \\
h_*u^*(x^*\eF)\otimes \wnor{u} \ar[rr]_{\eqref{map:sh-foo}} && \Rr^r_Z (fx)_*x^*\eF
}}
\]
 commutes. We leave the verification to the reader, but point out that one method is
 to move to formal schemes, using \eqref{diag:psi-psi}, noting that the completion of
 $X$ along $Z$ is the same as the completion of $U$ along $Z$. This means $\ttr{h}$
 is unaffected if $X$ is replaced by $U$.
\end{rem}

\subsection{A residue formula for Cohen-Macaulay maps}
Consider again Diagram \eqref{diag:i-f-h}. Suppose now that $X$, $Y$, and $Z$ are affine,
say $X=\Spec{\,R}$, $Y=\Spec{\,A}$ and $Z=\Spec{\,B}$. In other
words $A\to R$ is a finite-type map of rings which is Cohen-Macaulay of relative
dimension $r$, we have an ideal $I$ in $R$ generated by a quasi-regular sequence ${\bf t}=(t_1,\dots, t_r)$ in $R$,  and $B=R/I$.  Assume
as before that $h$ is a finite (and hence flat) surjective map. 

Let us write $\omgs{R/A}=\Gamma(X,\,\omgs{f})$, $\omgs{B/A}=\Gamma(Z,\,\omgs{h})$, 
$\tin{B/A}=\Gamma(Y,\,\tin{h})$. The global sections of $\ttr{h}$ give us an $A$-linear map
\stepcounter{thm}
 \begin{equation*}\label{map:finflat-B}\tag{\thethm}
\ttr{B/A}(=\ttr{B/A,R})\colon \omgs{R/A}\otimes_R\wI{B}{I} \lra A
\end{equation*}
such that the following diagram commutes
\[
\xymatrix{
\omgs{R/A}\otimes_R \wI{B}{I} \ar[d]_{\ttr{B/A}} \ar[rr]^-{\Iso} && \omgs{B/A} \ar[d]^{\tin{B/A}} \\
A \ar@{=}[rr] && A}
\]
where the horizontal isomorphism on the top row is the global sections of the composite
\[h_*(i^*\omgs{f}\otimes_{\co_Z}\wI{\co_Z}{\I})\xrightarrow[{\eta'_i}]{\Iso}%{\phantom{X}\Iso\phantom{X}}
h_*\Hr^0(i^!f^!\co_Y) = \Hr^0h_*(i^!f^!\co_Y) \xrightarrow{\Iso}%{\phantom{X}\Iso\phantom{X}} 
\Hr^0(h_*h^!\co_Y).\]

\setcounter{subsubsection}{\value{thm}}
\subsubsection{}\stepcounter{thm}{\bf Notation.}
In an obvious extension of our notational philosophy, we should use
the symbol $\ares{I}$
for the global sections of the residue map 
$\ares{Z}$ in \eqref{def:ares}. However, for psychological reasons we will continue
to use the symbol $\ares{Z}$ to denote this map. Thus we have
\[\ares{Z} \colon \Hr^r_I(\omgs{R/A})\to A.\]
In what follows, elements of $\Hr^r_I(\omgs{R/A})$ are denoted by generalized fractions
\[
\begin{bmatrix}\nu\\
t_1,\,\dots,\,t_r
\end{bmatrix}
\]
as in \Ssref{ss:stabkoz} (see especially \eqref{iso:k-infty-gam} and \eqref{eq:2gen-fracs}
and the discussions around them). 

Finally, define
\stepcounter{thm}
\begin{equation*}\label{def:1/t}\tag{\thethm}
\frac{\bf 1}{\bf t}\in\wI{B}{I} 
\end{equation*}
as the element which sends $(t_1+I^2)\wedge\dots\wedge (t_r+I^2)\in \wedge_B^r{I/I^2}$ 
to $1$.

\begin{prop}\label{prop:res-form1} With the above notations, 
for any $\nu \in \omgs{R/A}$ we have
\[\ares{Z}\begin{bmatrix} \nu\\
t_1,\,\dots,\,t_r
\end{bmatrix} = \ttr{B/A}\Bigl(\nu\otimes \frac{\bf 1}{\bf t}\Bigr)\]
where $\frac{\bf 1}{\bf t}$ is as in \eqref{def:1/t}.
\end{prop}
\proof

According to \Tref{thm:sh-foo}, the following diagram commutes.
\[
\xymatrix{
\omgs{R/A}\otimes_R \wI{B}{I} \ar[rd]_{\ttr{B/A}} \ar[rr]^{\textup{(\ref{foo})}} && \Hr^r_I(\omgs{R/A}) \ar[ld]^{\ares{Z}} \\
& A}
\] 
The Proposition then follows from  \Lref{lem:gen-frac}. 
\qed

\begin{thm}\label{thm:denom} Suppose $J$ is another ideal in $R$ such that
 $I\subset J$ and $J$ is generated by a quasi-regular sequence ${\bf g}=(g_1,\dots,g_r)$. Let
 $t_i=\sum_ju_{ij}g_j$, $u_{ij}\in R$. Let $W=\Spec{\,R/J}$. Then, for any $\nu\in\omgs{R/A}$
 \[
 \ares{Z}\begin{bmatrix}\det(u_{ij})\nu\\
t_1,\,\dots,\,t_r
\end{bmatrix} =
\ares{W}\begin{bmatrix}\nu\\
g_1,\,\dots,\,g_r
\end{bmatrix}
 \]
\end{thm} 
\proof
This is an immediate consequence of \Tref{thm:dir-im} and \Pref{prop:res-form1}.
\qed

\section{\bf Base change for residues}\label{s:res-bc}

\subsection{Hypotheses}\label{ss:hyp} Throughout this section (i.e., \S\,\ref{s:res-bc}), we fix
a commutative diagram of ordinary schemes
\stepcounter{thm} \stepcounter{subsubsection}
\[
\begin{aligned}\label{diag:basic}
\xymatrix{
Z' \ar[rr]^w \ar[d]_{j} \ar@/_3.5pc/[dd]_{h'} \ar@{}[drr]|\square
& & Z \ar[d]^{i}
\ar@/^3.5pc/[dd]^{h}\\
X' \ar[rr]^v \ar[d]_g \ar@{}[drr]|\square && X
\ar[d]^{f} \\
Y' \ar[rr]_u && Y
}
\end{aligned}\tag{\thethm}
\]
with $f$ separated 
Cohen-Macaulay of relative dimension $r$, the rectangles cartesian, $i\colon Z\to X$
a closed immersion such that  $h=f\smcirc i\colon Z\to Y$ is finite and
the quasi-coherent ideal sheaf $\I$ of $Z$ is generated by global sections
$t_1,\dots, t_r \in \Gamma(X,\,\co_X)$ with the property that ${\bf t}=(t_1,\,\dots,\,t_r)$ is 
$\co_{X,z}$-regular for every $z\in Z\subset X$. (Note that
$Z\to Y$ is flat  by \cite[$0_{\rm{IV}}$, 15.1.16]{ega}.) We also use the following additional
notations: $\J=v^*\I$  is the ideal sheaf of $Z'$,  $\eN=\wI{\co_Z}{\I}$, and
$\eN'=\wI{\co_{Z'}}{\J}=w^*\eN$.

\subsection{Base change for direct image with supports} 
Since $f$ is {\emph{Cohen-Macaulay of relative dimension $r$}}, therefore, according to
\cite[p.\,740,\,Theorem\,2.3.5\,(a)]{cm}, we have a base-change isomorphism
\[\theta_u^f\colon v^*\omgs{f}\iso \omgs{g}.\]
The principal aim of this section is to show that the composite
\[u^*h_*(i^*\omgs{f}\otimes_{\co_Z}\eN) \iso h'_*(j^*v^*\omgs{f}\otimes_{\co_{Z'}}\eN')
\xrightarrow[\theta_u^f]{\>\Iso\>}h'_*(j^*\omgs{g}\otimes_{\co_{Z'}}\eN') \xrightarrow{\>\ttr{h'}\>} \co_{Y'}
\]
is $u^*\ttr{h}$, i.e., speaking informally, $\ttr{h}$ is stable under base change
(here the first isomorphism results from the fact that $h$ is an affine map). 
We would also like to show that the result in \Tref{thm:sh-foo} is stable under
base change. Indeed, that is how we will prove that $\ttr{h}$ is stable under 
base change. To set things up, we now discuss, very briefly,
base change for cohomology with
supports, at least for the situation we are in.

In our situation, we have  base-change maps (see, for example, \cite[p.\,768, (A.5)]{cm}),
one for each $k$
\stepcounter{thm}
\begin{equation*}\label{map:res-bc}\tag{\thethm}
b(u,f)=b(u,f,k)\colon u^*\Rr^k_Zf_* \lra \Rr^k_{Z'}g_*v^*.
\end{equation*}
These are natural transformation of functors on quasi-coherent sheaves on $X$. In the event
$u$ is {\emph{flat}}, $b(u, f)$ is an isomorphism. In fact, in this case, $b(u, f, k)$ is
  $H^k(\boldsymbol{-})$ applied to the composite of
natural isomorphisms
\[u^*\Rfs\smcirc\R\iG{Z} \iso \R g_*v^*\smcirc \R\iG{Z} \iso \R g_*\smcirc\R\iG{Z'}v^*.\]

It is useful for us to recast $b(u,f)$ in terms of the formal completions of $X$ and $X'$.
To that end, let $\kappa\colon \X\to X$ (resp.~$\kappa'\colon \X'\to X'$) be the completion of $X$ along
$Z$ (resp.~ of $X'$ along $Z'$) and let $\alpha \colon Z\to \X$, $\beta\colon Z'\to \X'$ be the
natural closed immersions, so that $i=\kappa\smcirc \alpha$ and $j=\kappa'\smcirc\beta$.
Let $\wid{f}=f\smcirc\kappa$, $\wid{g}=g\smcirc\kappa'$, and finally
 let $\wid{v}\colon \X'\to \X$ be the natural map induced by $v$, so that the following
 diagram is cartesian:
 \stepcounter{thm}
\[
\begin{aligned}\label{map:tr-bc}
{\xymatrix{
\X' \ar[r]^{\wid{v}} \ar[d]_{\wid{g}} \ar@{}[dr]|{\square} & \X \ar[d]^{\wid{f}} \\
Y' \ar[r]_u & Y
}}
\end{aligned}\tag{\thethm}
\]
The maps $b(u, f, k)$ in \eqref{map:res-bc} give rise, in a natural way, maps 
\stepcounter{thm}
\begin{equation*}\label{map:res-bc2}\tag{\thethm}
b(u, \wid{f})=b(u, \wid{f}, k)\colon u^*\Rp{\X}^k\wid{f}_* \lra \Rp{\X'}^k\wid{g}_*\>\wid{v}\>\>^*
\end{equation*}
induced by \eqref{iso:iGZ-iGpX} applied to $\kappa$ and to $\kappa'$. 
In the event $u$ is flat, then as in the case of ordinary
schemes, $b(u, \wid{f}, k)$ is an isomorphism and  is, in fact, 
$H^k(\boldsymbol{-})$ applied to the natural composite
\[ u^*\R\wid{f}\smcirc\R\iGp{\X} \iso \R\wid{g}\>\>\>\wid{v}\>\>^*\smcirc\R\iGp{\X} 
\iso \R\wid{g}\smcirc\R\iGp{\X'}\wid{v}\>\>^*.\]

We will in fact show that when $k=r$, the map $b(u, f, k)$
is an isomorphism \emph{even when $u$ is not flat}.

\begin{prop}\label{prop:i-bc} Suppose $u$ is a flat map and $\eF\in\Dqc(X)$. Then the following diagram commutes,
where the unlabelled arrows arise from the natural maps (``evaluation at $1$") $i_*i^\flat \to \R\iG{Z}$
and $j_*j^\flat\to \R\iG{Z'}$:
\[
{\xymatrix{
j_*j^\btrg v^*\eF\ar[rrr]^\Iso_{\eqref{thm:eta-i}} \ar@{=}[d]
&& & j_*j^\flat v^*\eF \ar[r]
& \R\iG{Z'}v^*\eF \\
j_*w^*i^\btrg\eF \ar@{=}[r] 
&v^*i_*i^\btrg\eF \ar[rr]^\Iso_{\eqref{thm:eta-i}}
& & v^*i_*i^\flat\eF \ar[r]
& v^*\R \iG{Z}\eF \ar[u]_{\>\rotatebox{90}{\makebox[-0.1cm]{\Iso}}}
}}
\]
\end{prop}
\noindent{\emph{Remark:}} The maps \eqref{thm:eta-i} make sense for $\eF\in\Dqc(X)$ because
$X$ is an ordinary scheme (see discussion in \Ssref{sss:Tr-concrete}).
Flat base change works in this case for all $\eF\in\Dqc(X)$ without
boundedness hypotheses because $i_*$ takes perfect complexes to perfect complexes (see
\cite[p.\,197, Thm.\,4.7.4]{notes}).
\proof Let $\Tr{i}^\flat\colon i_*i^\flat \to {\bf 1}$  be as in \eqref{def:Tr-i-flat},
i.e., $\Tr{i}^\flat$ is the composite
$ i_*i^\flat \to \R\iG{Z} \to {\bf 1}$ (see the
discussion in \Ssref{subsec:closed-imm} with $\Z=Z$ and $\X=X$).
Then $\Tr{i}^\flat$ is simply evaluation at $1$, and hence equals
the composite 
\[i_*i^\flat \iso i_*i^! \xrightarrow{\Tr{i}} {\bf 1}.\]
The two maps, $i_*i^\flat \to \R\iG{Z}$ and $\Tr{i}^\flat$, determine each other and hence we have show that the diagram
\[
\begin{aligned}
{\xymatrix{
j_*j^\btrg v^* \ar[rr]^\Iso_{\eqref{thm:eta-i}} \ar@{=}[d] && j_*j^\flat v^* \ar[rr]^{\Tr{j}^\flat} 
&& v^* \ar@{=}[d]\\
v^*i_*i^\btrg \ar[rr]^\Iso_{\eqref{thm:eta-i}} && v^*i_*i^\flat \ar[rr]^{v^*(\Tr{i}^\flat)} && v^*
}}
\end{aligned}\tag{\dag}
\]
commutes. 

The composite $i_*i^\btrg \xrightarrow{\eqref{thm:eta-i}} i_*i^\flat \xrightarrow{\Tr{i}^\flat} {\bf 1}$
is clearly the same as the
composite $i_*i^\btrg \xrightarrow{\eqref{iso:eta'-i}} i_*i^! \xrightarrow{\Tr{i}} {\bf 1}$. We will denote
the common value by 
\[\Tr{i}^\btrg\colon i_*i^\btrg \lra {\bf 1}.\]
We have to show that
\[v^*\smcirc\Tr{i}^\btrg = \Tr{j}^\btrg\smcirc v^*.\]
The question in local on $X$ and $X'$ and hence we assume that $X=\Spec{\,R}$, $Z=\Spec{\,A}$,
$X'=\Spec{\,R'}$, $Z'=\Spec{\,A'}$ where $A'=A\otimes_RR'$. We write $I$ for the ideal in $R$
generated by ${\bf t}$, $J$ for its extension to $R'$, $N$ for the $A$-module $\wI{A}{I}$, 
and $N'$ for $N\otimes_AA'=\wI{A'}{J}$. Finally, let $\Tr{A/R}^\btrg$ and $\Tr{A'/R'}^\btrg$
be the maps in $\D({\mathrm{Mod}}_R)$ and $\D(\mathrm{Mod}_{R'})$ whose ``sheafified" 
versions are
$\Tr{i}^\btrg$ and $\Tr{j}^\btrg$ respectively. The discussions in \Rref{rem:Tr-tensor} and
in \Ssref{sss:Tr-concrete} apply. In particular, from the commutative diagram \eqref{diag:1-tensor-Tr},
we only have to show:
\[\Tr{A/R}^\btrg(R)\otimes_RR'=\Tr{A'/R'}^\btrg(R'). \leqno{(*)}\]
This follows from the explicit description of $\Tr{A/R}(R)$ in \eqref{map:Tr-concrete}, for
the maps $\varphi_{\bf t}$ and $\pi_{\bf t}$ occuring in \textit{loc.cit.}~are compatible with base
change. In greater detail, if $t_i'$ are the images in $R'$ of $t_i$ and ${\bf t}' =(t_1', \dots, t_r')$,
then $\varphi_{\bf t}\otimes_RR'= \varphi_{{\bf t}'}$ and $\pi_{\bf t}\otimes_RR'= \pi_{{\bf t}'}$.
Since $\Tr{A/R}(R) = \pi_{\bf t}\smcirc\varphi_{\bf t}^{-1}$ and 
$\Tr{A'/R'}(R')=\pi_{{\bf t}'}\smcirc \varphi_{{\bf t}'}^{-1}$, the relation asserted in $(*)$ is true.
\qed

\begin{rem}\label{rem:bc-reg-imm} Formula $(*)$ in the above proof
is true in greater generality. Suppose
$Z=\Spec{\,A}$, $X=\Spec{\,R}$, and we have a regular immersion $i\colon Z\hookrightarrow X$
given by an $R$-sequence ${\bf t}=(t_1, \dots, t_r)$. Then $(*)$ remains valid without the
assumption that $X$ be Cohen-Macaulay over another scheme. In fact, by checking locally
one easily deduces that if $i\colon Z\hookrightarrow X$ is a regular immersion (not necessarily
of affine schemes, and not necessarily given globally by the vanishing of a sequence of the
form ${\bf t}$) and we have a cartesian diagram
\[
{\xymatrix{
Z' \ar[r]^w \ar[d]_j \ar@{}[dr]|\square & Z \ar[d]^i \\
X' \ar[r]_v & X 
}}
\]
with $u$ flat and if $\Theta\colon w^*i^!\co_X \iso j^!\co_{X'}$
is the flat base change isomorphism, then the following diagram commutes;
\[
{\xymatrix{
w^*\wnor{i}[-r]  \ar@{=}[r] \ar[d]_{\eta'_j} & \wnor{j}[-r] \ar[d]^{\eta'_i} \\
w^*i^!\co_X \ar[r]^{\Iso}_{\Theta} & j^!\co_{X'}
}}
\]
The flatness hypothesis on $v$ can be relaxed, since $(*)$ works even when $R'$ is not
flat over $R$, but for now, we leave matters as they are.
\end{rem}

\subsection{Base-change theorems}
We now prove that $\ttr{h}$ is stable under arbitrary base change. We embed that result in
a larger set of base-change results, namely in \Tref{thm:good}.
%%%%%%%%%
\begin{comment}
In the proof of \Tref{thm:good} we will repeatedly need to refer to the following kind of 
commutative cartesian
cube, where $\wit{Y}=\Spec{\,\wid{\co}_{y}}$ is the spectrum of the completion of the local ring 
$\co_{Y,y}$ at a point $y\in Y$ and $s\colon \wid{Y}\to Y$ is the natural map. The entire cube
\eqref{diag:3D-bc} results from the base change map $s$ applied to the square at the bottom
of \eqref{diag:basic}, i.e., the face of  \eqref{diag:3D-bc} occurring on the back. 
All faces in \eqref{diag:3D-bc} are cartesian. We set 
$\wit{Z}=t^{-1}(Z)$ and $\wit{Z}'=\vartheta^{-1}(Z')$.
\stepcounter{thm}
\[
\begin{aligned}\label{diag:3D-bc-old}
{\xymatrix{
& X'\ar'[d][dd]_{g} \ar[rr]^{v} & &  X \ar[dd]^f \\
\wit{X}' \ar[ru]^{\vartheta} \ar[rr]^{\phantom{XXXX}\wit{v}} \ar[dd]_{\wit{g}} & & \wit{X} \ar[ru]^{t} 
\ar[dd]_>>>>>>>>>>>>>>>>{\wit{f}} & \\
& Y' \ar'[r][rr]^{u}& & Y\\
\wit{Y}' \ar[ru]^{\sigma} \ar[rr]_{\wit{u}} & & \wit{Y} \ar[ru]_s & 
}}
\end{aligned}\tag{\thethm}
\]
\end{comment}
%%%%%%%
\begin{thm}\label{thm:good} With the hypotheses as in \Ssref{ss:hyp} we have:
\begin{enumerate}
\item[(a)] For a coherent $\co_X$-module $\eF$, the diagram
\[
{\xymatrix{
  h'_*(j^*v^*\eF\otimes_{\co_{Z'}}\eN') \ar[d]_{\eqref{map:sh-foo}}\ar@{=}[rr] &&
 u^*h_*(i^*\eF\otimes_{\co_Z}\eN)\ar[d]^{u^*\eqref{map:sh-foo}}\\
\Rr^r_{Z'}v^*\eF && u^*\Rr^r_Z\eF \ar[ll]^{\eqref{map:res-bc}}  
}}
\]
commutes.
\item[(b)] The map 
\[b(u, f)\colon u^*\Rr^r_Zf_* \lra \Rr^r_{Z'}g_*v^*\]
is an isomorphism.

\item[(c)]  
The diagram 
\[
{\xymatrix{
h'_*(j^*(\omgs{g})\otimes_{\co_{Z'}}\eN')  \ar[dd]_{\ttr{h'}} 
& & h'_*(j^*v^*(\omgs{f})\otimes_{\co_{Z'}}w^*\eN) 
\ar[ll]_{\Iso}^{h'_*(\theta_u^f\otimes{\bf 1})} \ar@{=}[d] \\
&&  h'_*w^*(i^*(\omgs{f})\otimes_{\co_{Z'}}w^*\eN)\ar@{=}[d] \\
\co_{Y'} && u^*h_*(i^*(\omgs{f})\otimes_{\co_Z}\eN) \ar[ll]^{u^*(\ttr{h})}
}}
\]
commutes, where $\theta_u^f\colon v^*\omgs{f}\iso \omgs{g}$ is the base-change
isomorphism of \cite[p.\,740,\,Theorem\,2.3.5\,(a)]{cm}.

\item[(d)] The diagram
\[
{\xymatrix{
u^*\Rr^r_Zf_*\omgs{f} \ar[d]_{u^*(\ares{Z})}\ar[rr]^{\Iso}_{b(u,f)} 
&& \Rr^r_{Z'}g_*v^*\omgs{f} \ar[d]^{\Rr^r_{Z'}g_*(\theta_u^f)} \\
\co_{Y'} && \Rr^r_{Z'}g_*\omgs{g} \ar[ll]^{\ares{{Z'}}}
}}
\]
commutes.

\end{enumerate}
\end{thm}

\proof 
From the definitions, we may, without loss of generality, assume that $Y=\Spec{\,A}$ and
$Y'=\Spec{\,A'}$. Consider the composite natural transform of functors of quasi-coherent
$\co_X$-modules:
\[
\mathrm{Ext}^r_A(\co_Z,\,\boldsymbol{-}) \lra 
\mathrm{Ext}^r_A(\co_Z,\,v_*v^*(\boldsymbol{-})) \lra 
\mathrm{Ext}^r_{A'}(\co_{Z'},\,v^*(\boldsymbol{-}))
\]
giving a base change map
\stepcounter{sth}
\begin{equation*}\label{map:ext-bc}\tag{\thesth}
A'\otimes_A\mathrm{Ext}^r_A(\co_Z,\,\boldsymbol{-}) \lra 
\mathrm{Ext}^r_{A'}(\co_{Z'},\,v^*(\boldsymbol{-}))
\end{equation*}
Fron the definition of \eqref{map:res-bc} it is easy to see that
\stepcounter{sth}
\[
\begin{aligned}\label{diag:ext-res-bc}
{\xymatrix{
A'\otimes_A\mathrm{Ext}^r_A(\co_Z,\,\boldsymbol{-}) \ar[rr]^{\eqref{map:ext-bc}} 
\ar[d]&&
\mathrm{Ext}^r_{A'}(\co_{Z'},\,v^*(\boldsymbol{-}))\ar[d]\\
A'\otimes_A\Hr^r_Z(X,\,\boldsymbol{-}) \ar[rr]_{\eqref{map:res-bc}} &&
\Hr^r_{Z'}(X',\,v^*(\boldsymbol{-}))
}}
\end{aligned}\tag{\thesth}
\]
commutes. 

Let $\Ext_f^i(\co_Z,\,\boldsymbol{-})$ be the $i^{\mathrm{th}}$ right
derived functor of $f_*\sHom_X(\co_Z,\,\boldsymbol{-})$. Since $Z$ is affine and
$\Ext_f^i(\co_Z,\,\boldsymbol{-})$ is supported on $Z$, this is simply
$h_*\mathrm{Ext}^i_X(\co_Z,\,\boldsymbol{-})$. 
Similarly, one defines $\Ext^i_g(\co_{Z'}, \boldsymbol{-})$. Using this, and
computing $\Ext^r_X(\co_Z, \boldsymbol{-})$ and $\Ext^r_{X'}(\co_{Z'}, \boldsymbol{-})$
via the Koszul resolutions on ${\bf t}$ of $\co_Z$ and $\co_{Z'}$, we get see easily that
the fundamental local isomorphisms \eqref{iso:fli} is
compatible with \eqref{map:ext-bc}. In other words, the following diagram of functors
of coherent $\co_X$-modules commutes:
\stepcounter{sth}
\[
\begin{aligned}\label{diag:fli-ext-bc}
{\xymatrix{
  h'_*(j^*v^*(\boldsymbol{-})\otimes_{\co_{Z'}}\eN') \ar[d]_{\eqref{iso:fli}}\ar@{=}[rr] &&
 u^*h_*(i^*(\boldsymbol{-})\otimes_{\co_Z}\eN)\ar[d]^{\eqref{iso:fli}}\\
\Ext_g^r(\co_{Z'},\,v^*(\boldsymbol{-})) && u^*\Ext^r_f(\co_Z,\,\boldsymbol{-}) 
\ar[ll]^{\eqref{map:ext-bc}}  
}}
\end{aligned}\tag{\thesth}
\]
This together with \eqref{diag:ext-res-bc} gives part\,(a). In particular, applied to coherent 
$\co_X$-modules, \eqref{map:ext-bc} is an isomorphism.

Applying the fact that \eqref{map:ext-bc} is an isomorphism
to the closed schemes $Z_n$ of $X$ defined by $t_1^n, \dots, t_r^n$, and
taking the direct limit as $n\to\infty$ we get (b) from \eqref{diag:ext-res-bc}.

According to  \cite[pp.\,755--756, Prop.\,6.2.2 (b) and (c)]{cm}, part (d) is true when
either $u$ is flat or when $Z\hookrightarrow X$ is a {\emph{good immersion}} for $f$, i.e., it satisfies:
\begin{itemize}
\item There is an affine open covering $\U=\{U_\alpha =\Spec{\,A_\alpha}\}$ of $Y$, and for
each index $\alpha$ an affine open scheme $V_\alpha=\Spec{\,R_\alpha}$ of $f^{-1}(U_\alpha)$ 
such that $Z\cap f^{-1}(U_\alpha)$ is contained in $V_\alpha$.
\item The closed immersion $i$ is given in $V_\alpha$ by a quasi-regular $R_\alpha$-sequence.
\item $Z\to Y$ is finite.
\end{itemize}
(See also \cite[p.\,744, Def.\,3.1.4]{cm} and \cite[pp.\,77--78, Assumptions\,4.3]{hk1}.) 

Let $\fp$ be a prime ideal of $A$, $y$ the point in $Y$ corresponding to $\fp$,
$\wit{Y}$ the completion of the local ring $A_\fp$, and
$\wit{Y}'$ the completion of $A'_\fp$ with respect to the ideal $\fp A'_\fp$. We then have a 
commutative diagram
\stepcounter{sth}
\[
\begin{aligned}\label{diag:3D-bc}
{\xymatrix{
& X'\ar'[d][dd]_{g} \ar[rr]^{v} & &  X \ar[dd]^f \\
\wit{X}' \ar[ru]^{\vartheta} \ar[rr]^{\phantom{XXXX}\wit{v}} \ar[dd]_{\wit{g}} & & \wit{X} \ar[ru]^{t} 
\ar[dd]_>>>>>>>>>>>>>>>>{\wit{f}} & \\
& Y' \ar'[r][rr]^{u}& & Y\\
\wit{Y}' \ar[ru]^{\sigma} \ar[rr]_{\wit{u}} & & \wit{Y} \ar[ru]_s & 
}}
\end{aligned}\tag{\thesth}
\]
All the lateral faces are cartesian, however the top and bottom faces need not be. We set 
$\wit{Z}=t^{-1}(Z)$ and $\wit{Z}'=\vartheta^{-1}(Z')$.

One checks easily that
\[
b(\sigma,g)\smcirc \sigma^*b(u,f) = b(u\smcirc\sigma, f)= b(\wit{u}, \wit{f})\smcirc\wit{u}^* b(s,f)
\leqno{(*)}
\]
and according to \cite[p.\,747, Remark 3.3.2]{cm}, we have
\[
\theta_{\wit{u}}^{\wit{f}}\smcirc\wit{v}^*\theta_s^f= \theta_{u\sigma}^f=\theta_\sigma^g\smcirc \vartheta^*\theta_u^f. \leqno{(\dag)}
\]
We remark
that Cohen-Macaulay maps of relative dimension $r$ are,  in the
terminology of \textit{ibid}, locally $r$-compactifiable.

From our observations about the compatibility of residues with certain base changes,
 (d) is true for the left, right and front faces of \eqref{diag:3D-bc}. Indeed, $s$ and $\sigma$
 are flat, whereas $\wit{Z}$ is a good immersion for $\wit{f}$. We therefore have:
\[
\begin{aligned}
\ares{{\wit{Z}}}\smcirc \Rr^r_{{\wit{Z}}}{\wit{f}}_*(\theta^f_s)\smcirc b(s,f) 
&= s^*(\ares{Z}) \\
\ares{{\wit{Z}'}}\smcirc \Rr^r_{{\wit{Z}'}}{\wit{g}}_*(\theta^{\wit{f}}_{\wit{u}})\smcirc b(\wit{u},\wit{f}) 
& = \wit{u}^*(\ares{{\wit{Z}}})\\
\ares{{\wit{Z}'}}\smcirc \Rr^r_{{\wit{Z}'}}{\wit{g}}_*(\theta^g_\sigma)\smcirc b(\sigma, g) 
&= \sigma^*(\ares{{Z'}})
\end{aligned}\tag{\ddag} \] 
The formulas $(*)$,  $(\dag)$, and $(\ddag)$ say that
the diagram in part (d) of the statement of the theorem commutes after applying $\sigma^*$.
Now use part (a), which we have proven, to see that the diagram in (c) commutes after
applying $\sigma^*$. Since all the sheaves involved in the diagram are cohenrent, this
means the diagram in (c) commutes in a  Zariski open neighbourhood of $u^{-1}(y)$. This
proves (c)  since $y\in Y$ is arbitrary.

Part\,(d) now follows by replacing  $Z$ by $Z_n$ as before, where $Z_n$ is defined by
$t_1^n,\dots,t_r^n$, applying (c) to $Z_n$, and taking direct limits.
\qed

\section{\bf Iterated traces}\label{s:fubini-1}
An important formula concerning residues is a Fubini like statement for iterated residues (see
\cite[p.198, (R4)]{RD}). To establish this via our approach to residues, i.e., via Verdier's isomorphism,
we have to understand iterated traces (for a composite of pseudo-proper
maps) in various ways. That is the main thrust of this section.
The circle of ideas is sometimes referred to as ``transitivity" (cf.\,\cite{jag}).
In somewhat greater detail suppose
\[\X\xrightarrow{f} \Y \xrightarrow{g} \Z\]
is a pair of pseudo-proper maps. Recall that $\Tr{f}\colon \Rfs\R\iGp{\X}\ush{f} \to {\bf{1}}_{\D(\Y)}$
factors through the natural map $\R\iGp{\Y}\to {\bf{1}}_{\D(\Y)}$. Moreover, we abuse notation
and write $\Tr{f}\colon  \Rfs\R\iGp{\X}\ush{f} \to \R\iGp{\Y}$ for the missing factor in the just mentioned
factorization of $\Tr{f}\colon \Rfs\R\iGp{\X}\ush{f} \to {\bf{1}}_{\D(\Y)}$. 
Given $\eF, \eG \in \wDqcp(\Y)$ the torsion version of the projection isomorphism,
which we shall denote as $p^t_f(\eF, \eG)$,
is the following composition
\[
\eF\overset{\bL}\otimes_{\co_\X} \Rfs\R\iGp{\X}\eG  \iso \Rfs(\bL f^* \eF\overset{\bL}\otimes_{\co_\X} \R\iGp{\X}\eG)
\iso \Rfs\R\iGp{\X}(\bL f^* \eF\overset{\bL}\otimes_{\co_\X} \eG)
\]
where the first isomorphism is induced by projection. 
In this situation, we have the following iterated trace on 
$\R (gf)_*\R\iGp{\X}(\bL f^*\ush{g}\co_\Z\overset{\bL}\otimes \ush{f}\co_\Y$)
where the map labelled~$p$ is the natural one induced by $(p^t_f)^{-1}$ while the 
one labelled $T$ is induced by $\Tr{f}$:
\begin{align*}
\R (gf)_*\R\iGp{\X}(\bL f^*\ush{g}\co_\Z\overset{\bL}\otimes \ush{f}\co_\Y) &
\xrightarrow{\Iso} \R g_*\Rfs \R\iGp{\X}(\bL f^*\ush{g}\co_\Z\overset{\bL}\otimes \ush{f}\co_\Y)\\
%&\xrightarrow{\Iso} \R g_*\Rfs (\bL f^* \ush{g}\co_\Z\overset{\bL}\otimes \R\iGp{\X}\ush{f}\co_\Y)\\
&\xrightarrow{\;\;p\;\,} \R g_*(\ush{g}\co_\Z\overset{\bL}\otimes \Rfs \R\iGp{\X}\ush{f}\co_\Y)\\
&\xrightarrow{\;\;T\;} \R g_*(\ush{g}\co_\Z\overset{\bL}\otimes \R\iGp{\Y}\co_\Y)\\
& \xrightarrow{\Iso\,} \R g_*\R\iGp{\Y}\ush{g}\co_\Z \\
& \xrightarrow{\Tr{g}\,} \co_\Z.
\end{align*}
By adjointness, this gives us a map 
\[\chi_{{}_{[g,f]}}\colon \bL f^*(\ush{g}\co_\Z)\overset{\bL}\otimes \ush{f}\co_\Y \to \ush{(gf)}\co_\Z.\]
In fact one does not need $f$ and $g$ to be pseudo-proper to define $\chi_{{}_{[g,f]}}$. Our definition
below works under the assumption that each of them is a composite of compactifiable maps. 

Part of the theme of transitivity
is to work out a concrete formula for $\chi_{{}_{[g,f]}}$ when $f$ and $g$ are smooth, and when 
$\ush{g}\co_\Z$, $\ush{f}\co_\Y$, and $\ush{(gf)}\co_\Z$ are substituted with suitable
differential forms (placed in the appropriate degree) via Verdier's isomorphism \cite{verdier}. That is 
done in a later paper based on the work done here.

%
%The simplest form of the question we explore in this section is this:
%Suppose $f\colon X\to Y$ and $g\colon Y\to Z$ are {\em smooth, proper} maps of ordinary
%schemes, of relative dimension $e$ and $d$ respectively. The Leray spectral sequence
%gives us an isomorphism of functors of quasi-coherent sheaves
%\[\Rr^dg_*\Rr^ef_* \iso \Rr^{d+e}(gf)_*.\]
%We therefore have a  Fubini-like iterated integral, namely the composite:
%\begin{align*}
%\Rr^{d+e}(gf)_*(f^*\omega_g\otimes_{\co_X}\omega_f) & 
%\xleftarrow{\,\,\,\phantom{X}\Iso\phantom{X}\,\,\,} 
%\Rr^dg_*\Rr^ef_*(f^*\omega_g\otimes_{\co_X}\omega_f)\\
%& \xleftarrow{\,\,\,\phantom{X}\Iso\phantom{X}\,\,\,} 
%\Rr^dg_*(\omega_g\otimes_{\co_Y}\Rr^ef_*\omega_f) \quad {\text{(Projection Formula)}} \\
%& \xrightarrow{\Rr^dg_*(\vin{f})} \Rr^dg_*\omega_g\\
%& \xrightarrow{\,\,\,\phantom{X}\vin{g}\phantom{X}\,\,\,\,} \co_Z.
%\end{align*}
%The universal property of the dualizing pair $(\omega_{gf},\,\vin{gf})$ then gives us a unique map
%\[\bar{\eta}_{g,f}\colon f^*\omega_g\otimes_{\co_X}\omega_f \lra \omega_{gf}\]
%such that 
%$\vin{gf}\smcirc \Rr^{d+e}(gf)_*({\bar{\eta}}_{g,f})\colon 
%\Rr^{d+e}(gf)_*(f^*\omega_g\otimes_{\co_X}\omega_f) \lra \co_Z$  is the Fubini
%composite in the display above. The map ${\bar{\eta}}_{g,f}$ has been defined via universal
%properties. The question then is, what is its concrete form in terms of differentials? 
%\Tref{thm:fubini} below answers this question. But first we need to do some necessary
%preliminary work.
\subsection{Traces in affine terms}\label{ss:aff-tr}
If $A\to B$ is a  pseudo-finite-type map of adic rings, $I\subset A$ and $J\subset B$ 
defining ideals for the adic rings $A$ and $B$ respectively,
and $f\colon \Spf{\,(B,\,J)} \to \Spf{\,(A,\,I)}$ the resulting map
of formal schemes, then the complex $\ush{f}\co_{\Spf{B}}$ can be represented by a  
bounded-below complex 
%\stepcounter{thm}
%\begin{equation*}\label{def:omega(B,J)(A,I)}\tag{\thethm}
\[
\ush\omega_{(B,J)/(A,I)}^\bullet = \ush\omega^\bullet_{B/A} \in \D^+({\mathrm{Mod}}_B)
\]
%\end{equation*}
which has finitely generated cohomology 
modules, where the more elaborate notation $\ush\omega_{(B,J)/(A,I)}^\bullet$ is used only when the
role of the adic structures needs to be emphasised. To simplify notation further, we shall use
$\omega^\bullet_{B/A}$ in place of $\ush\omega^\bullet_{B/A}$ from now on.

It then follows that if $f$ is {\em Cohen-Macaulay} then $\omega^\bullet_{B/A} = \omgs{B/A}[d]$.
%Finally if $f$ is {\em smooth} of relative dimension $d$ we set
%\[\omega^\bullet_{B/A} \set \omega_{B/A}[d]\]
%(the universally finite module of degree $d$ differentials forms for $B/A$).

Regarding the affine version of traces there are two related situations which we wish to discuss. 

\subsection*{\bf A} Suppose $A\to B/J$ is finite. Recall that the trace map 
\[\Tr{f}\colon \R\iGp{\Spf{(B,J)}}\ush{f}\co_{\Spf{(A,I)}}\to \co_{\Spf{(A,I)}}[0]\] 
factors through the natural map $\R\iGp{\Spf{(A,I)}}\co_{\Spf{(A,I)}}[0] \to \co_{\Spf{(A,I)}}[0]$ and that
the map $\R\iGp{\Spf{(B,J)}}\ush{f}\co_{\Spf{(A,I)}}\to \R\iGp{\Spf{(A,I)}}\co_{\Spf{(A,I)}}[0]$ 
inducing this trace map
is also called the trace map, and is also denoted $\Tr{f}$. Corresponding to these maps $\Tr{f}$
we have, at the affine level, two maps, again denoted by the same symbol
$\Tr{B/A}(=\Tr{(B,J)/(A,I)})$ 
\stepcounter{thm}
\begin{equation*}\label{map:Tr-B/A-I}\tag{\thethm}
\Tr{B/A}=\Tr{(B,J)/(A,I} \colon \R\Gamma_J \omega^\bullet_{B/A} \to \R\Gamma_IA[0].
\end{equation*}
and
\stepcounter{thm}
\begin{equation*}\label{map:Tr-B/A}\tag{\thethm}
\Tr{B/A}=\Tr{(B,J)/(A,I)} \colon \R\Gamma_J \omega^\bullet_{B/A} \to A[0].
\end{equation*}
Note that the two uses of the symbol $\Tr{B/A}$ occur in the following commutative diagram:
\[
{\xymatrix{
\R\Gamma_J \omega^\bullet_{B/A} \ar[dr]^{\Tr{B/A}} \ar[d]_{\Tr{B/A}}& \\
\R\Gamma_IA[0] \ar[r] & A[0]
}}
\]

\subsection*{B} Next suppose $A$ and $B$ both have discrete topology, and we have a finite-type 
map $A\to B$.
Suppose $J$ is an ideal in $B$ such that $A\to B/J$ is finite. Let ${\widehat{B}}$ be the
completion of $B$ with respect to $J$. Note that if 
$\kappa\colon \Spf{({\widehat{B}},J{\widehat{B}})} \to \Spec{\,B}$ is the completion map,
then the canonical isomorphism $\kappa^*\ush{f}  \iso \ush{(f\kappa)}$ results in a canonical
isomorphism $\omega^\bullet_{B/A}\otimes_B{\widehat{B}}
\iso \omega^\bullet_{\widehat{B}/A}$.
Define
\stepcounter{thm}
\begin{equation*}\label{map:Tr-J}\tag{\thethm}
\Tr{J} \colon \R\Gamma_J \omega^\bullet_{B/A} \to A[0]
\end{equation*}
as the composite
\begin{align*}
\R\Gamma_J \omega^\bullet_{B/A} & \xrightarrow{\,\,\,\,\Iso\,\,\,\,}
\R\Gamma_{J\widehat{B}}(\omega^\bullet_{B/A}\otimes_B\widehat{B})\\
& \xrightarrow{\,\,\,\,\Iso\,\,\,\,} \R\Gamma_{J\widehat{B}}\omega^\bullet_{\widehat{B}/A} \\
& \xrightarrow[\Tr{{\widehat{B}}/A}]{} A[0]
\end{align*}

\setcounter{subsubsection}{\value{thm}}
%\begin{rem} 
\subsubsection{} \stepcounter{thm}\label{ss:IJ-trans}
There is potential for confusion over the symbol $\omega^\bullet_{B/A}$ in a situation 
we will be in and we would like to clarify the issues here. Let $(A,I)$ and
$(B,J)$ be adic rings. Let $L=IB+J\subset B$, and assume further that $B$ is also
$L$-adically complete. Suppose there is a ring homomorphism $A \to B$ such that 
the induced map $A\to B/J$ is finite. Then $A\to B/L$ is also finite and the formal-scheme maps 
$\Spf{(B,\,J)}\xrightarrow{p} \Spec{\,A}$ and $\Spf{(B,\,L)}\xrightarrow{f} \Spf{(A,\,I)}$ are both 
pseudo-finite. Moreover we have a cartesian square as follows.
\[
{\xymatrix{
\Spf{(B,\,L)} \ar[d]_f \ar[r]^{\kappa_{{}_L}} \ar@{}[dr]|\square & \Spf{(B,\,J)} \ar[d]^p \\
\Spf{(A,\,I)} \ar[r]_{\kappa_{{}_I}} & \Spec{\,A}
}}
\]
Since $\kappa_{{}_I}$ is flat, we have 
$\kappa_{{}_L}^*\ush{p}\co_{\Spec{\,A}}\iso \ush{f}\kappa_{{}_I}^*\co_{\Spec{\,A}}
= \ush{f}\co_{\Spf{(A,\,I)}}$. This means we can, and {\em we will}, identify 
$\omega^\bullet_{(B,L)/(A,I)}$ and $\omega^\bullet_{(B,J/)(A,0)}$. Therefore, denoting the common
complex $\omega^\bullet_{B/A}$ in this situation causes no confusion. Thus,
\[\omega^\bullet_{B/A}=\omega^\bullet_{(B,L)/(A,I)}=\omega^\bullet_{(B,J)/(A,0)}.\]
We have two maps $\Tr{L}\colon \R\Gamma_L\omega^\bullet_{B/A} \to \R\Gamma_I(A[0])$
and $\Tr{J}\colon \R\Gamma_J\omega^\bullet_{B/A} \to A[0]$ corresponding to $\Tr{f}$ 
(cf.\,\eqref{map:Tr-B/A-I}) and $\Tr{p}$ (cf.\,\eqref{map:Tr-B/A}) respectively. In these circumstances, 
according to \Pref{prop:iterated-trace} in the appendix, 
the following diagram commutes:
\stepcounter{sth}
\[
%\begin{equation*}\label{diag:IJ-trans}\tag{\thesth}
\begin{aligned}\label{diag:IJ-trans}
{\xymatrix{
\R\Gamma_I\R\Gamma_J(\omega^\bullet_{B/A}) \ar[d]_{\R\Gamma_I(\Tr{J})}  \ar[rr]^{\Iso}
&& \R\Gamma_L(\omega^\bullet_{B/A}) \ar[d]^{\Tr{J}} \\
\R\Gamma_IA[0] \ar@{=}[rr]  && \R\Gamma_IA[0] 
}}
\end{aligned}\tag{\thesth}
%\end{equation*}
\]

\subsection{Abstract Transitivity}\label{ss:abs-trans}
This section is a digression on setting up a suitable bifunctor for every morphism in~$\bbG$
which will then be used to define an abstract transitivity relation.

For a morphism $f\colon\X\to \Y$ in $\bbG$, and complexes $\eF, \eG \in \wDqcp(\Y)$,
%\marginpar{Have to put in Suresh's work here}
we shall now define a bifunctorial map 
\stepcounter{thm}
\begin{equation*}\label{map:chi-f}\tag{\thethm}
\chi^{f}(\eF,\eG)\colon \bL f^*\eF\overset{\bL}{\otimes}_{\co_\X}\ush{f}\eG \lra 
\ush{f}(\eF\overset{\bL}{\otimes}_{\co_\Y}\eG)
\end{equation*}
which, a-priori, will depend on the choice of a factorization $f = f_nf_{n-1}  \cdots f_1$
where each $f_i$ is either an open immersion or a pseudoproper map. In these two special cases,
there is a simple version of this bifunctorial map and the general case is handled by putting
together these special ones. In Proposition \ref{prop:chi-welldef} below
we prove that $\chi^f(-,-)$ is independent of the choice
of the factorization. 

Since $\ush{(-)}$ is only a pre-pseudofunctor, even for $f$ any identity map, say~$f = 1_{\X}$,
some non-trivial considerations arise.
For $\eF, \eG \in \wDqcp(\X)$, we define
\[
q^{}_{\X}(\eF, \eG) \colon \eF \overset{\bL}{\otimes}_{\co_\X} \BL_{\X}\eG \to \BL_{\X}(\eF\overset{\bL}{\otimes}_{\co_\X}\eG)
\]
to be the map, which, via right adjointness of $\BL_{\X}$ to $\R\iGp{\X}$, corresponds to the composite 
of natural maps
\[
\R\iGp{\X}(\eF \overset{\bL}{\otimes}_{\co_\X} \BL_{\X}\eG) \iso \eF \overset{\bL}{\otimes}_{\co_\X} \R\iGp{\X}\BL_{\X}\eG
\to  \eF \overset{\bL}{\otimes}_{\co_\X} \eG.
\]
Below we shall define $\chi^{1_{\X}}$ to be $q^{}_{\X}$. For now, we collect a few properties of 
$q^{}_{\X}$ that we shall use.

The natural map $1 \to \BL_{\X}$ on 
$\eF \overset{\bL}{\otimes}_{\co_\X} \eG$ factors through $q^{}_{\X}(\eF, \eG)$ : 
\[
\eF \overset{\bL}{\otimes}_{\co_\X} \eG \to \eF \overset{\bL}{\otimes}_{\co_\X} \BL_{\X}\eG \xrightarrow{q_{\X}(\eF,\eG)}
\BL_{\X}(\eF\overset{\bL}{\otimes}_{\co_\X}\eG).
\]
Note that $q^{}_{\X}(\eF, \eG)$ is an isomorphism if both
$\eG$ and $\eF \overset{\bL}{\otimes}_{\co_\X} \eG$ have coherent homology or if $\eF$ is perfect,
i.e., locally isomorphic to bounded complex of finite-rank locally free modules. Also note 
that $\R\iGp{\X}q^{}_{\X}$ is an isomorphism and hence $\BL_{\X}q^{}_{\X}$,
which is isomorphic to $\BL_{\X}\R\iGp{\X}q^{}_{\X}$, is also an isomorphism,
i.e., the natural map is an isomorphism
\[
\BL_{\X}q^{}_{\X}(\eF, \eG) \colon \BL_{\X}(\eF \overset{\bL}{\otimes}_{\co_\X} \BL_{\X}\eG) \iso
\BL_{\X}(\eF\overset{\bL}{\otimes}_{\co_\X}\eG).
\]
Via natural identifications, $q_{\X}(\co_{\X}, \eG)$ identifies with the identity map on $\BL_{\X}\eG$.
Finally, to get a more explicit description of $q_{\X}$, 
if we choose the adjoint pair $(\BL_{\X}, \varepsilon)$ to $\R\iGp{\X}$ to be
$\BL_{\X}\eM = \R\sHom_{\X}(\R\iGp{\X}\cO_{\X}, \eM)$ with $\varepsilon$ being 
the following composition on canonical maps
\[
%\R\iGp{\X}\BL_{\X}\eF = 
\R\iGp{\X}\R\sHom_{\X}(\R\iGp{\X}\cO_{\X}, \eM) \iso 
\R\sHom_{\X}(\R\iGp{\X}\cO_{\X}, \eM) \overset{\bL}{\otimes}_{\X} 
\R\iGp{\X}\cO_{\X} \xrightarrow{\text{eval}} \eM,
\]
then $q_{\X}(\eF, \eG)$ can be described as the canonical map
\[
\eF \overset{\bL}{\otimes}_{\co_\X} \R\sHom_{\X}(\R\iGp{\X}\cO_{\X}, \eG) 
\to \R\sHom_{\X}(\R\iGp{\X}\cO_{\X}, \eF\overset{\bL}{\otimes}_{\co_\X}\eG).
\]

We will now set up some notation that will be useful for keeping track of the numerous issues that
arise out handling factorizations of maps into open immersions and pseudoproper maps.

Let  $f\colon\X\to \Y$ be a morphism in $\bbG$ and $f = f_n f_{n-1} \cdots f_1$
a factorization where each $f_i \colon \X_i \to \X_{i+1}$ is an open immersion or a pseudoproper map
with $\X_1 \set \X$ and $\X_{n+1} \set \Y$.
Let us assign to each $f_i$ a label $\lambda_i$, with $\lambda_i$ being one of either
$\mathsf O$ or $\mathsf P$ (where $\mathsf O$ = open immersions and $\mathsf P$ = pseudoproper maps), 
together with the requirement that each~$f_i$ lies in
the subcategory corresponding to $\lambda_i$. We shall denote the labelled map as $f_i^{\lambda_i}$ 
and the above factorization together with the assigned labels will be called 
a labelled factorization (of~$f$). The corresponding sequence $F = (f_1^{\lambda_1}, \ldots, f_n^{\lambda_n})$ will be 
called a labelled sequence of length~$n$ and~$|F|$ shall 
denote the composite~$f$. To ease notation, the labels shall often be suppressed
and we shall spell them out only when it is necessary. Thus we shall often denote a labelled map $f^{\lambda}$
by the underlying map $f$ itself. 
%In particular, a labelled sequence of length 1 will often be denoted by its underlying map itself.
If $F = (f_1, \ldots, f_n)$ and  $G = (g_1, \ldots, g_m)$ are labelled sequences, 
and if $g_1f_n$ makes sense, then we denote the composite 
labelled sequence $(f_1, \ldots, f_n,g_1, \ldots, g_m)$ as~$(F,G)$, which is evidently a 
labelled factorization of $|(F,G)| = |G||F|$.

For a labelled sequence $F= (f_1, \ldots, f_n)$, set %say $f = f_n f_{n-1} \cdots f_1$, set
\[
F^* \set \bL f_1^*\bL f_2^* \cdots \bL f_n^*, \qquad \qquad  \ush{F} \set \ush{f_1}\ush{f_2} \cdots \ush{f_n}.
\]
With $|F| = f$, there are canonical pseudofunctorial isomorphisms 
$F^* \iso \bL f^*$ and $\ush{F} \iso \ush{f}$. 
If $F, G$ are labelled sequences such that the composite $(F,G)$ exists, then
$\ush{(F, G)} = \ush{F}\ush{G}$ and $(F,G)^* = F^*G^*$.

For a labelled sequence $F = (f_1, \ldots, f_n)$ factoring $f\colon\X\to \Y$ and for 
complexes $\eF, \eG \in \wDqcp(\Y)$, we recursively define 
\[
\chi^{}_F(\eF,\eG)\colon F^*\eF\overset{\bL}{\otimes}_{\co_\X}\ush{F}\eG \lra 
\ush{F}(\eF\overset{\bL}{\otimes}_{\co_\Y}\eG)
\]
%(also to be denoted as $\chi^{}_F(\eF, \eG)$) 
as follows. %Let $n$ be the length of $F$. 
If $n=1$, then
$F = f^{\mathsf O}$ or $F = f^{\mathsf P}$
%with label either $\mathsf O$ or $\mathsf P$ 
and moreover $F^* = \bL f^* = f^*, \ush{F} = \ush{f}$. 
If $F = f^{\mathsf O}$, so that $f$ is an open immersion, then
using $\ush{f} = \BL_{\X} f^* \iso f^* \BL_{\Y} $, we take $\chi_F(\eF,\eG)$ to be the composite
along the top row of the following commutative diagram
\stepcounter{thm}
\[
\begin{aligned}\label{eq:chi-f^O}
\xymatrix{
f^*\eF \overset{\bL}{\otimes}_{\co_\X} \BL_{\X} f^*\eG 
\ar[d]_-{\rotatebox{-90}{\makebox[-0.1cm]{\Iso}\,\,}} \ar[r]^{q^{}_{\X}}  &
\BL_{\X} (f^*\eF \overset{\bL}{\otimes}_{\co_\X} f^*\eG) \ar[r]^{\Iso} 
& \BL_{\X} f^*(\eF \overset{\bL}{\otimes}_{\co_\Y} \eG) 
\ar[d]^-{\rotatebox{-90}{\makebox[-0.1cm]{\Iso}}\,\,\,} \\
f^*\eF \overset{\bL}{\otimes}_{\co_\X}  f^*\BL_{\Y}\eG  \ar[r]^{\Iso}  &
f^*(\eF \overset{\bL}{\otimes}_{\co_\X} \BL_{\Y}\eG) \ar[r]^{f^*q_{\Y}} 
& f^*\BL_{\Y}(\eF \overset{\bL}{\otimes}_{\co_\Y} \eG) 
}
\end{aligned}\tag{\thethm}
\]
where $q^{}_-$ is defined above. (The commutativity of this diagram, which will only be used later,
follows easily from the explicit description of $q^{}_-$ above.)
If $F = f^{\mathsf P}$, so that $f$ is pseudoproper, we set $\chi_F(\eF,\eG)$ to be the map 
adjoint to the composite
\[
\R f_*\R\iGp{\X}(\bL f^*\eF\overset{\bL}{\otimes}_{\co_\X}\ush{f}\eG) 
%\R f_*(\bL f^*\eF \overset{\bL}{\otimes}_{\co_\X} \R\iGp{\X}\ush{f}\eG) 
\xrightarrow[\cong]{\text{via }(p^t_f)^{-1}} \eF \overset{\bL}{\otimes}_{\co_\Y} \R f_*\R\iGp{\X}\ush{f}\eG 
\xrightarrow{\text{via }\Tr{f}} \eF\overset{\bL}{\otimes}_{\co_\Y}\eG
\]
where $p^t_-$ is the torsion projection isomorphism as defined in the beginning of~\S\ref{s:fubini-1}.
In general, if $n > 1$, we decompose $F$ as $\X \xrightarrow{f_1} \X_2 \xrightarrow{g} \Y$ 
where $G= (f_2, \ldots, f_n)$ gives a labelled factorization of $g$ while $f_1$ is
naturally a labelled sequence of length~1. 
Assuming $\chi^{}_G(\eF,\eG)$ is already defined, we define $\chi^{}_F(\eF,\eG)$ to be the composite
(with ${\otimes}_{\X} = {\otimes}_{\cO_\X}$, ${\otimes}_{\Y} = {\otimes}_{\cO_\Y}$)
\begin{align*}
F^*\eF\overset{\bL}{\otimes}_{\X}\ush{F}\eG =
f_1^*G^*\eF\overset{\bL}{\otimes}_{\X}\ush{f_1}\ush{G}\eG
&\xrightarrow{\chi^{}_{f_1}(G^*\eF, \ush{G}\eG))} \ush{f_1}(G^*\eF\overset{\bL}{\otimes}_{\X_2}\ush{G}\eG) \\
&\xrightarrow{\ush{f_1}\chi^{}_G(\eF, \eG)} \ush{f_1}\ush{G}(\eF\overset{\bL}{\otimes}_{\Y}\eG)
= \ush{F}(\eF\overset{\bL}{\otimes}_{\Y}\eG).
\end{align*}
%where the unlabelled isomorphisms are the natural pseudofunctorial ones.
In more concise terms, we may write that if $F = (f_1, G)$, then 
$\chi^{}_F = (\ush{f_1}\chi^{}_G) \circ \chi^{}_{f_1}$. 
It follows from the 
recursive nature of the definition that for any decomposition $F = (F_1,F_2)$ %is a composite labelled sequence, then 
we have $\chi^{}_F = \chi^{}_{(F_1,F_2)} = (\ush{F_1}\chi^{}_{F_2}) \circ \chi^{}_{F_1}$.

For $\eF = \co_{\X}$, via the obvious natural identifications $F^*\co_{\Y}\overset{\bL}{\otimes}_{\X}\ush{F}\eG \iso 
\ush{F}\eG$ and $\ush{F}(\co_{\Y}\overset{\bL}{\otimes}_{\Y}\eG) \iso \ush{F}\eG$ 
we see that $\chi_F(\co_{\X}, \eG)$ identifies with the identity map on~$\ush{F}\eG$.

The identity map $1_{\X}$ for any formal scheme $\X$, being in both $\mathsf O$ and $\mathsf P$,
forms a factorization of length 1 of itself for any of the two labels. With either label,  
we see that $\chi^{}_{1_{\X}}$ equals $q^{}_{\X}$ defined above. 

More generally, $f \colon \X \to \Y$ is in $\mathsf O$ and $\mathsf P$ iff $f$ is an isomorphism of $\X$ onto a connected
component of $\Y$ and so any such 
$f$ is a length-one factorization of itself with either label. 
\begin{lem}
\label{lem:chi-iso}
If $f \colon \X \to \Y$ is an isomorphism, then $\chi^{}_{f^{\mathsf P}} = \chi^{}_{f^{\mathsf O}}$.
\end{lem}
\begin{proof}
In this case, $f^* = \bL f^*$ and $f_* = \R f_*$ are both left and right adjoint to each other and the unit/counit
maps for either adjoint pair is given by the canonical isomorphisms $f_*f^* \cong 1$, $f^*f_* \cong 1$. 
Since $\ush{f} = \BL f^*$, the result 
follows from the commutativity of the outer border of the following diagram for $\eF, \eG \in \wDqcp(\Y)$, 
where to reduce clutter, the derived functors are denoted by their 
non-derived counterparts and moreover $\BL = \BL_{\X}, \iGp{}= \R\iGp{\X}$.
\[
\xymatrix{
& (f^*\eF \otimes \BL f^*\eG) \ar[ld] \ar[d] \\
\BL\iGp{}(f^*\eF \otimes \BL f^*\eG) \ar[d] \ar[r] & \BL f^*f_*\iGp{}(f^*\eF \otimes \BL f^*\eG) \ar[d] \\
\BL(f^*\eF \otimes \iGp{}\BL f^*\eG) \ar[r] \ar[d] & \BL f^*f_*(f^*\eF \otimes \iGp{}\BL f^*\eG) \ar[r] \ar[d]
& \BL f^*(\eF \otimes f_*\iGp{}\BL f^*\eG) \ar[d] \\
\BL(f^*\eF \otimes f^*\eG) \ar[r] \ar[rd] & \BL f^*f_*(f^*\eF \otimes f^*\eG) \ar[r] 
& \BL f^*(\eF \otimes f_*f^*\eG) \ar[ld] \\
& \BL f^*(\eF \otimes \eG)
}
\]
The bottom triangle is seen to commute by unravelling the definition of the projection isomorphism.
The remaining parts commute trivially. 
\end{proof}

If $f \colon \X \to \Y$ is in $\bbG$ and $F$ is a labelled factorization of $f$, for $\eF, \eG \in \wDqcp(\Y)$ we set
$\chi^f_F(\eF, \eG)$ to be the composite 
\[
\bL f^*\eF\overset{\bL}{\otimes}_{\co_\X}\ush{f}\eG  \iso F^*\eF \overset{\bL}{\otimes}_{\co_\X}\ush{F}\eG 
\xrightarrow{ \chi^{}_F} \ush{F}(\eF\overset{\bL}{\otimes}_{\co_\Y}\eG) \iso
\ush{f}(\eF\overset{\bL}{\otimes}_{\co_\Y}\eG).
\]
If $f = 1_{\X}$, then $\chi^{f}_{f} = q^{}_{\X}$ for either label as mentioned before.

%In what follows, we often use $\chi^f_F$ instead of $\chi^f_F(-,-)$ for simplicity.

\begin{propdef}
\label{prop:chi-welldef}
\hfill
\begin{enumerate} 
\item If $F_1$ and $F_2$ are two labelled factorizations of a map $f \colon \X \to \Y$ in $\bbG$, then
$\chi^f_{F_1}(-,-) = \chi^f_{F_2}(-,-)$. We thus define $\chi^f(-,-)$ in \eqref{map:chi-f}  to be 
$\chi^f_F(-,-)$ for any choice of a labelled factorization $F$ of $f$.
\item If $\X \xrightarrow{f} \Y \xrightarrow{g} \Z$ are maps in $\bbG$, then for any complexes
$\eF, \eG \in \wDqcp(\Z)$, the following diagram commutes.  
\[
\xymatrix{
\bL f^*\bL g^*\eF\overset{\bL}\otimes_{\co_\X} \ush{f}\ush{g}\eG \ar[r]^{\chi^f} \ar[d]^{\cong} & 
  \ush{f}(\bL g^*\eF \overset{\bL}\otimes_{\co_\Y} \ush{g}\eG) \ar[r]^{\ush{f}\chi^g} & 
  \ush{f}\ush{g}(\eF \overset{\bL}\otimes_{\co_\Z} \eG) \ar[d]_{\cong} \\
\bL (gf)^*\eF\overset{\bL}\otimes_{\co_\X} \ush{(gf)}\eG  \ar[rr]^{\chi^{gf}} & & \ush{(gf)}(\eF \overset{\bL}\otimes_{\co_\Z} \eG)
}
\]
\end{enumerate}
\end{propdef}

Part (ii) of  Proposition \ref{prop:chi-welldef} is the transitivity property for $\chi$. 
It is an easy consequence of part (i) and the relation $\chi^{}_{(F,G)} = (\ush{F}\chi^{}_{G}) \circ \chi^{}_{F}$
where $F, G$ are labelled factorizations of~$f,g$ respectively so that $(F,G)$ can be chosen as a labelled 
factorization of $fg$.
 
The proof of Proposition \ref{prop:chi-welldef}(i) is somewhat long and proceeds via several special cases 
of both parts (i) and (ii) first. We tackle these in the next few lemmas. In all these proofs, to reduce
clutter in numerous diagrams, we shall use the following shorthand notation where $f$ is a generic name for a map
and $\X$ for a formal scheme.
\[
f^* = \bL f^*, \qquad \otimes = \overset{\bL}\otimes, \qquad \iGp{\X} = \R\iGp{\X}, \qquad f^t_* = \R f_*\R\iGp{\X}
\]
%$f^*$ shall mean $\bL f^*$,
%$\otimes$ is the derived tensor and the subscript to $\otimes$ is also often omitted. Similarly, 
%$\R\iGp{\X}$ is usually denoted as $\iGp{\X}$. Finally, we use
%$f^t_*$ to denote $\R f_*\R\iGp{\X}$.

\begin{lem}
\label{lem:chi-parts}
Let $\X \xrightarrow{f} \Y$ be a map in $\bbG$. 
\begin{enumerate}
\item If $F_1, F_2$ are labelled 
factorizations of $f$ such that $\chi^f_{F_1} = \chi^f_{F_2}$, then for any maps $\W \xrightarrow{h} \X$ and
$\Y \xrightarrow{g} \Z$ in $\bbG$ and 
labelled factorizations $G, H$ of $g,h$ respectively, we have
$\chi^{gf}_{(F_1, G)} = \chi^{gf}_{(F_2, G)}$ and $\chi^{fh}_{(H, F_1)} = \chi^{fh}_{(H, F_2)}$.
\item For any labelled factorization $F$ of $f$ we have 
$\chi^f_{(1_{\X}, F)} = \chi^f_F = \chi^f_{(F,1_{\Y})}$.
\end{enumerate}
\end{lem}
\begin{proof}
(i) To prove that $\chi^{gf}_{(F_1, G)} = \chi^{gf}_{(F_2, G)}$ it suffices to check that for any complexes
$\eF, \eG \in \wDqcp(\Z)$  the outer border of the following diagram commutes.
\[
\xymatrix{
F_1^*G^*\eF \otimes \ush{F_1}\ush{G}\eG \ar[d]^{\cong} \ar[r] & \ush{F_1}(G^*\eF \otimes \ush{G}\eG) 
\ar[d]^{\cong} \ar[r] &  \ush{F_1}\ush{G}(\eF \otimes \eG) \ar[d]^{\cong} \\
f^*g^*\eF \otimes \ush{f}\ush{g}\eG   \ar[d]^{\cong} & \ush{f}(G^*\eF \otimes \ush{G}\eG) \ar[d]^{\cong} & 
\ush{f}\ush{g}(\eF \otimes \eG) \ar[d]^{\cong}  \\
F_2^*G^*\eF \otimes \ush{F_2}\ush{G}\eG \ar[r] & \ush{F_2}(G^*\eF \otimes \ush{G}\eG) \ar[r]
& \ush{F_2}\ush{G}(\eF \otimes \eG) 
}
\]
Along the leftmost and the rightmost columns, the composite of maps remains unchanged if in the 
objects of the middle row, $g^*, \ush{g}$ are replaced by $G^*, \ush{G}$ respectively. Thus  
the left half commutes if $\chi^f_{F_1} = \chi^f_{F_2}$ while the right one commutes for 
functorial reasons. The proof of the other relation is similar.

(ii) For $\eF, \eG \in \wDqcp(\Y)$, the following diagram, where $1 = 1_{\X}$, is easily seen to 
commute keeping in mind the isomorphisms $\ush{F} \iso \BL_{\X}\ush{F} = \ush{1}\ush{F}$.
\[
\xymatrix{
f^*\eF \otimes \ush{f}\eG \ar[d]_{\cong}\ar[r]^{\cong} & F^*\eF \otimes \ush{F}\eG \ar[ld]_{\cong}\ar[d]\ar[r] 
& \ush{F}(\eF \otimes \eG) \ar[d]_{\cong}\ar[r]^{\cong} & \ush{f}(\eF \otimes \eG) \ar[ld]^{\cong} \\
1^*F^*\eF \otimes \ush{1}\ush{F}\eG \ar[r] & \ush{1}(F^*\eF \otimes \ush{F}\eG) \ar[r] & 
\ush{1}\ush{F}(\eF \otimes \eG)
}
\] 
From the outer border we get that $\chi^f_{(1_{\X},F)} = \chi^f_F$.
The other relation is proved similarly.
\end{proof}

\begin{lem}
\label{lem:chi-samelabel}
Let $f \colon \X \to \Y$ be a map in $\bbG$ and 
$F = (f_1, \ldots, f_n)$ a labelled sequence factoring~$f$ such that all the $f_i$'s have the same label, say~$\lambda$,
so that~$f$ is also in~$\lambda$. Then $\chi^f_{F}  = \chi^f_{f^{\lambda}}$. 
%where $f$ is also being considered as a labelled sequence of length 1 with label same as that of the $f_i$'s. 
\end{lem}
\begin{proof}
%If all the labels are $\mathsf O$, then pseudofunctoriality of $(-)^*$ proves the lemma, so we shall assume 
%that all the labels are $\mathsf P$. 
It suffices to prove the case $n=2$ for then, by Lemma \ref{lem:chi-parts}(i),
in the general case we have 
$\chi^f_{(f_1, \ldots, f_n)} = \chi^f_{(f_2f_1, \ldots, f_n)}$, whence the result follows by induction.

In effect, we have reduced to proving Proposition \ref{prop:chi-welldef} (ii), %with $f,g$ pseudoproper
with $\chi^f, \chi^g, \chi^{gf}$ replaced by $\chi^{}_f, \chi^{}_g, \chi^{}_{gf}$ respectively. 
For the rest of the proof, we use the notation from there. Let $\eF,\eG \in \wDqcp(\Z)$.

If $\lambda = \mathsf O$, then the result follows from 
the outer border of the following commutative diagram where $(\ddag)$ commutes by \eqref{eq:chi-f^O}
and the remaining parts commute for trivial reasons.
\begin{small}
\[
\xymatrix{
f^*g^*\eF \otimes \BL_{\X}f^*\BL_{\Y}g^*\eG \ar[d] \ar[r] & \BL_{\X}(f^*g^*\eF \otimes f^*\BL_{\Y}g^*\eG) 
\ar[d]^{\hspace{6em}(\ddag)} \ar[r]  & \BL_{\X}f^*(g^*\eF \otimes \BL_{\Y}g^*\eG) \ar[d] \\
f^*g^*\eF \otimes \BL_{\X}f^*g^*\eG \ar[d] \ar[r] \ar[rd] & \BL_{\X}(f^*g^*\eF \otimes \BL_{\X}f^*g^*\eG) \ar[d]
& \BL_{\X}f^*\BL_{\Y}(g^*\eF \otimes g^*\eG) \ar[ldd] \ar[d]  \\
(gf)^*\eF \otimes \BL_{\X}(gf)^*\eG \ar[d] & \BL_{\X}(f^*g^*\eF \otimes f^*g^*\eG) \ar[ld] \ar[d] 
& \BL_{\X}f^*\BL_{\Y}g^*(\eF \otimes \eG) \ar[dd]  \\
\BL_{\X}((gf)^*\eF \otimes (gf)^*\eG) \ar[d] &  \BL_{\X}f^*(g^*\eF \otimes g^*\eG) \ar[rd] \\
\BL_{\X}(gf)^*(\eF \otimes \eG) \ar[rr] & &  \BL_{\X}f^*g^*(\eF \otimes \eG)
}
\]
\end{small}

If $\lambda = \mathsf P$, then 
by adjointness it suffices to check that the outer border of the following diagram commutes
where $p^t_-$ is the torsion projection isomorphism as defined in the beginning of~\S\ref{s:fubini-1}.
\[
\xymatrix{
g^t_*f^t_*(f^*g^*\eF \otimes \ush{f}\ush{g}\eG) \ar[r]^{(p^t_f)^{-1}} \ar[d] & g^t_*(g^*\eF \otimes f^t_*\ush{f}\ush{g}\eG)
\ar[r]^{\Tr{f}} \ar[d]^{(p^t_g)^{-1}}  &g^t_*(g^*\eF \otimes \ush{g}\eG) \ar[d]^{(p^t_g)^{-1}} \\
g_*f^t_*(f^*g^*\eF \otimes \ush{f}\ush{g}\eG)  \ar[d]^{\cong} & \eF \otimes g^t_*f^t_*\ush{f}\ush{g}\eG
\ar[r]^{\Tr{f}} \ar[d]^{\cong}  & \eF \otimes g^t_*\ush{g}\eG \ar[d]^{\Tr{g}} \\
(gf)^t_*((gf)^*\eF \otimes \ush{(gf)}\eG) \ar[r]^{\quad (p^t_{gf})^{-1}} & \eF \otimes (gf)^t_*\ush{(gf)}\eG
\ar[r]^{\Tr{gf}}  & \eF \otimes \eG  
}
\]
The upper rectangle on the right commutes trivially, while the lower one on the right commutes because of the 
way the composite of adjoints is identified as an adjoint pseudofunctorially. 
Commutativity of the diagram on the left 
follows easily from the outer border of the following one with obvious natural maps 
where we use $\eE = \ush{f}\ush{g}\eG$.
\[
\xymatrix{
g_*^tf_*^t(f^*g^*\eF \otimes \eE) \ar[r] \ar[d] &  g_*^tf_*(f^*g^*\eF \otimes \iGp{\X}\eE) \ar[r] \ar[ldd]
& g_*^t(g^*\eF \otimes f_*\iGp{\X}\eE) \ar[ld] \ar[d] \\
g_*f_*^t(f^*g^*\eF \otimes \eE) \ar@/_5pc/[dd]_{} \ar[d]  & g_*(g^*\eF \otimes f_*\iGp{\X}\eE) \ar[d]  \ar[ld]
& g_*(g^*\eF \otimes \iGp{\Y}f_*\iGp{\X}\eE) \ar[l] \ar[d] \\
g_*f_*(f^*g^*\eF \otimes \iGp{\X}\eE)\makebox[0pt]{\hspace{5em} $(\dagger)$} 
\ar[rd] & \eF \otimes g_*f_*\iGp{\X}\eE  \ar[rd] 
& \eF \otimes g^t_*f_*\iGp{\X}\eE \ar[l] \ar[d] \\
(gf)^t_*((gf)^*\eF \otimes \eE) \ar[r] &  (gf)_*((gf)^*\eF \otimes \iGp{\X}\eE) \ar[r] & \eF \otimes (gf)_*\iGp{\X}\eE
}
\]
Here $(\dagger)$ commutes by \cite[p.~125, Prop.~3.7.1]{notes}. Commutativity of the remaining parts is obvious.
\end{proof}

Consider a cartesian diagram in $\bbG$ as follows.
\stepcounter{thm}
\[
\begin{aligned}
\label{eq:fiber-sq}
\xymatrix{
\W \ar[d]_g \ar[r]^v & \X \ar[d]^f \\
\Z \ar[r]_u & \Y
}
\end{aligned}\tag{\thethm}
\]
Pick labelled factorizations $F = (f_1, \ldots, f_n)$ and $U = (u_1, \ldots, u_m)$ of $f,u$ respectively so that 
these in turn, induce corresponding ones $G,V$ of $g,v$ by base change in the obvious manner. 
Thus the composite map $h = fv = ug$
admits two labelled factorizations, namely, $(V,F)$ and $(G,U)$.
\begin{lem} 
\label{lem:chi-cartesian}
In the above setup, $\chi^h_{(V,F)} = \chi^h_{(G,U)}$.
\end{lem}
\begin{proof}
By decomposing the factorizations $F$ and $U$, let us first reduce to the case $m=n=1$.
For instance,
a decomposition $U = (U', U'')$, induces a corresponding one $V = (V', V'')$ and we have a
horizontally decomposed cartesian diagram as follows.
\[
\xymatrix{
\W \ar[d]_g \ar[r]^{v'} &  \W' \ar[d]^{g'} \ar[r]^{v''}  & \X \ar[d]^f \\
\Z \ar[r]_{u'} & \Z' \ar[r]_{u''} & \Y
}
\]
Set $h' \set u'g = g'v'$ and $h'' \set u''g' = fv''$. 
If $G'$ is the induced factorization of $g'$, then by Lemma \ref{lem:chi-parts}(i), it suffices to prove that  
$\chi^{h'}_{(G, U')} = \chi^{h'}_{(V',G')}$ and $\chi^{h''}_{(G', U'')} = \chi^{h''}_{(V'', F)}$. Thus 
we inductively reduce to the case $m=1$. A similar argument further reduces it to $n=1$. 
Moreover, after assuming $m=n=1$, by Lemma \ref{lem:chi-samelabel} it suffices to consider the case when $f,g$ 
have label $\mathsf P$ while $u,v$ have label $\mathsf O$. 

For this special case, using the identifications $\ush{u} = \BL_{\Z}u^*, \ush{v} = \BL_{\W}v^*$, 
proving the relation $\chi^h_{(v,f)} = \chi^h_{(g,u)}$
amounts to proving that the following diagram commutes for $\eF, \eG \in \wDqcp(\Y)$.
\begin{equation}
\begin{aligned}
\label{lem:chi-cartesian:eq1}
\xymatrix{
v^*f^*\eF \otimes \BL_{\W}v^*\ush{f}\eG \ar[d] \ar[r]^{\quad\Iso} & h^*\eF \otimes \ush{h}\eG \ar[r]^{\Iso\qquad}
& g^*u^*\eF \otimes \ush{g}\BL_{\Z}u^*\eG \ar[d] \\
\BL_{\W}v^*(f^*\eF \otimes \ush{f}\eG) \ar[d] & & \ush{g}(u^*\eF \otimes \BL_{\Z}u^*\eG) \ar[d] \\
\BL_{\W}v^*\ush{f}(\eF \otimes \eG) \ar[r]^{\quad\Iso} & \ush{h}(\eF \otimes \eG) \ar[r]^{\Iso\qquad} 
&  \ush{g}\BL_{\Z}u^*(\eF \otimes \eG) 
}
\end{aligned}
\end{equation}
The composite along each of the two rows is induced by the composite isomorphism
$\ush{v}\ush{f} \iso \ush{h} \iso \ush{g}\ush{u}$ which in fact, 
identifies with the base-change isomorphism
$\ush{\beta} \colon \BL_{\W}v^*\ush{f} \iso \ush{g}\BL_{\Z}u^*$.
Using the adjointness of $\ush{g}$ to $g^t_* = g_*\iGp{\W}$, we consider the adjoint of \eqref{lem:chi-cartesian:eq1}.
The adjoint diagram is expanded below, which, for convenience, is broken into two parts.
Thus the rightmost column of \eqref{lem:chi-cartesian:eq2} is the same as the left column of 
\eqref{lem:chi-cartesian:eq3} and the outer border of the conjoined diagram is the adjoint of \eqref{lem:chi-cartesian:eq1}. 
The maps are natural ones induced by isomorphisms
\[
\iGp{?}(\M \otimes \N) \iso \iGp{?}\M \otimes \N, \quad \iGp{?}\BL \iso \iGp{?}, \quad \iGp{\Z}u^* \iso u^*\iGp{\Y}, 
\quad \iGp{\W}v^* \iso v^*\iGp{\X},
\]
and also the isomorphisms $\iGp{\X}\ush{f} \iso f^!\iGp{\Y}$, $\iGp{\W}\ush{g} \iso g^!\iGp{\Z}$.
\begin{equation}
\label{lem:chi-cartesian:eq2}
\begin{small}
\begin{aligned}
\xymatrix{
g^t_*(v^*f^*\eF \otimes \BL_{\W}v^*\ush{f}\eG) \ar[dd] \ar[rd] \ar[rr] & 
& g^t_*(g^*u^*\eF \otimes \ush{g}\BL_{\Z}u^*\eG) \ar[d] \\
& g_*(v^*f^*\eF \otimes v^*f^!\iGp{\Y}\eG) \ar[d] \ar[r] & g_*(g^*u^*\eF \otimes g^!u^*\iGp{\Y}\eG) \ar[d] \\
g^t_*\BL_{\W}v^*(f^*\eF \otimes \ush{f}\eG) \ar[d] \ar[r] 
& g_*v^*(f^*\eF \otimes f^!\iGp{\Y}\eG)  \ar[d]^{\hspace{6em} \blacksquare} & u^*\eF \otimes g_*g^!u^*\iGp{\Y}\eG \ar[d] \\
g^t_*\BL_{\W}v^*\ush{f}(\eF \otimes \eG) \ar[d] & g_*v^*f^!(\eF \otimes \iGp{\Y}\eG) \ar[d] 
& u^*\eF \otimes u^*\iGp{\Y}\eG \ar[d] \\
g_*v^*f^!\iGp{\Y}(\eF \otimes \eG) \ar[ru] & u^*f_*f^!(\eF \otimes \iGp{\Y}\eG) \ar[r] & u^*(\eF \otimes \iGp{\Y}\eG)
}
\end{aligned}
\end{small}
\end{equation}
\begin{equation}
\label{lem:chi-cartesian:eq3}
\begin{aligned}
\xymatrix{
g^t_*(g^*u^*\eF \otimes \ush{g}\BL_{\Z}u^*\eG) \ar[d] \ar@{=}[r]
& g^t_*(g^*u^*\eF \otimes \ush{g}\BL_{\Z}u^*\eG) \ar[d] \\
g_*(g^*u^*\eF \otimes g^!u^*\iGp{\Y}\eG) \ar[d] \ar[r] & g_*(g^*u^*\eF \otimes \iGp{\W}\ush{g}\BL_{\Z}u^*\eG) \ar[d] \\
u^*\eF \otimes g_*g^!u^*\iGp{\Y}\eG \ar[d] \ar[r] & u^*\eF \otimes  g_*\iGp{\W}\ush{g}\BL_{\Z}u^*\eG \ar[d] \\
u^*\eF \otimes u^*\iGp{\Y}\eG \ar[d] \ar[r] & u^*\eF \otimes u^*\BL_{\Y}\eG\ar[d]  \\
u^*(\eF \otimes \iGp{\Y}\eG) \ar[r] & \BL_{\Z}(u^*\eF \otimes u^*\eG)
}
\end{aligned}
\end{equation}
For the commutativity of the diagram labelled as {\small{$\blacksquare$}} in \eqref{lem:chi-cartesian:eq2} we refer to
the diagram at the bottom of \cite[p.~196]{notes}. The commutativity of the remaining parts is straightforward.
\end{proof}

\begin{lem}
\label{lem:chi-fundamental}
For any labelled factorization $F$ of the identity map $1_{\X}$, we have 
$\chi^{1_{\X}}_F = \chi^{1_{\X}}_{1_{\X}} = q^{}_{\X}$.
\end{lem}
\begin{proof}
First we consider the special case where the length of $F$ is 2 , say $F = (f_1, f_2)$,
and where the label of $f_1$ is $\mathsf P$. If the label of $f_2$ is also $\mathsf P$ then the result follows from 
Lemma \ref{lem:chi-samelabel} while if the label is $\mathsf O$, then $f_2$, which is necessarily surjective
(as $f_2f_1 = 1_{\X}$), is an isomorphism. By Lemma \ref{lem:chi-iso}, $\chi^{1_{\X}}_F$ does not change
if we replace the label of $f_2$ by $\mathsf P$, and upon doing so, the result follows from Lemma \ref{lem:chi-samelabel}.

In general, fix an integer $n \ge 2$ and let $F = (f_1, \ldots, f_n)$ be a factorization of~$1_{\X}$ 
so that with  $f_i \colon \X_i \to \X_{i+1}$ we have $\X_1 = \X = \X_{n+1}$.
Let $r(F)$ be the largest integer between 1 and $n$ 
such that if $1 \le i \le r$ then $f_i$ has label $\mathsf P$. 
We prove the result by descending induction on $r(F)$. If $r(F)=n$, then the result follows by 
Lemma~\ref{lem:chi-samelabel}. Let $r(F) = k$ and assume that the result is true for any complex
for which $r > k$.
Consider the following diagram containing a fibered square 
\[
\xymatrix{
\X \ar[r]^{\Delta\qquad} & \X \times_{\X_n} \X \ar[r]^{\quad p_2} \ar[d]^{p_1} & \X \ar[d]^g \\
& \X \ar[r]^g & \X_n \ar[r]^{f_n} & \X
}
\]
where $g = f_{n-1}\cdots f_1$ is pseudoproper, $\Delta$ is the diagonal map and $p_i$
are the usual projections. For the map $g$ which is drawn in parallel  to $p_2$ we choose the 
factorization $G = (f_1, \ldots, f_{n-1})$ while for the other one we choose the 
length 1 factorization~$g$ itself with label $\mathsf P$. The parallel edges pick up 
corresponding labelled factorizations by base change: for $p_2$ we denote it as
$G' = (f'_1, \ldots, f'_{n-1})$ while for $p_1$ it is~$p_1$ itself with label $\mathsf P$.
Finally, we assign to $\Delta$, the label $\mathsf P$.

By the special case considered in the first para, 
$\chi^{1_{\X}}_{(\Delta, p_1)} = \chi^{1_{\X}}_{1_{\X}} = \chi^{1_{\X}}_{(g, f_n)}$.
Since $(\Delta, G')$ has length $n$  and $r(\Delta, G') = k+1$, hence by induction, 
$\chi^{1_{\X}}_{(\Delta, G')} =  \chi^{1_{\X}}_{1_{\X}}$.
Therefore, by Lemma \ref{lem:chi-parts} and Lemma \ref{lem:chi-cartesian},
\begin{align*}
\chi^{1_{\X}}_F = \chi^{1_{\X}}_{(G, f_n)} = 
\chi^{1_{\X}}_{(1_{\X}, G, f_n)} &= \chi^{1_{\X}}_{(\Delta, p_1, G, f_n)} \\
&= \chi^{1_{\X}}_{(\Delta, G', g, f_n)} = \chi^{1_{\X}}_{(1_{\X}, g, f_n)} = 
\chi^{1_{\X}}_{(g, f_n)} = \chi^{1_{\X}}_{1_{\X}}. 
\end{align*}
\end{proof}

Using the above lemmas, Proposition \ref{prop:chi-welldef} is proved as follows.

\begin{proof}[Proof of \textup{\ref{prop:chi-welldef}(i)}]
%Proof of (i): 
Consider the following diagram with a fibered square.
\[
 \xymatrix{
 \X \ar[r]^{\Delta\qquad} & \X \times_{\Y} \X \ar[d]_{p_1} \ar[r]^{\quad p_2} & \X \ar[d]^{f} \\
 & \X \ar[r]_{f} & \Y
 }
 \]
We choose for the map $f$ which is drawn parallel to $p_1$, the factorization $F_1$,
and denote the induced factorization of $p_1$ by $F'_1$, while for the other $f$ we choose $F_2$
as a factorization and denote the induced one on $p_2$ by $F_2'$. Assigning to~$\Delta$, 
the label~$\mathsf P$, by Lemma \ref{lem:chi-parts}, Lemma \ref{lem:chi-cartesian} 
and Lemma \ref{lem:chi-fundamental} we have 
\[
\chi^{f}_{F_2} = \chi^{f}_{(1_{\X}, F_2)} = \chi^{f}_{(\Delta, F_1', F_2)} =
\chi^{f}_{(\Delta, F'_2, F_1)} = \chi^{f}_{(1_{\X}, F_1)} = \chi^{f}_{F_1}. 
\]
\end{proof}

\begin{proof}[Proof of \textup{\ref{prop:chi-welldef}(ii)}]
%Proof of (ii): 
Let us pick labelled factorizations $F, G$ of $f,g$ respectively so that $(F,G)$ is a 
factorization for $gf$. 
It suffices to prove that the outer border of the following diagram of obvious natural maps commutes.
\[
\xymatrix{
f^*g^*\eF \otimes \ush{f}\ush{g}\eG \ar[d]^{\hspace{5em} \square_f} \ar[r] 
& \ush{f}(g^*\eF \otimes \ush{g}\eG) \ar[d] \ar[r] & \ush{f}\ush{g}(\eF \otimes \eG) \ar[d] \\
F^*g^*\eF \otimes \ush{F}\ush{g}\eG \ar[d] \ar[r] 
& \ush{F}(g^*\eF \otimes \ush{g}\eG) \ar[d]^{\hspace{5em} \square_g} \ar[r] & \ush{F}\ush{g}(\eF \otimes \eG) \ar[d] \\
F^*G^*\eF \otimes \ush{F}\ush{G}\eG \ar[d]^{\hspace{10em} \square_{gf}} \ar[r] 
& \ush{F}(G^*\eF \otimes \ush{G}\eG) \ar[r] & \ush{F}\ush{G}(\eF \otimes \eG) \ar[d] \\
(gf)^*\eF \otimes \ush{(gf)}\eG \ar[rr] &  & \ush{(gf)}(\eF \otimes \eG) 
}
\]
Here $\square_f, \square_g, \square_{gf}$ commute by definition of $\chi^f_F, \chi^g_G, \chi^{gf}_{(F,G)}$
respectively while the rest of the diagram commutes trivially.
\end{proof}

Here are some additional properties of $\chi$ that we need below.
We begin with compatibility with flat base change. For simplicity we shall assume that 
the complexes have coherent homology.
\begin{prop}\label{prop:chi-bc}
Suppose $\sigma$ is a cartesian square as follows
\[
{\xymatrix{
\V \ar[d]_{g} \ar[r]^{v} \ar@{}[dr]|\square & \X \ar[d]^{f}\\
\U \ar[r]_u  & \Y
}}
\]
with $f$ and $g$ in $\bbG$ and $u$ {\em flat}. Then 
for any $\eF, \eG \in \Dc^+(\Y)$ the diagram $D_{\sigma}$ given as follows, 
commutes:
%\stepcounter{thm}
%\begin{equation*}\label{diag:chi-bc}\tag{\thethm}
\[
{\xymatrix{
v^*(f^*\eF\overset{\bL}\otimes\ush{f}\eG) \ar@{=}[d] \ar[rr]^{v^*\chi^{f}} &&
v^*\ush{f}(\eF\overset{\bL}\otimes\eG) \ar[d]^-{\rotatebox{-90}{\makebox[-0.1cm]{\Iso}}\,\,\,} \\
g^*u^*\eF\overset{\bL}\otimes v^*\ush{f}\eG \ar[d]_-{\,\,\,\rotatebox{-90}{\makebox[-0.1cm]{\Iso}}}  & & 
\ush{g}u^*(\eF\overset{\bL}\otimes \eG)  \ar@{=}[d] \\
g^*u^*\eF\overset{\bL}\otimes \ush{g}u^*\eG \ar[rr]_{\chi^{g}} && 
\ush{g}(u^*\eF\overset{\bL}\otimes u^*\eG)
}}
\]
%\end{equation*}
where the two isomorphisms displayed are the ones arising from the flat base change 
isomorphism $\ush{\beta_\sigma} \colon v^*\ush{f}\iso \ush{g}u^*$. 
\end{prop}
\begin{proof}
%Let $D_{\sigma}$ denote the diagram whose commutativity needs to be verified.
If $f$ is an open immersion, the base-change isomorphism $\ush{\beta_\sigma}$ is induced, via canonical
identifications, by the pseudofunctoriality of $(-)^*$ and the same is true of $\chi^f, \chi^g$,
hence the result is obvious in this case. If $f$ is pseudoproper, the result follows from the proof of 
commutativity of \eqref{lem:chi-cartesian:eq1}. In general, suppose $f = f_2f_1$ where $f_i \in \bbG$,
so that $\sigma$ can be correspondingly expanded into a diagram as follows.
\[
\xymatrix{
\V \ar[d]_{g_1} \ar[r]^{v} \ar@{}[dr]|{\sigma_1} & \X \ar[d]^{f_1}\\
\W \ar[d]_{g_2} \ar[r]^{w} \ar@{}[dr]|{\sigma_2} & \Z \ar[d]^{f_2}\\
\U \ar[r]_u  & \Y
}
\]
Then checking commutativity of $D_{\sigma}$ reduces to checking that of 
the outer border of the following diagram of obvious natural maps.
\[
%\small{
\xymatrix{
v^*\!f_1^*f_2^*\eF\overset{\bL}\otimes v^*\!\ush{f_1}\ush{f_2}\eG \ar[r] \ar[d] &  
v^*(f_1^*f_2^*\eF\overset{\bL}\otimes\ush{f_1}\ush{f_2}\eG) \ar[r] \ar@{}[d]|{D_{\sigma_1}} &  
v^*\!\ush{f_1}(f_2^*\eF\overset{\bL}\otimes \ush{f_2}\eG) \ar[d] \ar[ldd] \\
g_1^*w^*\!f_2^*\eF\overset{\bL}\otimes \ush{g_1}w^*\!\ush{f_2}\eG \ar[r] \ar[d] & 
\ush{g_1}(w^*f_2^*\eF\overset{\bL}\otimes w^*\ush{f_2}\eG) \ar[d] \ar[ldd] & 
v^*\!\ush{f_1}\ush{f_2}(\eF\overset{\bL}\otimes\eG) \ar[d] \\
g_1^*g_2^*u^*\eF\overset{\bL}\otimes\ush{g_1}\ush{g_2}u^*\eG \ar[d] & 
\ush{g_1}w^*(f_2^*\eF\overset{\bL}\otimes \ush{f_2}\eG) \ar[r] \ar@{}[d]|{D_{\sigma_2}} & 
\ush{g_1}w^*\!\ush{f_2}(\eF\overset{\bL}\otimes\eG) \ar[d] \\ 
\ush{g_1}(g_2^*u^*\eF\overset{\bL}\otimes\ush{g_2}u^*\eG) \ar[r] & 
\ush{g_1}\ush{g_2}(u^*\eF\overset{\bL}\otimes u^*\eG) \ar[r] & 
\ush{g_1}\ush{g_2}u^*(\eF\overset{\bL}\otimes\eG)
}
\]
Since the unlabelled parts commute trivially, we reduce to checking commutativity of
$D_{\sigma_1}, D_{\sigma_2}$. Thus if we fix a labelled factorization of $f$ 
then proceeding inductively we reduce to the already-resolved case of when the length of the factorization is~1.
%when $f$ is either an open immersion or is pseudoproper.
\end{proof}

\begin{lem}
\label{lem:chi-kappa} 
Let $\Y$ be a formal scheme and $\I$ a coherent open 
$\co_\Y$-ideal. Let $\kappa\colon \X\to \Y$ be the completion of $\Y$ with respect to $\I$.
Let $\eF, \eG \in \Dc^+(\Y)$. Then the following diagram commutes where
the vertically drawn maps are induced by the natural isomorphisms 
$\kappa^* \iso \ush{\kappa}$ on $\Dc^+(\Y)$ while the map in the top row is the 
obvious isomorphism.
\[
\xymatrix{
\kappa^*\eF\overset{\bL}{\otimes}_{\co_\X}\kappa^*\eG 
\ar[d] \ar[r] & \kappa^*(\eF\overset{\bL}{\otimes}_{\co_\Y}\eG) \ar[d] \\
\kappa^*\eF\overset{\bL}{\otimes}_{\co_\X}\ush{\kappa}\eG 
\ar[r]^{\chi^{\kappa}}  & \ush{\kappa}(\eF\overset{\bL}{\otimes}_{\co_\Y}\eG) 
}
\]
\end{lem}
\begin{comment}
Identifying $\ush{\kappa}$ with $\kappa^*$  in the usual way \textup{(}i.e., the identification such that
 $\Tr{\kappa}$ is the composition of maps 
 $\kappa_*\R\iGp{\X}\kappa^* \iso \R\iG{\I} \to {\bf 1}_{\D(\Y)}$\textup{)}
we have 
\[\chi^\kappa_{\eF,\,\eG} \colon 
\kappa^*\eF\overset{\bL}{\otimes}_{\co_\X}\kappa^*\eG 
\lra \kappa^*(\eF\overset{\bL}{\otimes}_{\co_\Y}\eG)\]
is the identity map on 
$\kappa^*\eF\overset{\bL}{\otimes}_{\co_\X}\kappa^*\eG 
= \kappa^*(\eF\overset{\bL}{\otimes}_{\co_\Y}\eG)$.
\end{comment}
\begin{proof}
By adjointness of $\ush{\kappa}$ to $\kappa_*\R\iGp{\X}$, the assertion reduces to 
checking commutativity of the corresponding adjoint diagram which appears as 
the outer border of the following one.
\[
\xymatrix{
\kappa_*\R\iGp{\X}(\kappa^*\eF\overset{\bL}{\otimes}_{\co_\X}\kappa^*\eG) 
\ar[r] \ar[rd] \ar[dd] &  \kappa_*\R\iGp{\X}\kappa^*(\eF\overset{\bL}{\otimes}_{\co_\Y}\eG) 
\ar[r] \ar@{}[d]|{\blacksquare} & \R\iG{\I}(\eF\overset{\bL}{\otimes}_{\co_\Y}\eG) \ar[d] \\
& \eF\overset{\bL}\otimes_{\co_\Y}\kappa_*\R\iGp{\X}\kappa^*\eG \ar[d] \ar[r] 
& \eF\overset{\bL}{\otimes}_{\co_\Y}\R\iG{\I}\eG \ar[d] \\
\kappa_*\R\iGp{\X}(\kappa^*\eF\overset{\bL}{\otimes}_{\co_\X}\ush{\kappa}\eG) 
\ar[r]  & \eF\overset{\bL}\otimes_{\co_\Y}\kappa_*\R\iGp{\X}\ush{\kappa}\eG
\ar[r]  & \eF\overset{\bL}{\otimes}_{\co_\Y}\eG 
}
\]
The unlabelled parts of the above diagram commute trivially, while
commutativity of $\blacksquare$
follows from that of the outer border of the following commutative
diagram of obvious natural isomorphisms, where for convenience, $\R\iGp{-}$ is replaced
by $\iGp{-}$.
\[
\xymatrix{
\eF\overset{\bL}\otimes_{\Y}\kappa_*\iGp{\X}\kappa^*\eG \ar[r] \ar[d] & 
\kappa_*\iGp{\X}(\kappa^*\eF\overset{\bL}{\otimes}_{\X}\kappa^*\eG) \ar[r] \ar[d] &
\kappa_*\iGp{\X}\kappa^*(\eF\overset{\bL}{\otimes}_{\Y}\eG) \ar[d] \\
\eF\overset{\bL}\otimes_{\Y}\kappa_*\iGp{\X}\kappa^*\iGp{\Y}\eG \ar[r] \ar[d] & 
\kappa_*\iGp{\X}(\kappa^*\eF\overset{\bL}{\otimes}_{\X}\kappa^*\iGp{\Y}\eG) \ar[r] \ar[d] &
\kappa_*\iGp{\X}\kappa^*(\eF\overset{\bL}{\otimes}_{\Y}\iGp{\Y}\eG) \ar[d] \\
\eF\overset{\bL}\otimes_{\Y}\kappa_*\kappa^*\iGp{\Y}\eG \ar[r] \ar[rrd]& 
\kappa_*(\kappa^*\eF\overset{\bL}{\otimes}_{\X}\kappa^*\iGp{\Y}\eG) \ar[r] &
\kappa_*\kappa^*(\eF\overset{\bL}{\otimes}_{\Y}\iGp{\Y}\eG) \ar[d] \\
& & \eF\overset{\bL}{\otimes}_{\Y}\iGp{\Y}\eG
}
\]
\end{proof}

For a map $f \colon \X \to \Y$ in $\bbG$, and $\eF, \eG \in \wDqcp(\Y)$, 
we define a conjugate version of ${\chi}^f(\eF, \eG)$, denoted as $\overline{\chi}^f(\eF, \eG)$, to be
the following composite of obvious natural maps
\stepcounter{thm}
\begin{equation*}
\label{eq:chi-bar}\tag{\thethm}
\ush{f}\eF \overset{\bL}{\otimes}_{\X} \bL f^*\eG \iso \bL f^*\eG  \overset{\bL}{\otimes}_{\X} \ush{f}\eF 
\xrightarrow{{\chi}^f} \ush{f}(\eG \overset{\bL}{\otimes}_{\Y} \eF) \iso \ush{f}(\eF \overset{\bL}{\otimes}_{\Y} \eG). 
\end{equation*}

\setcounter{subsubsection}{\value{thm}}
\subsubsection{}\stepcounter{thm}{\bf Transitivity, completions and traces.} 
%\subsubsection{\bf Transitivity, completions and traces} 
We apply the abstract results of the previous subsection to relative dualizing modules.

\begin{lem}
\label{lem:completions-trans}
Suppose 
\[\X\xrightarrow{f} \Y_1 \xrightarrow{\kappa} \Y_2 \xrightarrow{g} \Z\]
are maps in $\bbG$ with $\kappa$ a completion map with respect to an ideal
$\I$ of $\co_{\Y_2}$. Then the following diagram commutes where the unlabelled arrows are 
the obvious natural isomorphisms.
\[
\xymatrix{
\bL f^*\ush{\kappa}\ush{g}\co_\Z \overset{\bL}\otimes_{\co_{\X}} \ush{f}\kappa^*\co_{\Y_2}  
\ar[r] \ar[d]_{\kappa^* \cong \;\ush{\kappa}} &
\bL f^*\ush{(g\kappa)}\co_\Z \overset{\bL}\otimes_{\co_{\X}} \ush{f}\co_{\Y_1} \ar[r]^{\qquad{\chi}^f} & 
\ush{f}\ush{(g\kappa)}\co_\Z  \ar[d] \\ 
\bL f^*\kappa^*\ush{g}\co_\Z \overset{\bL}\otimes_{\co_{\X}} \ush{f}\ush{\kappa}\co_{\Y_2} \ar[r] & 
\bL (\kappa f)^*\ush{g}\co_\Z \overset{\bL}\otimes_{\co_{\X}} \ush{(\kappa f)}\co_{\Y_2} 
\ar[r]^{\qquad\quad\chi^{\kappa f}} & \ush{(\kappa f)}\ush{g}\co_\Z
}
\]
\end{lem}
%\stepcounter{thm}\begin{equation*}\label{diag:chi-gkf}\tag{\thethm}
%\xymatrix{
%\bL f^*\ush{(g\kappa)}\co_\Z\overset{\bL}\otimes\ush{f}\co_{\Y_1} \ar[rr]^{\chi_{\<\<{}_{g\kappa,f}}}
%&& \ush{(g\kappa f)}\co_\Z \\
%\bL f^*(\kappa^*\ush{g})\co_\Z\overset{\bL}\otimes\ush{f}\kappa^*\co_{\Y_2}
%\ar[u]^{\,\rotatebox{-90}{\makebox[-0.1cm]{\Iso}}} && 
%\ar[ll]_{\Iso} \bL(\kappa f)^*\ush{g}\co_\Z\overset{\bL}\otimes\ush{(\kappa f)}\co_{\Y_2}
%\ar[u]_{\chi_{\<\<{}_{g,\kappa f}}}\\ }
%\end{equation*}
\begin{proof}
It suffices to prove that the following diagram commutes since the outer  border gives us 
the required commutativity. As before, to simplify notation, we use~$f^*$ instead of~$\bL f^*$
and drop the subscripts to~$\otimes$. For the definition of~$\overline\chi$ we refer to~\eqref{eq:chi-bar}.
\[
\xymatrix{
f^*\ush{(g\kappa)}\co_\Z \otimes \ush{f}\co_{\Y_1} \ar[r]^{\chi^f} \ar[d]
& \ush{f}(\ush{(g\kappa)}\co_\Z \otimes \co_{\Y_1}) \ar[r]^{\quad\ush{f}\overline{\chi}^{g\kappa}} 
\ar[d]^{\hspace{6em} \overline{\lozenge}} & \ush{f}\ush{(g\kappa)}\co_\Z \ar[d] \\
f^*\ush{\kappa}\ush{g}\co_\Z \otimes \ush{f}\kappa^*\co_{\Y_2} \ar[r]^{\chi^f} \ar[d]_{\kappa^* \cong \;\ush{\kappa}} & 
\ush{f}(\ush{\kappa}\ush{g}\co_\Z \otimes \kappa^*\co_{\Y_2}) 
\ar[d]^{\quad \;\;\triangle}_{\kappa^* \cong \;\ush{\kappa}} 
\ar[rd]^{\ush{f}\overline{\chi}^{\kappa}} %\ar@{}[r]|{\qquad\overline{\lozenge}}
& \ush{f}\ush{\kappa}\ush{g}\co_\Z  \ar@/^4pc/[dd]_{} \\
f^*\kappa^*\ush{g}\co_\Z \otimes_{\co_{\X}} \ush{f}\ush{\kappa}\co_{\Y_2} \ar[d] \ar[r]^{\chi^f} &
\ush{f}(\kappa^*\ush{g}\co_\Z \otimes \ush{\kappa}\co_{\Y_2}) \ar[r]^{\ush{f}\chi^{\kappa}} \ar@{}[d]|{\lozenge}
& \ush{f}\ush{\kappa}(\ush{g}\co_\Z \otimes \co_{\Y_2}) \ar[u]^{\ush{f}\ush{\kappa}\overline{\chi}^g} \\
(\kappa f)^*\ush{g}\co_\Z \otimes \ush{(\kappa f)}\co_{\Y_2} \ar[r]^{\chi^{\kappa f}}
& \ush{(\kappa f)}(\ush{g}\co_\Z \otimes \co_{\Y_2}) \ar[r]^{\quad\ush{(\kappa f)}\overline{\chi}^g} \ar[ru] 
& \ush{(\kappa f)}\ush{g}\co_\Z 
}
\]
%\textcircled{$\triangle$}
The unlabelled parts commute for functorial reasons. Both $\lozenge, \overline{\lozenge}$ commute by 
Proposition \ref{prop:chi-welldef} (ii), namely transitivity
of $\chi$ (which also implies transitivity of~$\overline{\chi}$).
Finally for $\triangle$ we use the outer border of the following diagram 
where $\eG = \ush{g}\co_\Z$ and where $\theta$ denotes the canonical isomorphism
$\eM \otimes \eN \iso \eN \otimes \eM$.
\[
\xymatrix{
\ush{\kappa}\eG \otimes \kappa^*\co_{\Y_2} \ar[r]^{\theta} \ar[d] & \kappa^*\co_{\Y_2} \otimes \ush{\kappa}\eG 
\ar[r] \ar[d]^{\qquad\ref{lem:chi-kappa}} & \ush{\kappa}(\co_{\Y_2} \otimes \eG) \ar[d] \\
\kappa^*\eG \otimes \kappa^*\co_{\Y_2} \ar[r]^{\theta} \ar[d]^{\qquad\ref{lem:chi-kappa}} \ar[rrd] & 
\kappa^*\co_{\Y_2} \otimes \kappa^*\eG \ar[r] & \kappa^*(\co_{\Y_2} \otimes \eG) \ar[d] \\
\kappa^*\eG \otimes \ush{\kappa}\co_{\Y_2} \ar[r] & \ush{\kappa}(\eG \otimes \co_{\Y_2}) \ar[r]
& \kappa^*(\eG \otimes \co_{\Y_2}) 
}
\]
\end{proof}

\begin{prop}\label{prop:chi-complete} Suppose
\[
{\xymatrix{
\wid{\X} \ar[d]_{\wid{f}} \ar[r]^{\kappa_2}  & \X \ar[d]^{f}\\
\wid{\Y}  \ar[dr]_{\wid{g}} \ar[r]^{\kappa_1} & \Y \ar[d]^g \\
& \Z
}}
\]
is a commutative diagram of formal schemes with $\kappa_1$ and $\kappa_2$ being completions
with respect to open coherent ideals of $\co_\Y$ and $\co_\X$ respectively. Then, making the 
identifications $\ush{\kappa_i}=\kappa_i^*$, $i=1,2$, the following diagram commutes, with the map
labelled $\alpha$ being the isomorphism arising from 
$\kappa_2^*\ush{f} = \ush{\kappa_2}\ush{f} \iso \ush{\wid{f}}\ush{\kappa_1} = \ush{\wid{f}}\kappa_1^*$.
%$\ush{\wid{f}}\kappa_1^* = \ush{\wid{f}}\ush{\kappa_1}\iso \ush{\kappa_2}\ush{f}=\kappa_2^*\ush{f}$.
\[
{\xymatrix{
\kappa_2^*(\bL f^*\ush{g}\co_\Z\overset{\bL}{\otimes}\ush{f}\co_\Y)
\ar@{=}[d] \ar[r]^{\qquad\kappa_2^*\chi^f} & 
\kappa_2^*\ush{f}\ush{g}\co_\Z
 \ar[r] & \kappa_2^*\ush{(gf)}\co_\Z \ar@{=}[d] \\
\kappa_2^*\bL f^*\ush{g}\co_\Z\overset{\bL}{\otimes}\kappa_2^*\ush{f}\co_\Y) 
\ar[d]^-{\alpha}_-{\,\,\,\rotatebox{-90}{\makebox[-0.1cm]{\Iso}}} 
& & 
\ush{\kappa_2}\ush{(gf)}\co_\Z \ar[d]^{\,\rotatebox{-90}{\makebox[-0.1cm]{\Iso}}} \\
\bL\wid{f}^*\kappa_1^*\ush{g}\co_\Z\overset{\bL}{\otimes} \ush{\wid{f}}\kappa_1^*\co_\Y
\ar[d]_-{\,\,\,\rotatebox{-90}{\makebox[-0.1cm]{\Iso}}}  & & \ush{(gf\kappa_2)}\co_\Z \ar@{=}[d] \\
\bL\wid{f}^*\ush{\wid{g}}\co_\Z\overset{\bL}{\otimes} \ush{\wid{f}}\co_{\wid{\Y}}
\ar[r]^{\qquad\chi^{\wid{g},\wid{f}}} & \ush{\wid{f}}\ush{\wid{g}}\co_\Z \ar[r] & \ush{(\wid{g}\wid{f})}\co_\Z
}}
\]
\end{prop}

\begin{proof}
This follows from the outer border of the following diagram where the unlabelled parts 
commute trivially.
\[
\xymatrix{
& & \kappa_2^*(f^*\ush{g}\co_\Z \otimes \ush{f}\co_\Y) 
\ar[r]^{\qquad\kappa_2^*\chi^f} \ar[d]_{\kappa_2^* \cong \ush{\kappa_2}} 
& \kappa_2^*\ush{f}\ush{g}\co_\Z \ar[d]^{\kappa_2^* \cong \ush{\kappa_2}} \\
\kappa_2^*f^*\ush{g}\co_\Z \otimes \kappa_2^*\ush{f}\co_\Y 
\ar[rru] \ar[rrd] \ar[d]_{\alpha} &  \scriptstyle\ref{lem:chi-kappa} & 
\ush{\kappa_2}(f^*\ush{g}\co_\Z \otimes \ush{f}\co_\Y) \ar[r]^{\qquad\ush{\kappa_2}\chi^f}  
& \ush{\kappa_2}\ush{f}\ush{g}\co_\Z \ar[d]_{\textup{\ref{prop:chi-welldef}(ii)}\qquad\quad} \\
\wid{f}^*\kappa_1^*\ush{g}\co_\Z \otimes \ush{\wid{f}}\kappa_1^*\co_\Y 
\ar[d]_{\kappa_1^* \cong \ush{\kappa_1}} \ar[rrd]
& & \kappa_2^*f^*\ush{g}\co_\Z \otimes \ush{\kappa_2}\ush{f}\co_\Y 
\ar[r]_-{\chi^{f\<\kappa_2}} \ar[u]^{\chi^{\kappa_2}} \ar[d] & \ush{(f\kappa_2)}\ush{g}\co_\Z \ar[d] \\
\wid{f}^*\ush{\kappa_1}\ush{g}\co_\Z \otimes \ush{\wid{f}}\kappa_1^*\co_\Y 
\ar[rr]_{\kappa_1^* \cong \ush{\kappa_1}} \ar[d] & &
\wid{f}^*\kappa_1^*\ush{g}\co_\Z \otimes \ush{\wid{f}}\ush{\kappa_1}\co_\Y 
\ar@{}[d]|{\ref{lem:completions-trans}} \ar[r]_-{\chi^{\kappa_1\<\wid{f}}} & \ush{(\kappa_1\wid{f})}\ush{g}\co_\Z \ar[d] \\
\wid{f}^*\ush{\wid{g}}\co_\Z \otimes \ush{\wid{f}}\co_{\wid{\Y}} \ar[rr] & & \ush{\wid{f}}\ush{\wid{g}}\co_\Z \ar[r] &
\ush{\wid{f}}\ush{\kappa_1}\ush{g}\co_\Z
}
\]
%\ar@{}[d]|{\blacksquare}
\end{proof}

\begin{comment}  %%%%%%%%%%
\proof \marginpar{Will have to relabel once Suresh's write-up comes in. Will not have to write ``by Property X (labelled Y)" and other awkward things}The assertion is that up to canonical isomorphisms, $\kappa_2^*(\chi_{\<\<{}_{f,g}})$ is
$\chi_{{}_{\wid{g}.\wid{f}}}$. By Property 1) (labelled ``Composites") at the beginning of 
\Ssref{ss:abs-trans}, and by \Lref{lem:chi-kappa},  $\kappa_2^*(\chi_{\<\<{}_{g,f}})$ can be identified
with $\chi_{{g, f\kappa_2}}=\chi_{{g, \kappa_1 \wid{f}}}$. By Property 3) 
(labelled ``Completions and transitivity"), $\chi_{{g, \kappa_1 \wid{f}}}$ can be identified
with $\chi_{{g\kappa_1, \wid{f}}} =\chi_{{\wid{g},\,\wid{f}}}$. This essentially completes the proof.
We leave it to the reader
to write out the appropriate commutative diagrams from the stated results in order to obtain
the commutativity of the diagram in the statement of the Proposition.
\qed
\end{comment}  
%%%%%%%%%%%

\begin{defi}
\label{def:chi-f-g}
Let $\X \xrightarrow{f} \Y \xrightarrow{g} \Z$ be maps in $\bbG$. We define $\chi_{{}_{[g,f]}}$ to the composite of the following natural maps:
\[
\bL f^*(\ush{g}\co_\Z) \overset{\bL}{\otimes}_{\co_\X} \ush{f}\co_\Y \xrightarrow{\chi^f_{}}
\ush{f}(\ush{g}\co_\Z \overset{\bL}{\otimes}_{\co_\Y} \co_\Y) \iso \ush{f}\ush{g}\co_\Y \iso \ush{(gf)}\co_\Y.
\]
\end{defi}
When, $f,g$ are both pseudoproper, this definition agrees with the one given in the beginning
of \S\ref{s:fubini-1}.

\begin{defi} 
$A\to R\to S$ be pseudo-finite maps between adic rings. 
Let $f \colon \Spf(S) \to \Spf(R)$  and $g \colon \Spf(R) \to \Spf(A)$ denote the resulting maps of
formal schemes. We define 
\[
\chi_{{}_{[S/R/A]}}\colon \omega_{R/A}^\bullet \overset{\bL}{\otimes}_R \omega_{S/R}^\bullet
\to \omega^\bullet_{S/A}
\]
to be the map corresponding to $\chi_{\<\<{}_{[g,f]}}$ of \ref{def:chi-f-g}, where $\omega^\bullet$
is defined as in the beginning of \S\ref{ss:aff-tr}.
\end{defi}

\begin{prop}\label{prop:chi-affine} Let $A\to R\to S$ be a pair of maps of rings, 
$I\subset R$, $J\subset S$ ideals, such that $A\to R/I$ and $R\to S/J$ are finite. Let
$L=IS+J$.
\begin{enumerate}
\item[(a)] Suppose $R$ is complete in the $I$-adic topology and $S$ is complete in the $L$-adic
topology (so that $S$ is then complete in the $J$-adic topology too). The following diagram 
commutes, with $\chi=\chi_{{}_{[S/R/A]}}$:
\[
{\xymatrix{
\R\Gamma_L(\omega_{R/A}^\bullet \overset{\bL}{\otimes}_R \omega_{S/R}^\bullet) 
\ar[rrr]^-{\R\Gamma_L(\chi)} &&& \R\Gamma_L\omega^\bullet_{S/A} \ar[dr]^{\Tr{S/A}} & \\
\R\Gamma_I\R\Gamma_J(\omega_{R/A}^\bullet \overset{\bL}{\otimes}_R \omega_{S/R}^\bullet)
\ar[u]^{\,\rotatebox{-90}{\makebox[-0.1cm]{\Iso}}} & & & & A[0] \\
\R\Gamma_I(\omega_{R/A}^\bullet \overset{\bL}{\otimes}_R \R\Gamma_J(\omega_{S/R}^\bullet))
\ar[u]^{\,\rotatebox{-90}{\makebox[-0.1cm]{\Iso}}} \ar[rrr]^-{\R\Gamma_I({\bf 1}\otimes \Tr{S/R})}
&&& \R\Gamma_I(\omega^\bullet_{R/A}) \ar[ur]_{\Tr{R/A}}
}}
\]
\item[(b)] Suppose the topology on $A$, $R$, and $S$ are discrete, and $A\to R$, $R\to S$
are of finite type. Then the following diagram commutes with $\chi=\chi_{{}_{[S/R/A]}}$:
\[
{\xymatrix{
\R\Gamma_L(\omega_{R/A}^\bullet \overset{\bL}{\otimes}_R \omega_{S/R}^\bullet) 
\ar[rrr]^-{\R\Gamma_L(\chi)} &&& \R\Gamma_L\omega^\bullet_{S/A} \ar[dr]^{\Tr{L}} & \\
\R\Gamma_I\R\Gamma_J(\omega_{R/A}^\bullet \overset{\bL}{\otimes}_R \omega_{S/R}^\bullet)
\ar[u]^{\,\rotatebox{-90}{\makebox[-0.1cm]{\Iso}}} & & & & A[0] \\
\R\Gamma_I(\omega_{R/A}^\bullet \overset{\bL}{\otimes}_R \R\Gamma_J(\omega_{S/R}^\bullet))
\ar[u]^{\,\rotatebox{-90}{\makebox[-0.1cm]{\Iso}}} \ar[rrr]^-{\R\Gamma_I({\bf 1}\otimes \Tr{J})}
&&& \R\Gamma_I(\omega^\bullet_{R/A}) \ar[ur]_{\Tr{I}}
}}
\]
\end{enumerate}
\end{prop}

\proof
For part\,(a), first note that the diagram below commutes:
\[
{\xymatrix{
\R\Gamma_L(\omega_{R/A}^\bullet \overset{\bL}{\otimes}_R \omega_{S/R}^\bullet) 
\ar[r]^{\Iso} & \omega_{R/A}^\bullet \overset{\bL}{\otimes}_R \R\Gamma_L(\omega_{S/R}^\bullet)\\
\R\Gamma_I\R\Gamma_J(\omega_{R/A}^\bullet \overset{\bL}{\otimes}_R \omega_{S/R}^\bullet)
\ar[r]^{\Iso} \ar[u]^{\,\rotatebox{-90}{\makebox[-0.1cm]{\Iso}}} 
& 
\omega_{R/A}^\bullet \overset{\bL}{\otimes}_R \R\Gamma_I\R\Gamma_J(\omega_{S/R}^\bullet)
\ar[u]_-{\<\rotatebox{90}{\makebox[-0.1cm]{\Iso}}} \ar[d]^-{\<\rotatebox{-90}{\makebox[-0.1cm]{\Iso}}}\\
&
 \R\Gamma_I(\omega_{R/A}^\bullet \overset{\bL}{\otimes}_R\R\Gamma_J(\omega_{S/R}^\bullet))
\ar[ul]^-{\<\<\rotatebox{-20}{\makebox[-0.1cm]{\Iso}}} 
}}
\]
The assertion now follows from the commutativity of \eqref{diag:IJ-trans} and the definition 
of~$\chi_{{}_{[S/R/A]}}$.

For part\,(b), let $\wid{R}$ be the completion of $R$ with respect to $I$, $S'$ the completion
of $S$ with respect to $J$ and $\wid{S}$ the completion of $S$ with respect to $L$.
Let $\wid{I}=I\wid{R}$, $J'=JS'$, $L'=LS'$, $\wid{L}=L\wid{S}$. Let 
$\X=\Spf{(\wid{S},\wid{L})}$, $X'=\Spf{(S',J')}$, $X=\Spec{\,S}$, 
$\Y=\Spf{(\wid{R},\wid{I})}$, $Y=\Spec{\,R}$, and $Z=\Spec{\,A}$. The various natural relations
between the adic rings can be represented by a commutative diagram of formal schemes:
\[
{\xymatrix{
\X \ar[r]^{\kappa_1} \ar[d]_{\wid{f}} & X'\ar[d]_{f'} \ar[r]^{\kappa_2} & X \ar[dl]^f \\
\Y \ar[r]^{\kappa_3}\ar[dr]_{\wid{g}} & Y \ar[d]^{g} & \\
& Z &
}}
\]
We have:
\[\kappa_1^*\ush{{f'}}\co_Y \iso \ush{(\kappa_1f')}\co_Y
\iso \ush{\wid{f}}\kappa_3^*\co_Y=\ush{\wid{f}}\co_{\Y}.\]
Moreover, $\kappa_2^*\ush{f}\co_Y\iso \ush{{f'}}\co_Y$. We may thus make the following
identifications: 
$\omega^\bullet_{S/R}\otimes_S\wid{S}= \omega^\bullet_{S/R}\otimes_S\wid{S} = 
\omega^\bullet_{\wid{S}/\wid{R}}$, and 
$\omega^\bullet_{S'/R}=\omega^\bullet_{S/R}\otimes_SS'$. The natural isomorphism 
$\kappa_3^*\ush{g}\co_Z\iso \ush{\wid{g}}\co_Z$ allows us to make the identification
$\omega^\bullet_{\wid{R}/A}=\omega^\bullet_{R/A}\otimes_R\wid{R}$.
Let us write $\chi=\chi_{{}_{[S/R/A]}}$ and $\wid{\chi}=\chi_{{}_{[\wid{S}/\wid{R}/A]}}$. The above identifications
and \Pref{prop:chi-complete} gives  $\chi\otimes\wid{S} =\wid{\chi}$, whence
 the following diagram commutes:
\[
{\xymatrix{
\R\Gamma_L(\omega^\bullet_{R/A}\overset{\bL}{\otimes}_R\omega^\bullet_{S/R}) \ar[r]^-{\chi} 
\ar[d]_-{\<\rotatebox{-90}{\makebox[-0.1cm]{\Iso}}} &
\R\Gamma_L\omega^\bullet_{S/A} \ar[d]^-{\,\rotatebox{-90}{\makebox[-0.1cm]{\Iso}}} \\
\R\Gamma_{\wid{L}}
(\omega^\bullet_{\wid{R}/A}\overset{\bL}{\otimes}_{\wid{R}}\omega^\bullet_{\wid{S}/\wid{R}}) 
\ar[r]_-{\wid{\chi}} & \R\Gamma_{\wid{L}}\omega^\bullet_{\wid{S}/A}
}}
\]

By part\,(a), it is therefore enough to prove that the diagram below commutes:
\[
\begin{aligned}
{\xymatrix{
\R\Gamma_L(\omega_{R/A}^\bullet \overset{\bL}{\otimes}_R \omega_{S/R}^\bullet) \ar[r]^{\Iso}
& \R\Gamma_{\wid{L}}(\omega_{\wid{R}/A}^\bullet \overset{\bL}{\otimes}_{\wid{R}} 
\omega_{\wid{S}/\wid{R}}^\bullet) \\
\R\Gamma_I\R\Gamma_J(\omega_{R/A}^\bullet \overset{\bL}{\otimes}_R \omega_{S/R}^\bullet)
\ar[u]^-{\<\rotatebox{90}{\makebox[-0.1cm]{\Iso}}} \ar[r]^{\Iso} &
\R\Gamma_{\wid{I}}\R\Gamma_{\wid{J}}(\omega_{\wid{R}/A}^\bullet \overset{\bL}{\otimes}_{\wid{R}} 
\omega_{\wid{S}/\wid{R}}^\bullet) \ar[u]_-{\<\rotatebox{90}{\makebox[-0.1cm]{\Iso}}} \\
\R\Gamma_I(\omega_{R/A}^\bullet \overset{\bL}{\otimes}_R \R\Gamma_J(\omega_{S/R}^\bullet)) 
\ar[u]^-{\<\rotatebox{90}{\makebox[-0.1cm]{\Iso}}} \ar[r]^{\Iso} \ar[d]_{\Tr{J}} 
& \R\Gamma_{\wid{I}}(\omega_{\wid{R}/A}^\bullet \overset{\bL}{\otimes}_{\wid{R}} 
\R\Gamma_{\wid{J}}(\omega_{\wid{S}/\wid{R}}^\bullet)) 
\ar[u]_-{\<\rotatebox{90}{\makebox[-0.1cm]{\Iso}}} \ar[d]^{\Tr{\wid{S}/\wid{R}}} \\
\R\Gamma_I(\omega^\bullet_{R/A}) \ar[d]_{\Tr{I}} \ar[r]^\Iso & 
R\Gamma_{\wid{I}}(\omega^\bullet_{\wid{R}/A}) \ar[d]_{\Tr{\wid{R}/A}} \\
A[0]  \ar@{=}[r] & A[0]
}}
\end{aligned}
\leqno{(*)}\]
The proof of the commutativity of $(*)$ is as follows. Suppose $F$ is a bounded-below complex of
$R$-modules with finitely generated cohomology and $G$ is a bounded-below complex of
finitely generated $S$-modules, we have a bifunctorial commutative diagram  
(with $\wid{F}=F\otimes_R\wid{R}$, $G'=G\otimes_SS'$, and $\wid{G}=G\otimes_S\wid{S}$):
\[
{\xymatrix{
\R\Gamma_L(F\overset{\bL}{\otimes}_R G) \ar[r]^-{\Iso}
& \R\Gamma_{L'}(F \overset{\bL}{\otimes}_R G')
\ar[r]^{\Iso}
& \R\Gamma_{\wid{L}}(\wid{F} \overset{\bL}{\otimes}_{\wid{R}} 
\wid{G}) \\
\R\Gamma_I\R\Gamma_J(F\overset{\bL}{\otimes}_R G)
\ar[u]^-{\<\rotatebox{90}{\makebox[-0.1cm]{\Iso}}} \ar[r]^{\Iso} & 
\R\Gamma_I\R\Gamma_{J'}(F \overset{\bL}{\otimes}_R G') 
\ar[u]_-{\<\rotatebox{90}{\makebox[-0.1cm]{\Iso}}} \ar[r]^{\Iso}&
\R\Gamma_{\wid{I}}\R\Gamma_{\wid{J}}(\wid{F} \overset{\bL}{\otimes}_{\wid{R}} 
\wid{G}) \ar[u]_-{\<\rotatebox{90}{\makebox[-0.1cm]{\Iso}}} \\
\R\Gamma_I(F \overset{\bL}{\otimes}_R \R\Gamma_J G) 
\ar[u]^-{\<\rotatebox{90}{\makebox[-0.1cm]{\Iso}}} \ar[r]^{\Iso}  & 
\R\Gamma_I(F \overset{\bL}{\otimes}_R 
\R\Gamma_{J'}G') 
\ar[u]_-{\<\rotatebox{90}{\makebox[-0.1cm]{\Iso}}} \ar[r]^{\Iso} &
\R\Gamma_{\wid{I}}(\wid{F} \overset{\bL}{\otimes}_{\wid{R}} 
\R\Gamma_{\wid{J}}\,\wid{G})
\ar[u]_-{\<\rotatebox{90}{\makebox[-0.1cm]{\Iso}}}  \\
}}
\]
This shows that the top two rectangles in $(*)$ commute. For the rest of $(*)$ it is enough to show
that the following diagram commutes:
\[
{\xymatrix{
\R\Gamma_I(\omega_{R/A}^\bullet \overset{\bL}{\otimes}_R \R\Gamma_J(\omega_{S/R}^\bullet)) 
\ar[d]_-{\<\rotatebox{-90}{\makebox[-0.1cm]{\Iso}}}  \ar[rr]^-{\Tr{J}} 
&& \R\Gamma_I(\omega^\bullet_{R/A}) \ar@{=}[d] \ar[dr]^{\Tr{I}}  & \\
\R\Gamma_I(\omega_{R/A}^\bullet \overset{\bL}{\otimes}_R 
\R\Gamma_{J'}(\omega_{S'/R}^\bullet)) 
\ar[d]_-{\<\rotatebox{-90}{\makebox[-0.1cm]{\Iso}}}  \ar[rr]^-{\Tr{S'/R}} &&
\R\Gamma_I(\omega^\bullet_{R/A}) \ar[d]_-{\<\rotatebox{-90}{\makebox[-0.1cm]{\Iso}}}& A[0]  \\
\R\Gamma_{\wid{I}}(\omega_{R/A}^\bullet \overset{\bL}{\otimes}_{\wid{R}} 
\R\Gamma_{\wid{J}}(\omega_{\wid{S}/\wid{R}}^\bullet)) 
 \ar[rr]_-{\Tr{\wid{S}/\wid{R}}} && \R\Gamma_{\wid{I}}(\omega^\bullet_{\wid{R}/A}) 
 \ar[ru]_{\Tr{\wid{R}/A}} &
}}
\]
The rectangle on the top commutes by definition of $\Tr{J}$. The triangle on the right end of the
diagram commutes by definition of $\Tr{I}$. The rectangle at the bottom commutes by
flat base change, since the following diagram is cartesian:
\[
{\xymatrix{
\Spf{(\wid{S},\,\wid{J})} \ar[d] \ar[r] \ar@{}[dr]|\square & \Spf{(S',\,J')} \ar[d] \\
\Spf{(\wid{R},\,\wid{I})} \ar[r] & \Spec{\,R} 
}}
\]
%In slightly greater detail,
%the horizontal row on the top is the composite
%\begin{align*}
%\R\Gamma_L(\omega_{R/A}^\bullet \overset{\bL}{\otimes}_R \omega_{S/R}^\bullet)
%& \xrightarrow {\Iso} 
%\R\Gamma_{L^*}(\omega_{R/A}^\bullet \overset{\bL}{\otimes}_R \omega_{S^*/R}^\bullet) \\
%& \xrightarrow{\Iso} \R\Gamma_{\wid{L}}(\omega_{R/A}^\bullet \overset{\bL}{\otimes}_{\wid{R}} 
%\omega_{\wid{S}/\wid{R}}^\bullet),
%\end{align*}
%the horizontal arrow in the second row breaks up as
%\begin{align*}
%\R\Gamma_I\R\Gamma_J(\omega_{R/A}^\bullet \overset{\bL}{\otimes}_R \omega_{S/R}^\bullet)
%& \xrightarrow{\Iso} 
%\R\Gamma_I\R\Gamma_{J^*}(\omega_{R/A}^\bullet \overset{\bL}{\otimes}_R 
%\omega_{S^*/R}^\bullet) \\
%& \xrightarrow{\Iso}
%\R\Gamma_{\wid{I}}\R\Gamma_{\wid{J}}(\omega_{R/A}^\bullet \overset{\bL}{\otimes}_{\wid{R}} 
%\omega_{\wid{S}/\wid{R}}^\bullet) \end{align*}
\qed

\section{\bf Iterated residues}\label{s:fubini-res}

\subsection{Comment on Translations}\label{ss:translation} This is more of an orienting remark.
Suppose $M$ and $N$ are $\co_X$-modules on a ringed
space $(X,\,\co_X)$, and $d$, $e$ are integers. According to \cite[pp.28--29, (1.5.4)]{notes}
the functor $F_{N[d]}=(\boldsymbol{-})\otimes N[d]$ on the homotopy category 
%or the derived category 
of complexes of $\co_X$-modules is triangle preserving 
%on complexes of $\co_X$-modules, 
with the isomorphism $(A^\bullet[1])\otimes N[d] \iso (A^\bullet\otimes N[d])[1]$
being the identity map (``without the intervention of signs" in the language of \cite{conrad}).
%where $\otimes$ is the tensor product over the homotopy category or the derived category.
Signs do intervene if
the first argument in the tensor product is fixed and the second varies. However, if the fixed first
argument is an $\co_X$-module, i.e., a complex concentrated in the $0^{\mathrm{th}}$-spot, then
signs do not intervene. More precisely,
$G_M=M\otimes (\boldsymbol{-})$ is triangle preserving, for the identity isomorphism
$M\otimes (B^\bullet[1]) \iso (M\otimes B^\bullet)[1]$. The same sign conventions apply 
for the derived tensor product on the derived category, see 
\cite[pp.62--63, (2.5.7)]{notes}.

For complexes of $\co_X$-modules $A^\bullet$ and $B^\bullet$, let 
\[\theta^{ij}\colon (A^\bullet[i])\otimes (B^\bullet[j]) \iso (A^\bullet\otimes B^\bullet)[i+j]\] 
be as in \cite[pp.28--29, (1.5.4)]{notes}. Then the following composite is a composite of {\em identity maps} and hence
is the identity map.
\stepcounter{thm}
\begin{align*}\label{map:[d][e]}\tag{\thethm}
M[e]\otimes N[d] \xrightarrow[\theta^{e0}]{\phantom{X}\Iso\phantom{X}} (M\otimes N[d])[e] 
&\xrightarrow[\theta^{0d}]{\phantom{X}\Iso\phantom{X}} (M\otimes N)[d][e] \\
& \xrightarrow{\phantom{X}=\phantom{X}} (M\otimes N)[d+e].
\end{align*}
(Strictly speaking, \eqref{map:[d][e]} is the identity map when the tensor product is in the ordinary category of complexes; over the derived category, the induced map on the homology in degree
$-(d+e)$ canonically identifies with the identity map. In particular, if either of $M,N$ is flat 
as $\co_X$-modules, then \eqref{map:[d][e]}, viewed as a derived-category map, also canonically identifies with identity.)

Thus, given a map $\bar{\psi}\colon M\otimes N \to T$ of $\co_X$-modules, we get a
map in $\D(X)$
\stepcounter{thm}
\begin{equation*}\label{map:[d][e]2}\tag{\thethm}
\psi\colon M[e]\otimes N[d] \to T[d+e] 
\end{equation*}
given by $(\bar{\psi}[d+e]) \smcirc\eqref{map:[d][e]}$. 
The maps $\bar{\psi}$ and $\psi$ determine each 
other. Indeed, $\bar{\psi}=\Hr^{-(d+e)}(\psi)$.

\subsection{Iterated generalized fractions}\label{ss:leray}
Let $R$ be a (noetherian) ring, $I\subset R$ an ideal generated by 
${\bf u}=(u_1,\dots,u_d)$. For any $R$-module $M$ we have a map of complexes
\stepcounter{thm}
\begin{equation*}\label{map:K-H}\tag{\thethm}
M[d]\otimes_RK^\bullet_\infty({\bf u}) \lra \Hr_I^d(M)[0]
\end{equation*}
defined on $0$-cochains by
\[
m\otimes\frac{1}{u_1^{\alpha_1}\ldots u_d^{\alpha_d}} \mapsto (-1)^d
\begin{bmatrix}
m\\
u_1^{\alpha_1}, \ldots, u_d^{\alpha_d}
\end{bmatrix}.
\]
This is a map of complexes since every $0$-cochain of $M[d]\otimes_RK^\bullet_\infty({\bf u})$
(and of $\Hr^d_I(M)[0]$) is a $0$-cocycle and because $\Hr^d_I(M)[0]$ is a complex concentrated
only in degree $0$. In the event $M$ is a free $R$-module and ${\bf u}$ is a quasi-regular sequence
(or if ${\bf u}$ is locally an $M$-sequence), \eqref{map:K-H} is a quasi-isomorphism. The map
\eqref{map:K-H} is functorial in $M$.
%
%Equivalently, since $\Hr^i_I(M)=0$ for $i\ge 0$, we have a map in $\D({\mathrm{Mod}}_R)$
%\stepcounter{thm}
%\begin{equation*}\label{diag:[d][e]3}
%\R\Gamma_I(M[d])\to \Hr^d_I(M)[0]
%\end{equation*}
%and the following diagram commutes in $\D({\mathrm{Mod}}_R)$

The above is a map of complexes, i.e., a morphism in the category $\C({\mathrm{Mod}}_R)$. There is
an analogous map in $\D({\mathrm{Mod}}_R)$ described as follows. Since $\Hr^j_I(M)=0$ for
$j>d$, there is a canonical map in  $\D({\mathrm{Mod}}_R)$,
\stepcounter{thm}
\begin{equation*}\label{gamma-I-H} \tag{\thethm}
\phi_{R,I}(M)\colon \R\Gamma_I(M[d]) \lra \Hr^d(M)[0]
\end{equation*}
such that $\Hr^0(\phi_{R,I}(M))$ is the identity map 
on $\Hr^d_I(M)$. One checks, using the definition of
the generalized fraction $\bigl[\begin{smallmatrix} m\\u_1^{\alpha_1}, \ldots, u_d^{\alpha_d}\bigr]
\end{smallmatrix}$, that the following diagram commutes in $\D({\mathrm{Mod}}_R)$
\stepcounter{thm}
\[
\begin{aligned}\label{diag:K-Gamma-I}
{\xymatrix{
M[d]\otimes_RK^\bullet_\infty({\bf u}) 
\ar[d]^-{\<\rotatebox{-90}{\makebox[-0.1cm]{\Iso}}}_{\eqref{iso:k-infty-gam}}
\ar[drr]^{\,\,\eqref{map:K-H}} && \\
\R\Gamma_I(M[d]) \ar[rr]_{\phi_{R,I}} && \Hr^d_I(M)[0]
}}
\end{aligned}\tag{\thethm}
\]

Next, suppose $S$ is an $R$-algebra and $J\subset S$ is an $S$-ideal generated by 
${\bf v}=(v_1,\ldots,v_e)$. Suppose $N$ is an $R$-module. We have an isomorphism
\stepcounter{thm}
\begin{equation*}\label{iso:leray-lc}\tag{\thethm}
\Hr^{d+e}_{IS+J}(M\otimes_RN) \iso \Hr^d_I(M\otimes_R\Hr^e_J(N))
\end{equation*}
given by
\stepcounter{thm}
\begin{equation*}\label{iso:leray-gf}\tag{\thethm}
\begin{bmatrix}
m\otimes n \\
v_1^{\beta_1}, \ldots, v_e^{\beta_e}, u_1^{\alpha_1}, \ldots, u_d^{\alpha_d}
\end{bmatrix}
\mapsto
\begin{bmatrix}
m\otimes \begin{bmatrix}
n \\
v_1^{\beta_1}, \ldots, v_e^{\beta_e}
\end{bmatrix}\\
u_1^{\alpha_1}, \ldots, u_d^{\alpha_d}
\end{bmatrix}
\end{equation*}
We claim that the following diagram commutes where we identify $M[d]\otimes_RN[e]$ with
$(M\otimes_RN)[d+e]$ as in \eqref{map:[d][e]}:
\stepcounter{thm}
\[
\begin{aligned}\label{diag:trans-infty}
{\xymatrix{
(M[d]\otimes_RN[e])\otimes_SK^\bullet_\infty({\bf{v, u}},S) 
\ar@{=}[d] \ar[r]^-{\eqref{map:K-H}} & \Hr^{d+e}_{IS+J}(M\otimes N)[0] \ar[dd]^{\eqref{iso:leray-lc}}\\
M[d]\otimes_R(N[e]\otimes_SK^\bullet_\infty({\bf v},S))\otimes_RK^\bullet_\infty({\bf u},R) 
\ar[d]_{\eqref{map:K-H}} & \\
M[d]\otimes_R\Hr^e_K(N)[0]\otimes_RK^\bullet_\infty({\bf u},R) 
\ar[r]_-{\eqref{map:K-H}} & \Hr^d_I(M\otimes_R\Hr^e_J(N))[0]
}}
\end{aligned}\tag{\thethm}
\]
Indeed, consider a $0$-cocycle 
$m\otimes n\otimes \frac{1}{v_1^{\beta_1} \ldots 
v_e^{\beta_e}, u_1^{\alpha_1} \ldots u_d^{\alpha_d}}$ of 
$M[d]\otimes N[e]\otimes K^\bullet_\infty({\bf{v, u}}, S)$. Its image along either possible
route (east-followed-by-south or south-followed-by-east) is 
$(-1)^{d+e}\bigl [\begin{smallmatrix}
m\otimes\bigl [\begin{smallmatrix}n\\{\bf v}^\beta\end{smallmatrix} \bigr ]\\ 
{\bf u}^\alpha 
\end{smallmatrix}\bigr]$. This proves that \eqref{diag:trans-infty} commutes.

The above is diagram in the category of complexes $\C({\mathrm{Mod}}_R)$. This can
be upgraded to the following:

\begin{prop}\label{prop:trans-lc} Suppose $S$-modules $N$ and $\Hr^e_J(N)$ are flat over the ring $R$, so that
${\boldsymbol{-}}\overset{\bL}{\otimes}_RN[e]={\boldsymbol{-}}\otimes_RN[e]$ and 
${\boldsymbol{-}}\overset{\bL}{\otimes}_R\Hr^e_J(N)[0] ={\boldsymbol{-}}\otimes_R\Hr^e_J(N)$.
Then the following diagram commutes in $\D({\mathrm{Mod}}_R)$:
\begin{equation*}
{\xymatrix{
\R\Gamma_{IS+J}(M[d]\otimes_RN[e])
 \ar[r]^-{\phi_{S,IS+J}} & \Hr^{d+e}_{IS+K}(M\otimes N)[0] \ar[ddd]^{\eqref{iso:leray-lc}}\\
\R\Gamma_{IS}\R\Gamma_J(M[d]\otimes_RN[e])  
\ar[u]^-{\<\rotatebox{-90}{\makebox[-0.1cm]{\Iso}}} & \\
\R\Gamma_I(M[d]\overset{\bL}{\otimes}_R\R\Gamma_JN[e]) 
\ar[u]^-{\<\rotatebox{-90}{\makebox[-0.1cm]{\Iso}}} \ar[d]_{\phi_{S,J}} & \\
\R\Gamma_I(M[d]\otimes_R\Hr^e_J(N)[0])
\ar[r]_-{\phi_{R,I}} & \Hr^d_I(M\otimes_R\Hr^e_J(N))[0]
}}
\end{equation*}
\end{prop}

\proof This is a straightforward re-interpretation of the commutativity of \eqref{diag:trans-infty} using
the commutativity of \eqref{diag:K-Gamma-I}.
\qed

\begin{rems}\label{rem:leray} {\bf{(1)}} {\em The assumptions on the flatness of $N$ and $\Hr^e_J(N)$
over $R$ are perhaps not necessary in view of \eqref{diag:trans-infty}, but since we are dealing
now with objects and functors in the derived category $\D({\mathrm{Mod}}_R)$, we have to make
sure that the arrows in the diagram in the proposition make sense. 
It was not clear to us that the second and third
arrows in the column on the left are meaningful in the derived category without our assumptions.

{\bf{(2)}} \Pref{prop:trans-lc} can be interpreted
as saying \eqref{iso:leray-lc} is the isomorphism given by the Leray
spectral sequence for the composite functor $\Gamma_{IS+J}=\Gamma_I\smcirc \Gamma_J$.
See \cite[Proposition\,(3.3.1)]{jag} as well as the correction by the second author.
}
\end{rems}

\subsection{Cohen-Macaulay maps and iterated residues}\label{ss:chi-cm}
Suppose $X=\Spec{\,S}$, $Y=\Spec{\,R}$, $Z=\Spec{\,A}$ are affine schemes, and
$f\colon X\to Y$ is Cohen-Macaulay of relative dimension $e$, $g\colon Y\to Z$ is Cohen-Macaulay
of relative dimension $d$. Note that we have finite type maps of rings $A\to R$ and $R\to S$.
Suppose $I \subset R$ and $J\subset S$ are as in \Ssref{ss:leray} with the added condition
that the given generators of $I$ and $J$, namely ${\bf u}=(u_1,\ldots, u_d)$ and 
${\bf v}=(v_1,\ldots, v_e)$ repectively are quasi-regular, and that $A\to R/I$ and $R\to S/J$ are 
finite. 

Since $A\to R$ and $R\to S$ are flat with Cohen-Macaulay fibres, under our hypotheses,
$A\to R/I$ and $R\to S/J$ are finite and flat, i.e., Cohen-Macaulay of relative dimension $0$.
Let $L=IS+J$, $W_1=\Spec{\,S/J}\hookrightarrow X$, $W_2=\Spec{\,R/I}\hookrightarrow Y$,
and $W=W_1\cap f^{-1}(W_2) = \Spec{\,S/L}\hookrightarrow X$.

In what follows, for $M\in{\mathrm{Mod}}_R$ and $N\in{\mathrm{Mod}}_S$ we make the standard
identifications, $\Hr^d_I(M)=\Hr^d_{W_2}(Y,\,\wit{M})$, $\Hr^e_J(N)=\Hr^e_{W_1}(X,\,\wit{N})$, and
$\Hr^{d+e}_L(N)=\Hr^{d+e}_W(\wit{N})$. We remind the reader that 
$\omega_{R/A}^\bullet=\omgs{R/A}[d]$, $\omega_{S/R}^\bullet=\omgs{S/R}[e]$, and
$\omega_{S/A}^\bullet=\omgs{S/A}[d+e]$.

Finally, let $\wid{R}$ be the $I$-adic completion of $R$, $\wid{S}$ the $L$-adic completion of $S$, 
and $S^*$ the $J$-adic completion of $S$. Let $\wid{J}=J\wid{S}$, $\wid{L}=L\wid{S}$, and 
$\wid{I}=I\wid{R}$. Note that $\wid{R}\to \wid{S}/\wid{J}$ is finite. Let 
$\omgs{\wid{S}/A}=\omgs{S/A}\otimes_S\wid{S}$, $\omgs{\wid{S}/\wid{R}}=\omgs{S/R}\otimes_S\wid{S}$ and $\omgs{\wid{R}/A}=\omgs{R/A}\otimes_R\wid{R}$
The maps $\Tr{(\wid{S}, \wid{L})/A}$,
$\Tr{(\wid{S}, \wid{J})}/\wid{R}$, and $\Tr{\wid{R}/A}$ give rise, on applying the cohomology functor
$\Hr^0(\boldsymbol{-})$ to maps $\tin{\wid{S}/A}\colon \Hr^{d+e}_{\wid{L}}(\omgs{\wid{S}/A})\to A$,
$\tin{\wid{S}/\wid{R}}\colon \Hr^{e}_{\wid{J}}(\omgs{\wid{S}/\wid{R}})\to \wid{R}$, and
$\tin{\wid{R}/A}\colon \Hr^{d}_{\wid{I}}(\omgs{\wid{R}/A})\to A$.

\begin{prop}\label{prop:leray-chi} Let notations be as above.
\begin{enumerate}
\item[(a)]
The following diagram commutes, with $\chi=\chi_{{}_{[g,f]}}$:
\[
{\xymatrix{
\Hr^{d+e}_L(\omgs{R/A}\otimes_R\omgs{S/R}) \ar[rr]^-{\Hr^{d+e}_L({\chi})}
\ar[d]_{\eqref{iso:leray-lc}}^{\,\rotatebox{-90}{\makebox[-0.1cm]{\Iso}}} && \Hr^{d+e}_L(\omgs{S/A})
\ar[dd]^{\ares{{W}}}\\
\Hr^d_I(\omgs{R/A}\otimes_R\Hr^e_J(\omgs{S/R})) \ar[d]_{\text{{\em{via}} $\ares{{W_1}}$}} & & \\
\Hr^d_I(\omgs{R/A}) \ar[rr]_{\ares{{W_2}}} && A
}}
\]

\item[(b)]  Let 
$\wid{\chi}=\chi_{{}_{[\wid{S}/\wid{R}/A]}}$. Then the following diagram commutes:
\[
{\xymatrix{
\Hr^{d+e}_{\wid{L}}(\omgs{\wid{R}/A}\otimes_{\wid{R}}\omgs{\wid{S}/\wid{R}}) 
\ar[rr]^-{{\text{{\em via} $\wid{\chi}$}}}
\ar[d]_{\eqref{iso:leray-lc}}^{\,\rotatebox{-90}{\makebox[-0.1cm]{\Iso}}} && 
\Hr^{d+e}_{\wid{L}}(\omgs{\wid{S}/A})
\ar[dd]^{\tin{\wid{S}/A}}\\
\Hr^d_{\wid{I}}(\omgs{\wid{R}/A}\otimes_{\wid{R}}\Hr^e_{\wid{J}}(\omgs{\wid{S}/\wid{R}})) \ar[d]_{\text{{\em{via}} $\tin{\wid{R}/A}$}} & & \\
\Hr^d_{\wid{I}}(\omgs{\wid{R}/A}) \ar[rr]_{\tin{\wid{R}/A}} && A
}}
\]
\end{enumerate}
\end{prop}

\proof  
Part\,(a) is mainly a re-statement of \Pref{prop:chi-affine}\,(b),  with Prop.\,\ref{prop:trans-lc} explaining
how \eqref{iso:leray-lc} enters into the picture.
 Before giving more details,
we make some observations. First, let $J_n$ denote the $S$-ideal
generated by $(v_1^n,\dots,v_e^n)$. Then~$S/J_n$ is finite and flat over~$R$, and hence
is Cohen-Macaulay of relative dimension~$0$ over~$R$. This means that the relative dualizing
module  for the algebra $R\to S/J_n$, i.e.,  $\omgs{S/R}\otimes_S\wedge^e_SJ_n/J_n^2$,
is flat over $R$, whence so its direct limit over $n$, namely $\Hr^e_J(\omgs{S/R})$. 
Since $f$, $g$ and $gf$ are Cohen-Macaulay of relative dimensions $e$, $d$, and $d+e$ respectively,
the maps $\phi_{S,J}(\omgs{S/R})$, $\phi_{R,I}(\omgs{R/A})$, and $\phi_{S,L}(\omgs{S/A})$  are
all isomorphisms. Moreover, since $\Hr^e_J(\omgs{S/R})$ is $R$-flat, the map
$\phi_{R,I}(\omgs{R/A}\otimes_R\Hr^e_J(\omgs{S/R}))$ is also an isomorphisms.

Since $\omgs{S/R}$ and $\Hr^e_J(\omgs{S/R})$ are both flat over $R$, 
we can apply \Pref{prop:trans-lc} with $M=\omgs{R/A}$,  $N=\omgs{S/R}$. 
Using the isomorphisms $\phi_{S,J}(\omgs{S/R})$, $\phi_{R,I}(\omgs{R/A})$,  $\phi_{S,L}(\omgs{S/A})$,  
$\phi_{R,I}(\omgs{R/A}\otimes_R\Hr^e_J(\omgs{S/R}))$, and applying \Pref{prop:trans-lc}, our 
assertion is equivalent to the commutativity of the diagram in part\,(b) of \Pref{prop:chi-affine}. 
This proves (a)

The proof of (b) is identical, with part(a) of \Pref{prop:chi-affine} replacing part (b) of loc.cit.
\qed

\bigskip

\Pref{prop:leray-chi} gives rise to two related residue formulas. 
The following is a consequence of part\,(a) of the proposition and the formula for the map 
\eqref{iso:leray-lc} given in \eqref{iso:leray-gf}.
For $\mu\in \omgs{R/A}$ and $\nu\in\omgs{S/R}$ and for integers
$\alpha_l >0$, $\beta_k >0$, $l\in\{1,\,\ldots,\,d\}$, $k\in\{1,\,\ldots,\,e\}$, we have
\stepcounter{thm}
\begin{equation*}\label{eq:res-res-1}\tag{\thethm}
\ares{{W_2}}\begin{bmatrix}
\ares{{W_1}}\begin{bmatrix} \nu \\ 
v_1^{\beta_1},\,\ldots,\,v_e^{\beta_e}
\end{bmatrix} \mu \\ 
u_1^{\alpha_1},\,\ldots,\,u_d^{\alpha_d}
\end{bmatrix}= \ares{W}\begin{bmatrix} \chi_{{}_{[S/R/A]}}(\mu\otimes\nu) \\ 
v_1^{\beta_1},\,\ldots,\,v_e^{\beta_e},\,u_1^{\alpha_1},\,\ldots,\,u_d^{\alpha_d}
\end{bmatrix}
\end{equation*}
Similarly, for $\mu\in \omgs{\wid{R}/A}$ and $\nu\in\omgs{\wid{S}/\wid{R}}$ and $\alpha_l$,
$\beta_k$ as above, we have by part\,(b) of the proposition, and the formula for the map 
\eqref{iso:leray-lc} given in \eqref{iso:leray-gf}, 
\stepcounter{thm}
\begin{equation*}\label{eq:res-res-2}\tag{\thethm}
\tin{\wid{R}/A}\begin{bmatrix}
\tin{\wid{S}/\wid{R}}\begin{bmatrix} \nu \\ 
v_1^{\beta_1},\,\ldots,\,v_e^{\beta_e}
\end{bmatrix} \mu \\ 
u_1^{\alpha_1},\,\ldots,\,u_d^{\alpha_d}
\end{bmatrix}= \tin{\wid{S}/A}\begin{bmatrix} \chi_{{}_{[\wid{S}/\wid{R}/A]}}(\mu\otimes\nu) \\ 
v_1^{\beta_1},\,\ldots,\,v_e^{\beta_e},\,u_1^{\alpha_1},\,\ldots,\,u_d^{\alpha_d}
\end{bmatrix}
\end{equation*}

\begin{rem}\label{rem:leray-pwr}
We will apply part\,(b) of the \Pref{prop:leray-chi} in a later paper in the following situation.
Let $R=A[u_1,\ldots,u_d]$, $S= R[v_1,\ldots,v_e]$ where ${\bf u}=(u_1,\ldots, u_d)$ and 
${\bf v}=(v_1,\ldots,v_e)$ are algebraically independent variables over $A$ and $R$ respectively.
Let $I={\bf u}R$, and $J={\bf v}S$. Then $\wid{R}=A[[{\bf u}]]$ and 
$\wid{S}=R[[{\bf v}]]=A[[{\bf u}, {\bf v}]]$.
\end{rem}

\appendix

\section{Base change and completions}
% \marginpar{Should say more about GM duality and give exact
%references. Source and target of the functors should be mentioned.} 
\subsection{}%{More on the flat-base-change isomorphism for $\ush{(-)}$}
\label{ss:more-base-ch}
We gather a few basic properties of the flat base-change map of
\eqref{iso:qc-basech}. By default, we work with complexes in $\wDqcp$.

Consider a cartesian square $\mathfrak s$ of noetherian formal schemes
\[
\xymatrix{
\V \ar[d]_{g}\ar[r]^v \ar@{}[dr]|\square& \X \ar[d]^{f} \\
\W \ar[r]_u & \Y
}
\]
with $f$ in $\bbG$ and $u$ flat so that we have a flat-base-change isomorphism
\[
\ush{\beta_{\mathfrak s}} \colon \bLambda_{\V}v^*\ush{f}\iso \ush{g}u^*
\]
as in \eqref{iso:qc-basech}. If $f$ (and hence $g$) is pseudoproper, then 
another description of $\ush{\beta_{\mathfrak s}}$ is that it is the map
adjoint to the composite of the following natural maps (cf. \cite[Theorem 8.1, p.~86]{dfs}).
\[
\R g_*\R\iGp{\V}\BL_{\V}v^*\ush{f} \iso \R g_*\R\iGp{\V}v^*\ush{f} 
\to \R g_*v^*\R\iGp{\X}\ush{f} \iso u^*\R f_*\R\iGp{\X}\ush{f} \xrightarrow{\Tr{f}} u^*
\]
If $f,g$ are formally \'{e}tale, then we have natural isomorphisms 
$\ush{f} \iso \BL_{\X}f^*$ and $\ush{g} \iso \BL_{\V}g^*$ induced by the corresponding
ones for $(-)^!$, $f^! \iso \R\iGp{\X}f^*$ and $g^! \iso \R\iGp{\V}g^*$ respectively. 
In this case, the base-change map $\beta^!_{\mathfrak s}$ for $(-)^!$ is induced by the composite
of the canonical isomorphisms
\[
\R\iGp{\V}v^*f^! \iso \R\iGp{\V}v^*\R\iGp{\X}f^* \iso \R\iGp{\V}v^*f^* \iso \R\iGp{\V}g^*u^* \iso g^!u^*.
\]
Hence another description of $\ush{\beta_{\mathfrak s}}$ is that it is given by the composite of the following 
isomorphisms
\[
\BL_{\V}v^*\ush{f} \iso \BL_{\V}v^*\BL_{\X}f^* \iso \BL_{\V}v^*f^* \iso \BL_{\V}g^*u^* \iso \ush{g}u^*.
\]
In particular, if $\eF \in \Dc^+(\Y)$, or if~$u$ is open or if~$\V$ is an ordinary scheme, 
then $\ush{\beta_{\mathfrak s}}(\eF)$ is given by the natural composite
\[
v^*\ush{f}\eF \iso v^*f^*\eF \iso g^*u^*\eF \iso \ush{g}u^*\eF.
\]

Next we look at transitivity properties of $\ush{\beta}$ vis-\'{a}-vis extension of the square~$\mathfrak s$
horizontally or vertically. These are also proved by reducing to the corresponding property 
for~$\beta^!$ (see \cite[Theorem\,2.3.2(i)]{pasting}).

\begin{prop}
\label{prop:basech-h-trans}
\hfill
\begin{enumerate}
\item
Consider cartesian squares $\mathfrak s_1$, $\mathfrak s_2$ as follows
\[
\xymatrix{
\V_2 \ar[r]^{v_2} \ar[d]_h^{\quad \;\;\mathfrak s_2} & \V_1 \ar[d]_{g}^{\quad\;\; \mathfrak s_1} \ar[r]^{v_1} 
& \X \ar[d]^{f} \\
\W_2 \ar[r]_{u_2} & \W_1 \ar[r]_{u_1} & \Y
}
\]
where $f,g,h$ are in $\bbG$ and $u_i,v_i$ are flat. Let $u = u_1u_2$ and $v= v_1v_2$ and
let $\mathfrak s$ denote the composite cartesian diagram. 
Then the following diagram of isomorphisms commutes.
\[
\xymatrix{
\BL_{\V_2}v_2^*v_1^*\ush{f} \ar[d] \ar[r] & \BL_{\V_2}v_2^*\BL_{\V_1}v_1^*\ush{f} 
\ar[r]^{\text{via}}_{\ush{\beta_{\mathfrak s_1}}} 
& \BL_{\V_2}v_2^*\ush{g}u_1^* \ar[r]^{\text{via}}_{\ush{\beta_{\mathfrak s_2}}}  & \ush{h}u_2^*u_1^* \ar[d] \\
\BL_{\V_2}v^*\ush{f} \ar[rrr]_{\text{via } \ush{\beta_{\mathfrak s}}}  &&& \ush{h}u^*
}
\]
\item Consider cartesian squares $\mathfrak s_1$, $\mathfrak s_2$ as follows
\[
\xymatrix{
\V_2 \ar[d]_{g_2}^{\quad \;\;\mathfrak s_2} \ar[r]^w & \X_2 \ar[d]^{f_2} \\
\V_1 \ar[d]_{g_1}^{\quad \;\;\mathfrak s_1} \ar[r]^v & \X_1 \ar[d]^{f_1} \\
\W \ar[r]_u & \Y
}
\]
where $f_i,g_i$ are in $\bbG$ and $u,v,w$ are flat. Let $f = f_1f_2$ and $g = g_1g_2$ and 
let $\mathfrak s$ denote the composite cartesian diagram. 
Then the following diagram of isomorphisms commutes.
\[
\xymatrix{
\BL_{\V_2}w^*\ush{f_2}\ush{f_1} \ar[d] \ar[r]^{\text{via}}_{\ush{\beta_{\mathfrak s_2}}}  
& \ush{g_2}v^*\ush{f_1} \ar[r] & \ush{g_2}\BL_{\V_1}v^*\ush{f_1} 
\ar[r]^{\text{via}}_{\ush{\beta_{\mathfrak s_1}}}  & \ush{g_2}\ush{g_1}u^* \ar[d] \\
\BL_{\V_2}w^*\ush{f} \ar[rrr]_{\text{via } \ush{\beta_{\mathfrak s}}} & & & \ush{g}u^*  
}
\]
\end{enumerate}
\end{prop}
\begin{proof}
(i). For convenience we shall consider the transposed version of the diagram in question.
Using the definitions $\ush{f} = \BL_{\X}f^!$,  $\ush{g} = \BL_{\V_1}g^!$, $\ush{h} = \BL_{\V_2}h^!$
and the isomorphisms in \eqref{iso*gamma-lambda}
we reduce to checking that the outer border of the following diagram of isomorphisms
commutes where to reduce clutter we have dropped the~$\R$'s.
\[
\xymatrix{
\BL_{\V_2}v_2^*v_1^*\BL_{\X}f^! \ar[rrr] \ar[d]  &&& \BL_{\V_2}v^*\BL_{\X}f^! \ar[d] \\
\BL_{\V_2}v_2^*v_1^*f^! \ar[d] \ar[r] & \BL_{\V_2}\iG{\V_2}v_2^*v_1^*f^! \ar[rr] \ar[d] 
&& \BL_{\V_2}\iG{\V_2}v^*f^!  \ar[dd]^{\text{via }\beta_{\mathfrak s}^!}  \\
\BL_{\V_2}v_2^*\BL_{\V_1}\iG{\V_1}v_1^*f^!  \ar[d]_{\text{via } \beta_{\mathfrak s_1}^!}  \ar[r]
& \BL_{\V_2}\iG{\V_2}v_2^*\iG{\V_1}v_1^*f^! \ar[d]_{\text{via } \beta_{\mathfrak s_1}^!} & \boxplus \\
\BL_{\V_2}v_2^*\BL_{\V_1}g^!u_1^* \ar[r] & \BL_{\V_2}\iG{\V_2}v_2^*g^!u_1^* 
\ar[r]_{\text{via } \beta_{\mathfrak s_2}^!} &  \BL_{\V_2}h^!u_2^*u_1^* \ar[r] & \BL_{\V_2}h^!u^* 
}
\]
The unlabelled arrows are obvious natural maps. %(see \eqref{iso*gamma-lambda}). 
The rectangle~$\boxplus$ commutes by the transitivity of base-change for $(-)^!$. Commutativity
of the remaining parts is obvious. 

(ii). Once again we consider the transpose of the diagram under consideration. Using the isomorphisms 
in~\eqref{iso*gamma-lambda} we reduce to checking that the outer border of the following diagram of
isomorphisms commutes.
\[
\xymatrix{
\BL_{\V_2}w^*\BL_{\X_2}f_2^!\BL_{\X_1}f_1^! \ar[d] \ar[r] 
& \BL_{\V_2}w^*\BL_{\X_2}f_2^!f_1^! \ar[r] \ar[d] & \BL_{\V_2}w^*\BL_{\X_2}f^! \ar[d] \\
\BL_{\V_2}\iG{\V_2}w^*f_2^!\BL_{\X_1}f_1^! \ar[d]_{\text{via } \beta^!_{\mathfrak s_2}} \ar[r] 
& \BL_{\V_2}\iG{\V_2}w^*f_2^!f_1^! \ar[r] \ar[d]_{\text{via } \beta^!_{\mathfrak s_2}} 
& \BL_{\V_2}\iG{\V_2}w^*f^! \ar[dd]^{\text{via } \beta^!_{\mathfrak s}}  \\
\BL_{\V_2}g_2^!v^*\BL_{\X_1}f_1^! \ar[d] \ar[r] & \BL_{\V_2}g_2^!v^*f_1^! \ar[d]^{\qquad\qquad\boxplus}  \\
\BL_{\V_2}g_2^!\BL_{\V_1}\iGp{\V_1}v^*f_1^!  \ar[d]_{\text{via } \beta^!_{\mathfrak s_1}} \ar[r]
& \BL_{\V_2}g_2^!\iGp{\V_1}v^*f_1^! \ar[d]_{\text{via } \beta^!_{\mathfrak s_1}} & \BL_{\V_2}g^!u^* \ar[d] \\
\BL_{\V_2}g_2^!\BL_{\V_1}g_1^!u^* \ar[r] & \BL_{\V_2}g_2^!g_1^!u^* \ar[r] & \BL_{\V_2}g^!u^* 
}
\]
The rectangle $\boxplus$ commutes by transitivity of base-change for $(-)^!$ while the other 
rectangles commute for obvious reasons. 
\end{proof}

\begin{comment} %%%%%%%%%%%%%
\subsection{} 
If $\X$ is a formal scheme $\I$ an open coherent ideal sheaf (i.e.,
an ideal sheaf containing an ideal of definition of $\X$), and $\kappa\colon \W\to\X$ the completion 
of~$\X$ by~$\I$, then according \cite{gm} we have an isomorphism
\[
\kappa^*\iso \ush{\kappa}
\]
which is adjoint to the composite 
\[\R\iGp{\W}\R\kappa_*\kappa^*\iso \R\iG{\I} \to {\bf 1},\]
or what amounts to the same thing, adjoint to the composite
\[\R\iGp{\W}\R\kappa_*\kappa^*\iso \R\iG{\I} \to \R\iGp{\X}.\]
\end{comment} 
%%%%%%%%%%%%%%

Completion maps, being pseudo-proper, formally \'etale, and flat, 
give rise to additional compatibility issues. 
Now we consider some special situations involving completion maps. 
%as these happen to be pseudoproper and formally \'etale (and hence flat).

Let $\X$ be a formal scheme and $\I \subset \co_\X$ an open coherent ideal. Let $\W \set \widehat{\X}$
be the completion of $\X$ along $\I$ and $\kappa \colon \W \to \X$ the corresponding completion
map. Then there are canonical isomorphisms 
(see proof of \cite[Lemma 4.1]{gm}, and of \cite[Proposition 5.2.4]{dfs}) 
\stepcounter{thm}
\begin{equation*}\label{eq:kappa-gamma}\tag{\thethm}
\kappa_*\R\iGp{\W}\kappa^* \iso \kappa_*\kappa^*\R\iG{\I} \iso \R\iG{\I}\kappa_*\kappa^* \osi \R\iG{\I}.
\end{equation*}

For the next two results regarding $\ush{(-)}$ for completion maps, we will first need to look at the corresponding 
results for $(-)^!$. For that purpose we recall that 
in~\cite{pasting}, $(-)^!$ is obtained by gluing the pseudofunctor
$(-)^{\times}$ over pseudoproper maps in~$\bbG$ given by 
\[
f^{\times} = \text{ right adjoint to } \R f_*, \qquad (\text{$f$ pseudoproper}),
\]
%$f^{\times} = $ right adjoint to $\R f_*$ (for $f$ pseudoproper) 
with the pseudofunctor $(-)^*_t$ over \'etale maps in $\bbG$ given by
\[
f \mapsto \R\iGp{\X}f_*, \qquad (\text{$f \colon \X \to \Y$ \'etale}),
\]
and this gluing utilizes, among other things, the \'etale base-change isomorphisms
associated to cartesian squares involving \'etale base change of a pseudoproper map 
(see \cite[Theorem 7.1.6, \S 7.2.7]{pasting}). 

\begin{lem}
\label{lem:kappa}
For a completion map $\kappa \colon \W \to \X$ by an open coherent ideal $\I \subset \co_\X$ as above
and for $\eF \in \Dc^+(\X)$,
the isomorphism $\kappa^*\eF \to \ush{\kappa}\eF$ of~\eqref{eq:gm} is also the map adjoint 
to the composite $\psi$ given by
\[
\kappa_*\R\iGp{\W}\kappa^* \iso \R\iG{\I} \to \mathbf{1}.
\]
\end{lem}
\begin{proof}[Sketch Proof]
It suffices to prove that the corresponding property for $\kappa^!$ holds, i.e.,  
the canonical isomorphism $\phi = \phi_{\kappa} \colon \R\iGp{\W}\kappa^* \iso \kappa^! $ 
is the map adjoint to~$\psi$.
Indeed, as per the proof of \cite[Theorem 7.1.6]{pasting},
the isomorphism $\phi$ equals $(\beta^!)^{-1}$ where 
$\beta^! \colon \bf{1}^*\kappa^! \iso \bf{1}^!\kappa^* = \R\iGp{\W}\kappa^*$ is the base-change 
isomorphism  associated to the cartesian square in the following diagram.
\[
\xymatrix{
\W \ar[r]^{\mathbf{1}} & \W \ar[r]^{\mathbf{1}} \ar[d]_{\mathbf{1}} & \W \ar[d] ^{\kappa} \\
& \W \ar[r]_{\kappa} & \X 
}
\]
Therefore, $\phi = \alpha_1\alpha_2^{-1}$ for $\alpha_i$ as given in the commutative diagram below, 
where $\alpha_1$ is the
canonical map $\kappa^*\kappa_* \to \mathbf{1}$, (which is an isomorphism over $\Dqct^+(\W)$)
while $\alpha_2$ results from the fact that the trace $\Tr{\kappa}^! \colon \kappa_*\kappa^! \to \bf{1}$ 
factors through $\R\iGp{\I} \to \mathbf{1}$. 
\[
\xymatrix{
\kappa^!  & \kappa^*\kappa_*\kappa^! 
\ar[l]_{\Iso}^{\alpha_1} \ar[r]^{\Iso}_{\alpha_2} \ar@/^1.5pc/[rr]^{\kappa^*\Tr{\kappa}^!}
& \R\iGp{\W}\kappa^* \ar[r] & \kappa^*
}
\]
The adjointness of $\phi$ and $\psi$ amounts to proving that $\Tr{\kappa}^!\kappa_*(\phi) = \psi$,
which results from the commutativity of the following.
\[
\xymatrix{
\kappa_*\kappa^!  \ar@{=}[rd] & \kappa_*\kappa^*\kappa_*\kappa^! 
\ar[l]_{\Iso}^{\kappa_*\alpha_1} \ar[rr]^{\Iso}_{\kappa_*\alpha_2} 
&& \kappa_*\R\iGp{\W}\kappa^* \ar[r] & \kappa_*\kappa^*  \\
& \kappa_*\kappa^! \ar[rr]_{\Tr{\kappa}^!} \ar[u]_{\bf{1} \to \kappa_*\kappa^*} 
&& \R\iG{\I} \ar[u]_{\cong} \ar[r] &  \bf{1} \ar[u]
}
\]
\end{proof}

\begin{lem}
\label{lem:bc-gm} 
Consider a cartesian diagram in $\bbG$ as follows
\[
\xymatrix{
\V \ar[d]_{g}\ar[r]^{\bar{\kappa}} \ar@{}[dr]|\square& \X \ar[d]^{f} \\
\W \ar[r]_{\kappa} & \Y
}
\]
%be a cartesian diagram in $\bbG$ 
where $\kappa, {\bar{\kappa}}$ are 
completion maps by open coherent ideal sheaves. Let $\eF \in \Dc^+(\Y)$.
\begin{enumerate}
%\item[(a)] If $f$ is a composite of compactifiable maps, then the base change map 
%$\bar\kappa^*\ush{f} \iso \ush{g}\kappa^*$ of \eqref{iso:bc-sharp} fits into the
%following commutative diagram
\item %For any $\eF \in \wDqcp(\Y)$, 
The following diagram of obvious natural isomorphisms commutes.
\[
\xymatrix{
\bar\kappa^*\ush{f}\eF \ar[r]^{\Iso}_{\eqref{iso:bc-sharp}} 
\ar[d]_{\eqref{eq:gm}}^{\,\rotatebox{-90}{\makebox[-0.1cm]{\Iso}}}
& \ush{g}\kappa^*\eF \ar[d]^{\eqref{eq:gm}}_{\,\rotatebox{-90}{\makebox[-0.1cm]{\Iso}}}\\
\ush{\bar\kappa}\ush{f}\eF \ar[r]^-{\Iso} %& \ush{(f\bar\kappa)}=\ush{(\kappa g)} \ar[r]^-{\Iso} 
& \ush{g}\ush{\kappa}\eF
}
\]
%\item[(b)] 
\item If $f$ is flat then the following diagram of obvious natural isomorphisms commutes.
\[
\xymatrix{
g^*\kappa^*\eF \ar[r]^{\Iso} \ar[d]_{\eqref{eq:gm}}^{\,\rotatebox{-90}{\makebox[-0.1cm]{\Iso}}}
& \bar\kappa^*f^*\eF \ar[d]^{\eqref{eq:gm}}_{\,\rotatebox{-90}{\makebox[-0.1cm]{\Iso}}}\\
g^*\ush{\kappa}\eF \ar[r]_-{\eqref{iso:bc-sharp}}^-{\Iso} & \ush{\bar\kappa}f^*\eF
}
\]
\end{enumerate}
\end{lem}
\begin{proof}[Sketch Proof]
As a consequence of the gluing result in \cite[Theorem 7.1.6]{pasting}, 
via the canonical isomorphisms $\phi_{\bar\kappa} \colon \bar\kappa^! \iso \R\iGp{\V}\bar\kappa^*$ and
$\phi_{\kappa} \colon \kappa^! \iso \R\iGp{\W}\kappa^* $, 
for the situation in (i), we have a commutative diagram of isomorphisms
\[
\xymatrix{
\R\iGp{\V}\bar\kappa^*f^! \ar[r]^{\beta^!} \ar[d] & g^!\R\iGp{\W}\kappa^* \ar[d] \\
\bar\kappa^!f^! \ar[r] & g^!\kappa^!
}
\]
reflecting the compatibility of $\beta$ with the pseudofunctorial structure of $(-)^!$,
while for the one in (ii), we have a commutative diagram of isomorphisms as follows,
\[
\xymatrix{
\R\iGp{\V}g^*\R\iGp{\W}\kappa^*  \ar[r] \ar[d] &  \R\iGp{\V}\bar\kappa^*f^* \ar[d] \\
\R\iGp{\V}g^*\kappa^!  \ar[r]_{\beta^!} & \bar\kappa^!f^*
}
\]
reflecting the compatibility of $\beta$ with the pseudofunctorial structure of $(-)^*_t$
over \'etale maps.
%In each case $\beta^!$ refers to the base-change isomorphism for $(-)^!$. %, see~\S\ref{subsec:base-ch}.
The result now follows by applying $\BL$'s appropriately in each diagram 
and using the pre-pseudofunctorial properties of $\ush{(-)}$.
\end{proof}

%\marginpar{\Lref{lem:bc-gm}'s proof is as in the scanned handwritten notes sent on Jul 10 
%with header ``[Scan] kappa".}

\subsection{} 
Let $f\colon \X\to \Y$ be a pseudo-proper map and let $\J$ be an ideal of definition of $\X$.
Suppose $\I$ is an open coherent
ideal in $\co_\Y$ and $\kappa\colon \U\to \Y$ is the completion of $\Y$
with respect to $\I$. Let $\V=\X\times_\Y\U$ and $\kappa'\colon \V\to \X$, $g\colon \V\to \U$ the
projection maps. Note that $\V$ is the completion of $\X$ with respect to the $\co_\X$-ideal 
$\I\co_\X+\J$, and $\kappa'$ is the completion map. We thus have a cartesian square:
\[
{\xymatrix{
\V \ar[d]_{f'} \ar[r]^{\kappa'} \ar@{}[dr]|\square & \X \ar[d]^{f}\\
\U \ar[r]_\kappa  & \Y
}}
\]
By \eqref{eq:kappa-gamma}, we have $\kappa_*\R\iGp{\U}\kappa^*\iso \R\iG{\I}$ and 
${\kappa'}_*\R\iGp{\V}{\kappa'}^*\iso \R\iG{\I\co_\X+\J}$.

\begin{prop}\label{prop:iterated-trace} The following diagram commutes:
\[
{\xymatrix{
\Rfs\kappa'_*\R\iGp{\V}{\kappa'}^*\ush{f} \ar[rr]^{\Iso}_{\gamma} 
\ar[d]_{\,\rotatebox{-90}{\makebox[-0.1cm]{\Iso}}}&& 
\kappa_*\R f'_*\R\iGp{\V}\ush{(f')}\kappa^* \ar[dd]^{\kappa_*\Tr{{f'}}}\\
\Rfs\R\iG{\I\co_\X+\J}\ush{f} \ar[d]_{\,\rotatebox{-90}{\makebox[-0.1cm]{\Iso}}} && \\
\Rfs\R\iG{\I\co_\X}\R\iGp{\X}\ush{f}  \ar[d]_{\,\rotatebox{-90}{\makebox[-0.1cm]{\Iso}}} && 
\kappa_*\R\iGp{\U}\kappa^* \ar[d]_{\,\rotatebox{-90}{\makebox[-0.1cm]{\Iso}}}  \\
\R\iG{\I}\Rfs\R\iGp{\X}\ush{f} %\ar[u]^{\,\rotatebox{-90}{\makebox[-0.1cm]{\Iso}}}
\ar[rr]_{\Tr{f}}  && \R\iG{\I}
}}
\]
%where the isomorphism 
%$\Rfs\kappa'_*\R\iGp{\V}{\kappa'}^*\ush{f}\iso \kappa_*\R f'_*\R\iGp{\V}\ush{(f')}\kappa^*$
%depicted by the horizontal arrow on the top row 
where $\gamma$ is induced by applying the functor
$\Rfs\kappa'_*\R\iGp{\V}{\kappa'}^*$ to the natural isomorphism
${\kappa'}^*\ush{f} \iso \ush{f'}\kappa^*$ and the upward pointing arrow on the southwest
corner is the isomorphism of \cite[Proposition 5.2.8 (d)]{dfs}.
\end{prop}
\proof
Consider the diagram in the proposition. 
Let $\alpha\colon \Rfs\kappa'_*\R\iGp{\V}{\kappa'}^*\ush{f} \to \R\iG{\I}$ be the map obtained
by composing maps along the route in the diagram which starts at the northwest corner, travelling
south and then east. Let $\beta\colon \Rfs\kappa'_*\R\iGp{\V}{\kappa'}^*\ush{f} \to \R\iG{\I}$ be 
the composition which starts in the easterly direction and then moves south. Let 
$\psi\colon \R\iG{\I}\to {\bf 1}_{\D(\Y)}$ be the natural map. We have to show that $\alpha=\beta$. 
This is equivalent to showing
\stepcounter{thm}
\begin{equation*}\label{eq:alpha-beta-psi}\tag{\thethm}
\psi\smcirc\alpha=\psi\smcirc\beta .
\end{equation*}
We now proceed to prove \eqref{eq:alpha-beta-psi}.
In what follows we identify ${\kappa'}^*$ with $\ush{\kappa'}$ and $\kappa^*$ with $\ush{\kappa}$.
Recall that the isomorphism ${\kappa'}^*\colon\ush{f} \iso \ush{f'}\kappa^*$ mentioned in the theorem
can be interpreted in two ways, and the
two interpretations agree: (a) as a base change isomorphism,  and (b) as the composite 
\stepcounter{thm}
\begin{equation*}\label{diag:trans-kappa}\tag{\thethm}
{\kappa'}^*\ush{f} = \ush{\kappa'}\ush{f} \iso \ush{(f\kappa')} = \ush{(\kappa f')}\iso \ush{f'}
\ush{\kappa} = \ush{f'}\kappa^*.
\end{equation*}
We point out the trace map $\Tr{\kappa}\colon \kappa_*\R\iGp{\U}\kappa^* \to {\bf}_{\D(\Y)}$ 
under the identification $\kappa^*=\ush{\kappa}$ 
is the composite $\kappa_*\R\iGp{\U}\kappa^* \iso \R\iG{\I} \to {\bf 1}_{\D(\Y)}$.
Similarly, $\Tr{\kappa'}\colon \kappa'_*\R\iGp{\V}{\kappa'}^* \to\ R\iGp{\X}$ is the
composite $\kappa'_*\R\iGp{\V}{\kappa'}^* \iso \R\iG{\I\co_\X+\J} \to \Rfs\R\iG{\J} = \R\iGp{\X}$.

From the definition of the isomorphism in \eqref{diag:trans-kappa} it follows that the following diagram 
commutes:
\[
{\xymatrix{
\Rfs\kappa'_*\R\iGp{\V}{\kappa'}^*\ush{f} \ar[rr]^{\Iso}_{\text{via \eqref{diag:trans-kappa}}}
\ar[d]_{\,\rotatebox{-90}{\makebox[-0.1cm]{\Iso}}}  && 
\kappa_*\R f'_*\R\iGp{\V}\ush{(f')}\kappa^* \ar[d]^{\kappa_*\Tr{f'}} \\
\Rfs\R\iG{\I\co_\X+\J}\ush{f} \ar[d] &&  \kappa_*\R\iGp{\U}\kappa^* 
\ar[d]^{\,\rotatebox{-90}{\makebox[-0.1cm]{\Iso}}}\\
\Rfs\R\iG{\J}\ush{f} \ar@{=}[d] &&  \R\iG{\I} \ar[d]^{\psi} \\
\Rfs\R\iGp{\X}\ush{f} \ar[rr]_{\Tr{f}} && {\bf 1}_{\D(\Y)}
}}
\]
Let $\theta\colon \Rfs\kappa'_*\R\iGp{\V}{\kappa'}^*\ush{f} \to {\bf 1}_{\D(\Y)}$ be the map
obtained from taking any route from the top left corner to the bottom right corner in the above
commutative diagram. Note that $\theta= \psi\smcirc\beta$.
It is therefore enough to show that $\theta=\psi\smcirc\alpha$. Consider the following
diagram where the arrow in the top row and the second map in the second row 
arise from the natural maps $\R\iG{\I\co_\X+\J} \to \R\iG{\J}$ and
$\R\iG{\I\co_\X} \to {\bf 1}_{\D(\X)}$ respectively:
\[
{\xymatrix{
& \Rfs\R\iG{\I\co_\X+\J}\ush{f} \ar[d]_{\,\rotatebox{-90}{\makebox[-0.1cm]{\Iso}}} \ar[r] & 
\Rfs \R\iG{\I}\ush{f} \ar@{=}[d]\\
\R\iG{\I}\Rfs\R\iGp{\X}\ush{f} \ar[r]^-{\Iso} \ar[d]_{\R\iG{\I}(\Tr{f})}& \Rfs\R\iG{\I\co_\X}\R\iGp{\X}\ush{f}
\ar[r] & \Rfs\R\iGp{\X}\ush{f} \ar[d]^{\Tr{f}} \\
\R\iG{\I} \ar[rr]_\psi && {\bf 1}_{\D(\Y)}
}}
\]
We claim this diagram commutes. The sub-rectangle on the top clearly commutes. 
According to \cite[Proposition 5.2.8 (d)]{dfs}, the composite of the two arrows in the second
row is the natural map arising from $\psi\colon \R\iG{\I}\to {\bf 1}_{\D(\Y)}$. 
It follows that the rectangle
at the bottom also commutes, whence the whole diagram commutes. This proves that
$\theta=\psi\smcirc\alpha$. Thus $\psi\smcirc\alpha = \theta =\psi\smcirc\beta$, establishing
\eqref{eq:alpha-beta-psi}.
\qed

\subsection{} Suppose $f\colon X\to Y$ is a map of ordinary schemes in $\bbG$
and $Z\hookrightarrow X$ is a
closed subscheme such that $Z\to Y$ is {\em proper}. Let $\kappa\colon \X=X_{/Z}\to X$ be
the formal completion of $X$ along $Z$ and $\wid{f}\colon \X\to Y$ the composition 
$\wid{f}=f\smcirc\kappa$. Then $\wid{f}$ is {\em pseudo-proper}. The isomorphism 
$\kappa^*\iso \ush{\kappa}$ of \eqref{eq:gm} gives us an isomorphism 
$\kappa^*\ush{f} \iso \ush{\wid{f}}$, and hence an isomorphism 
$\alpha\colon\Rfs \kappa_*\R\iGp{\X}\kappa^*\ush{f} \iso \R{\wid{f}}_*\R\iGp{\X}\ush{\wid{f}}$. On the
other hand we have $\beta\colon  \Rfs \kappa_*\R\iGp{\X}\kappa^*\ush{f}  \iso \Rfs \R\iG{Z}\ush{f}$
induced by \eqref{eq:kappa-gamma}.
%\marginpar{Need to give appropriate reference for $\beta$ from \cite{gm} and \cite{dfs}.}
We thus have an isomorphism
\stepcounter{thm}
\begin{equation*}\label{iso:iGZ-iGpX}\tag{\thethm}
\alpha\smcirc\beta^{-1}\colon \Rfs \R\iG{Z}\ush{f} \iso \R{\wid{f}}_*\R\iGp{\X}\ush{\wid{f}}.
\end{equation*}
If $u\colon X\to X'$ is an open immersion of finite type $Y$-schemes, with
$g\colon X'\to Y$ the structure map, then the natural isomorphism 
\[\Rfs\R\iG{Z}\ush{f}\iso \Rfs\R\iG{Z}u^*\ush{g} = \R g_*\R\iG{u(Z)}\ush{g}\]
fits into a commutative diagram
\stepcounter{thm}
\[
\begin{aligned}\label{diag:iGZ-iGpX}
\xymatrix{
\Rfs\R\iG{Z}\ush{f} \ar[dr]^{\eqref{iso:iGZ-iGpX}} 
\ar[d]_{\,\rotatebox{-90}{\makebox[-0.1cm]{\Iso}}} & \\
\R g_*\R\iG{u(Z)}\ush{g}  \ar[r]_{\eqref{iso:iGZ-iGpX}}&\R{\wid{f}}_*\R\iGp{\X}\ush{\wid{f}}
}
\end{aligned}\tag{\thethm}
\]

If $f$ is {\em proper}, then the isomorphism $\ush{\kappa}\ush{f} \iso \ush{\wid{f}}$
is the one adjoint to the composite 
\[
\Rfs\kappa_*\R\iGp\X\ush{\kappa}\ush{f} \xrightarrow{\Rfs(\Tr{\kappa})} \Rfs\ush{f} \to {\bf 1}
\]
%Together with the isomorphism \eqref{eq:gm} this says that 
and so the isomorphism $\kappa^*\!\ush{f}\iso \ush{\wid{f}}$ is characterised by
the commutativity of the following diagram.
\stepcounter{thm}
\[
\begin{aligned}\label{diag:alpha-beta}
\xymatrix{
\Rfs\R\iG{Z}\ush{f} \ar[d]_{\text{natural}} & \Rfs\kappa_*\R\iGp{\X}\kappa^*\ush{f} \ar[l]_-{\Iso}^-{\beta}
\ar[r]^-{\Iso}_-{\alpha} & \R{\wid{f}}_*\R\iGp{\X}\ush{\wid{f}} \ar[d]^{\Tr{\wid{f}}} \\
\Rfs\ush{f} \ar[rr]_{\Tr{f}} && {\bf 1}
}
\end{aligned}\tag{\thethm}
\]

In general, when $f$ is not necessarily proper, it is still separated (being in
$\bbG$) and hence we do have a {\em compactification} of $f$, i.e.,
an open immersion  of $Y$-schemes $u\colon X\to {\bar{X}}$, such that
the structure map ${\bar{f}}\colon {\bar{X}}\to Y$ is {\em proper}. We have a
commutative diagram:
\stepcounter{thm}
\[
\begin{aligned}\label{diag:nonproper}
\xymatrix{
\X \ar[r]^{\kappa}  \ar[dr]_{\wid{f}} & X   \ar[d]_f \ar@{^(->}[r]^u & {\bar X} 
\ar[dl]^{\bar{f}}\\
& Y &
}
\end{aligned}\tag{\thethm}
\]
We then have the following lemma.

\begin{lem}
\label{lem:indep1}  
Under the assumptions and notation of \eqref{diag:nonproper}, the following
diagram commutes:
\[
\xymatrix{
\Rfs\R\iG{Z}\ush{f} \ar[rr]^{\Iso}_{\eqref{iso:iGZ-iGpX}} 
\ar[d]_-{\,\rotatebox{-90}{\makebox[-0.1cm]{\Iso}}}
&& \R{\wid{f}}_*\iGp{\X}\ush{\wid{f}} \ar[d]^{\Tr{\wid{f}}}  \\
 \R{\bar{f}}_*\R\iG{u(Z)}\ush{\bar{f}} \ar[r] & \R{\bar{f}}_*\ush{\bar{f}}
 \ar[r]_{\Tr{\bar{f}}} & {\bf 1}
}
\]
In particular, the composite 
\[
\Rfs\R\iG{Z}\ush{f} \iso \R{\bar{f}}_*\R\iG{u(Z)}\ush{\bar{f}} \to 
\R{\bar{f}}_*\ush{\bar{f}} \xrightarrow{\Tr{\bar{f}}} {\bf 1} 
\]
is independent of the compactification $(u,\,\bar{f})$ of $f$.
\end{lem} 

\proof We expand the diagram to
\[
\xymatrix{
\Rfs\R\iG{Z}\ush{f} \ar[rr]^{\Iso}_{\eqref{iso:iGZ-iGpX}} 
\ar[d]_-{\,\rotatebox{-90}{\makebox[-0.1cm]{\Iso}}}
&& \R{\wid{f}}_*\iGp{\X}\ush{\wid{f}} \ar[rr]^{\Tr{\wid{f}}}  && {\bf 1}\\
 \R{\bar{f}}_*\R\iG{u(Z)}\ush{\bar{f}} \ar[rr] 
 \ar[urr]_-{\eqref{iso:iGZ-iGpX}}^{\,\rotatebox{25}{\makebox[-0.1cm]{\Iso}}}
 && \R{\bar{f}}_*\ush{\bar{f}}
 \ar[rru]_{\Tr{\bar{f}}} &&
}
\]
The triangle on the left commutes by \eqref{diag:iGZ-iGpX}. The parallelogram is
simply \eqref{diag:alpha-beta}, for $\alpha\smcirc\beta^{-1}=\eqref{iso:iGZ-iGpX}$.
\qed

\section{Closed immersions and completions}\label{ss:flat-sharp}

\subsection{}
\label{subsec:closed-imm}
Let $i \colon \Z \to \X$ be a closed immersion of noetherian formal schemes. 
We use~$\bar{i}$ to denote the flat map of ringed spaces $(\Z, \co_{\Z}) \to (\X, i_*\co_{\Z})$.
%Then $i_*$ and $\bar{i}^*$ are exact functors and are inverse to each other. 
We define the functor $i^{\flat} \colon \D(\X) \to \D(\Z)$ by
\[
i^{\flat} \set {\bar i}^*\R\sHomb_{\X}(i_*\co_{\Z},\,\boldsymbol{-}).
\]
The functor $i^{\flat}$ enjoys the following properties (see \cite[Examples 6.1.3(4)]{dfs}).

1) $i^{\flat}(\Dqc^+(\X)) \subset \Dqc^+(\Z)$ and $i^{\flat}(\Dc^+(\X)) \subset \Dc^+(\Z)$.
This follows from the fact that $i_*\co_{\Z}$ is coherent $\co_\X$-module.

2) There is a natural isomorphism $i^{\flat}\R\iGp{\X} \iso \R\iG{\Z}i^{\flat}$ whose composition 
with the natural map $\R\iG{\Z}i^{\flat} \to i^{\flat}$ is the natural map $i^{\flat}\R\iGp{\X} \to i^{\flat}$.

%It follows from the definition that the canonical map $i^{\flat}\R\iGp{\X} \to i^{\flat}$ factors through the
%map $\R\iG{\Z}i^{\flat} \to i^{\flat}$ and the natural map is an isomorphism $i^{\flat}\R\iGp{\X} \iso \R\iG{\Z}i^{\flat}$. 
%Moreover $i^{\flat}(\Dqc^+(\X)) \subset \Dqc^+(\Z)$ and $i^{\flat}(\Dc^+(\X)) \subset \Dc^+(\Z)$.
3) Using 2) we also obtain that $i^{\flat}(\Dqct^+(\X)) \subset \Dqct^+(\Z)$. Hence we also deduce that
$i^{\flat}(\wDqcp(\X)) \subset \wDqcp(\Z)$.

4) There is a canonical trace map on $\D(\X)$, namely
\stepcounter{thm}
\begin{equation*}\label{def:Tr-i-flat}\tag{\thethm}
\Tr{i}^\flat\colon i_*i^\flat = \R\sHomb_{\X}(i_*\co_\Z,\,\boldsymbol{-}) \lra {\bf 1},
\end{equation*}
which is given by ``evaluation at 1'', and which induces a natural map of functors from 
$i^{\flat} \colon \wDqcp(\X) \to \wDqcp(\Z)$ to the right adjoint $i^{\times}$ of $i_* \colon \wDqcp(\Z) \to \wDqcp(\X)$.
Moreover, this induced map $i^{\flat} \to i^{\times}$ is an isomorphism. 
Keeping in mind that the values of $(-)^!$ range in $\Dqct^+$, 
we deduce that for any $\eF \in \wDqcp(\X)$, there is a natural isomorphism
\[
i^{\flat}\R\iGp{\X}\eF \iso \R\iG{\Z}i^{\flat}\eF \xrightarrow[\text{via }\Tr{i}]{\Iso} i^!\eF
\]
and hence for $\eF \in \wDqcp(\X)$, there is also a natural isomorphism
\[
\BL_{\Z}i^{\flat}\eF \iso \ush{i}\eF
\]
where the corresponding trace map $\Tr{i}$ is the natural composite
\[
%i_*\R\iGp{\Z}\BL_{\Z}i^{\flat} \iso i_*\R\iGp{\Z}i^{\flat} \iso  i_*i^{\flat}\R\iGp{\X} \xrightarrow{\text{via } \Tr{i}^\flat} 
%\R\iGp{\X} \to 1.
i_*\R\iGp{\Z}\BL_{\Z}i^{\flat} \iso i_*\R\iGp{\Z}i^{\flat} \to  i_*i^{\flat} \xrightarrow{\Tr{i}^\flat}  1.
\]
In particular, if $\eF \in \Dc^+(\X)$, or if $\Z$ is an ordinary scheme, then we have a canonical 
isomorphism 
\stepcounter{thm}
\begin{equation*}
\label{iso:flat-sharp}\tag{\thethm}
i^{\flat}\eF \iso \ush{i}\eF.
\end{equation*}

\subsection{} 
\label{subsec:i-j-k}
Suppose $X$ is an ordinary scheme, $\I$ a 
coherent ideal sheaf on $X$, $Z$ the closed subscheme of $X$ defined by $\I$, and
$\kappa\colon \X=X_{/Z} \to X$ the completion of $X$ along $Z$. We then have a 
commutative diagram with $i$ and $j$ closed immersions:
\[
\xymatrix{
Z\phantom{X} \ar@{^(->}[r]^{j} \ar@{_(->}[dr]_{i} & \X \ar[d]^{\kappa} \\
& X
}
\]
We define $\bar{i}$ and ${\bar j}$ as in \ref{subsec:closed-imm} above, and 
it follows that if $\eF$ is a $j_*\co_Z$-module, then ${\bar{i}}^*\kappa_*\eF={\bar j}^*\eF$. 
We also define $i^\flat, j^\flat$ as in~\ref{subsec:closed-imm} and in what follows we will drop the symbols 
$i_*, j_*$ occurring in the definition of $i^\flat, j^\flat$ respectively.
Finally note that, since~$Z$ is an ordinary scheme so that~$\iGp{Z}$ is the 
identity functor, $\ush{i}$ and~$\ush{j}$ are right adjoint to~$i_*$ and~$j_*$ respectively.

\begin{comment} %%%%%%%%%%%
The operations $i_*$ and ${\bar i}^*$ are inverses, as are the operations $j_*$ and ${\bar j}^*$.
It follows that if $\eF$ is a $j_*\co_Z$-module, then ${\bar{i}}^*\kappa_*\eF={\bar j}^*\eF$. 

Define the functors $i^\flat \colon \Dqc^+(X) \to \Dqc^+(Z)$ and $j^\flat \colon \wDqcp(\X) \to \Dqc^+(Z)$ 
as
\[i^\flat \set {\bar i}^*\R\sHomb_X(i_*\co_Z,\,\boldsymbol{-})\]
and
\[j^\flat \set {\bar j}^*\R\sHomb_\X(j_*\co_Z,\,\boldsymbol{-}).\]

In what follows we will drop the symbol $i_*$ occurring in the definition of $i^\flat$ and similarly drop
$j_*$ from the formula for $j^\flat$. Also since $Z$ is an ordinary scheme, $\ush{i}$ and $\ush{j}$
are right adjoint to $i_*$ and $j_*$ respectively.

There are canonical isomorphisms
\stepcounter{thm}
\begin{gather*}\label{iso:flat-sharp}\tag{\thethm}
i^\flat \iso \ush{i},  \\
j^\flat \iso \ush{j}.
\end{gather*}
In greater detail, we have maps 
\[\Tr{i}^\flat\colon i_*i^\flat = \R\sHomb_X(\co_Z,\,\boldsymbol{-}) \lra {\bf 1}\] 
and 
\[\Tr{j}^\flat\colon j_*j^\flat = \R\sHomb_\X(\co_Z,\,\boldsymbol{-}) \lra {\bf 1}\] 
which are given by ``evaluation at 1", and
the resulting maps $i^\flat\to \ush{i}$ and $j^\flat\to \ush{j}$ are isomorphisms. Note that
$\Tr{j}^\flat$ factors as
\[\R\sHomb_\X(\co_Z,\,\boldsymbol{-}) \lra \R\iGp{\X} \lra {\bf 1}\]
whereas $\Tr{i}^\flat$ factors as
\[\R\sHomb_X(\co_Z,\,\boldsymbol{-}) \lra \R\iG{Z} \lra {\bf 1}.\]

\end{comment} 
%%%%%%%%%%%%%%%%

The natural map
\stepcounter{thm}
\begin{equation*}\label{map:ij-flat}\tag{\thethm}
\R\sHomb_X(\co_Z,\,\boldsymbol{-}) \lra \kappa_*\R\sHomb_\X(\co_Z,\,\kappa^*\boldsymbol{-})
\end{equation*}
is an isomorphism, whence we have an isomorphism
\stepcounter{thm}
\begin{equation*}\label{iso:ij-flat}\tag{\thethm}
i^\flat \iso j^\flat\kappa^*
\end{equation*}
given by
\[
{\bar i}^*\R\sHomb_X(\co_Z,\,\boldsymbol{-}) \xrightarrow[{\bar i}^*\eqref{map:ij-flat}]{\Iso}
{\bar i}^*\kappa_*\R\sHomb_\X(\co_Z,\,\kappa^*\boldsymbol{-})=
{\bar j}^*\R\sHomb_\X(\co_Z,\,\kappa^*\boldsymbol{-}).
\]

The essential content of the following lemma is that, \eqref{iso:ij-flat} is, up to canonical identifications, 
the inverse of the canonical isomorphism 
$\ush{j}\ush{\kappa}\iso \ush{(\kappa j)}=\ush{i}$.

\begin{lem}\label{lem:ij-kappa} The following diagram commutes
\[
\xymatrix{
j^\flat\kappa^* \ar[r]^{\Iso}_{\eqref{iso:flat-sharp}} & \ush{j}\kappa^* \ar[r]^{\Iso}_{\eqref{eq:gm}}
& \ush{j}\ush{\kappa} 
\ar[d]^{\,\rotatebox{-90}{\makebox[-0.1cm]{\Iso}}}\\
i^\flat \ar[u]^{\eqref{iso:ij-flat}}_{\,\rotatebox{90}{\makebox[-0.1cm]{\Iso}}} 
\ar[rr]^{\Iso}_{\eqref{iso:flat-sharp}}& & \ush{i} 
}
\]
where the unlabelled isomorphism $\ush{j}\ush{\kappa}\iso \ush{i}$ is the canonical one. 
\end{lem}
\begin{proof}
Keeping in mind that the 
canonical maps $j_*j^{\flat} \to 1$ and $j_*\ush{j} \to 1$ factor through
$\R\iGp{\X} \to 1$ and that the canonical map $i_*i^{\flat} \to 1$ factors through $\R\iG{Z} \to 1$
we see that the diagram of the lemma corresponds, via 
adjointness of $\ush{i}$ to $i_*$, to the outer border of the following commutative diagram
of obvious natural maps.
\[
\xymatrix{
i_*j^{\flat}\kappa^* \ar[rd] \ar[rr] && i_*\ush{j}\kappa^* \ar[d] \ar[r] & i_*\ush{j}\ush{\kappa} \ar[d] \\
& \kappa_*j_*j^{\flat}\kappa^* \ar@{=}[d]  \ar[r]   & \kappa_*j_*\ush{j}\kappa^* \ar[d] \ar[r] 
& \kappa_*j_*\ush{j}\ush{\kappa} \ar[d] \\
& \kappa_*\R\sHomb_\X(\co_Z,\,\kappa^*\(\boldsymbol{-})) \ar[d] \ar[r] 
& \kappa_*\R\iGp{\X}\kappa^* \ar[d] \ar[r] & \kappa_*\R\iGp{\X}\ush{\kappa} \ar[d] \\
i_*i^{\flat} \ar[uuu] \ar@{=}[r] & \R\sHomb_X(\co_Z,\,\boldsymbol{-}) \ar[r] & \R\iG{Z} \ar[r] & 1
}
\]
\end{proof}

\begin{comment} %%%%%%%%%%%%%%%
First recall that the isomorphism $\ush{j}\ush{\kappa}\iso\ush{i}$ is the adjoint to the
composite 
\[i_*\ush{j}\ush{\kappa} = \kappa_*j_*\ush{j}\ush{\kappa} \xrightarrow{\Tr{j}}
\kappa_*\R\iGp{\X}\ush{\kappa} \xrightarrow{\Tr{\kappa}} {\bf 1}\]
and that the map $\Tr{\kappa}$ can be further de-constructed via the commutative diagram.
\[\xymatrix{
\kappa_*\R\iGp{\X}\ush{\kappa} \ar[d]_{\Tr{\kappa}} 
& \kappa_*\R\iGp{\X}\kappa^* \ar[l]_{\Iso}^{\eqref{eq:gm}}  \\
{\bf 1} & \R\iG{Z} \ar[l]^{\text{canonical}}  \ar[u]^{\,\rotatebox{90}{\makebox[-0.1cm]{\Iso}}} 
}\]
The maps $\Tr{j}^\flat\colon\R\sHomb_\X(\co_Z,\,\boldsymbol{-}) \to{\bf 1}$
and $\Tr{i}^\flat\colon\R\sHomb_X(\co_Z,\,\boldsymbol{-}) \to {\bf 1}$
factor respectively as
\[
\R\sHomb_\X(\co_Z,\,\boldsymbol{-}) \to \R\iGp{\X} \to {\bf 1}
\qquad \text{and} \qquad
\R\sHomb_X(\co_Z,\,\boldsymbol{-}) \to \R\iG{Z} \to {\bf 1}
\]
where the first arrow in each case is the natural one. 

The result follows from the commutativity of the following diagram, and the facts listed above:
\[
\xymatrix{
 \kappa_*j_*j^\flat\kappa^* \ar@{=}[r]
& \kappa_*\R\sHomb_\X(\co_Z,\,\kappa^*\(\boldsymbol{-})) \ar[r]
& \kappa_*\R\iGp{\X}\kappa^* \\
i_*j^\flat\kappa^* \ar@{=}[u] & & \\
i_*i^\flat \ar[u]^{i_*\eqref{iso:ij-flat}}_{\,\rotatebox{90}{\makebox[-0.1cm]{\Iso}}}  \ar@{=}[r]
& \R\sHomb_X(\co_Z,\,\boldsymbol{-}) 
\ar[uu]_{\,\rotatebox{90}{\makebox[-0.1cm]{\Iso}}}^{\eqref{map:ij-flat}} \ar[r] 
& \R\iG{Z} \ar[uu]_{\,\rotatebox{90}{\makebox[-0.1cm]{\Iso}}}
}
\]
\end{comment} 
%%%%%%%%%%%%%%%%%%%%%

\section{Koszul complexes}\label{s:kosz}
Since our goal is to understand Verdier's 
isomorphism explicitly, we have to lay out our conventions for maps between complexes,
especially the fundamental local isomorphism which is at the heart of explicit
formulas for residues, and hence integrals (i.e., traces). 

\subsection{Our version of Koszul complexes}\label{ss:kosz}
Let $R$ be a noetherian ring. For $t\in R$, we write
$K_\bullet(t)$ for the homology complex 
\[ 0 \lra K_1(t) \lra K_0(t) \lra 0\]
where $K_1(t)=K_0(t)=R$ and the arrow between them is multiplication by $t$.
For a sequence of elements ${\bf t} = (t_1, \dots, t_r)$ in $R$, we set $K_\bullet({\bf t})$
to be the complex:
\[K_\bullet({\bf t}) = K_\bullet(t_1)\otimes_R \dots\otimes_R K_\bullet(t_r).\]
For an $R$-module $M$ and an integer $i$, we write $K^i({\bf t},\,M) = \Hom_R(K_i({\bf t},\,M)$
and define $\partial^i\colon K^i({\bf t},\,M) \to K^{i+1}({\bf t},\,M)$ to be the transpose
of the differential $K_{i+1}({\bf t}) \to K_i({\bf t})$, without the intervention of
any signs. Then $K^\bullet({\bf t},\,M)$ together with
$\partial^\bullet$ is a cohomology complex, and this is what we will call the 
{\emph{Koszul (cohomology) complex}} on $M$ and ${\bf t}$. We write $K^\bullet({\bf t})$
for $K^\bullet({\bf t},\,R)$. We refer the reader to \cite[pp.\,17--18]{conrad} for a discussion
of various versions of Koszul complexes and the relationship between them. Here are three basic
properties:

1) $K^\bullet({\bf t}, M)$ is bounded by degrees $0$ and $r$, with
%at the extreme ends (i.e., at degrees $0$ and $r$) we have 
$K^0({\bf t},\,M) = K^r({\bf t}, M) =M$.

2) $K^\bullet({\bf t}, M) = M\otimes_R K^\bullet({\bf t})$.

3) $K^i({\bf t}, M)$ is the direct sum of $\binom{n}{i}$ copies of $M$.

The reason we use this version is the relationship with a certain  \v{C}ech complex
associated to an affine open cover of $\Spec{\,R}\smallsetminus Z$, where $Z$ is the
closed subscheme defined by the vanishing of the $t_i$'s (see \Ssref{ss:stabkoz}). 
The homology complex $K_\bullet({\bf t})$ is also called a Koszul complex, and to
distinguish it from $K^\bullet({\bf t})$, we will call it the {\emph{Koszul homology complex}}
on ${\bf t}$.

There is a well known way in which these Koszul complexes vary with
respect to ${\bf t}$. Let $I$ be the ideal generated by ${\bf t}$. Let $J$ be an ideal in $R$
such that $I\subset J$, and such that $J$ is generated by ${\bf g}=(g_1,\dots, g_r)$.
Since $I\subset J$ we have $u_{ij}\in R$ such that 
\[t_i=\sum_{j=1}^r u_{ij}g_j \qquad (i=1,\dots, r).\]
As is well-known, one has a map of homology Koszul complexes
\[U_\bullet\colon K_\bullet({\bf t}) \lra K_\bullet({\bf g})\]
such that
\begin{itemize}
\item[--] $H_0(U_\bullet)\colon R/I \to R/J$ is the natural surjection.
\item[--] $R=K_0({\bf t}) \xrightarrow{U_0} K_0({\bf g})=R$ is the identity map on $R$.
\item[--] $R=K_n({\bf t}) \xrightarrow{U_n} K_n({\bf g})=R$ is the map $x\mapsto \det(u_{ij})\cdot x$.
\end{itemize}
Taking transposes and tensoring with $M$ we get a map on (cohomology) Koszul complexes:
\stepcounter{thm}
\begin{equation*}\label{map:KU}\tag{\thethm}
U^\bullet=U^\bullet_M\colon K^\bullet({\bf g},\,M) \lra K^\bullet({\bf t},\,M)
\end{equation*}
such that $U^0$ is the identity map on $M$ and
\stepcounter{thm}
\begin{equation*}\label{map:KUn}\tag{\thethm}
U^n\colon M\to M\end{equation*}
is the map $m\mapsto \det(u_{ij})\cdot m$.

\subsection{The Fundamental Local Isomorphism}\label{ss:fli}
With $R$ and $M$ as above, suppose ${\bf t} = (t_1, \dots, t_r)$ is an $R$-sequence,
$I$ the ideal generated by $\{t_1, \dots, t_r\}$, and $A=R/I$. Then
\begin{comment}
Let $K^\bullet({\bf t}, M)$ be the
{\em Koszul (cohomology) complex} on $M$ and ${\bf t}$.
\marginpar{Have to give explicit description of $K^\bullet({\bf t}, M)$. Equally important, spell out
conventions for tensor products of complexes, Homs, etc -- as in Joe's notes. I essentially copied
that in \cite{cm} and so all that can be found there too. The parenthesis around $A[-r]$ in
\eqref{iso:kosz-r2} is essential because of the sign conventions.} 
We write $K^\bullet({\bf t})$ for $K^\bullet({\bf t},\,R)$. Then

1) $K^\bullet({\bf t}, M)$ is bounded by degrees $0$ and $r$, with
%at the extreme ends (i.e., at degrees $0$ and $r$) we have 
$K^0({\bf t},\,M) = K^r({\bf t}, M) =M$.

2) $K^\bullet({\bf t}, M) = M\otimes_R K^\bullet({\bf t})$.
\end{comment}

1) The ideal $I$ is the image of the coboundary map from $K^{r-1}({\bf t})$ to $K^r({\bf t})=R$, 
%i.e., $I= \im{(K^{r-1}({\bf t}) \to K^\bullet({\bf t})=R)}$, 
and the resulting map of complexes
$K^\bullet({\bf t})\to A[-r]$ 
%(via the canonical map $K^r({\bf t})=R \to R/I=A$) 
is a quasi-isomorphism.
Thus we have an isomorphism in $\D({\mathrm{Mod}}_R)$:
\stepcounter{thm}
\begin{equation*}\label{iso:kosz-r1}\tag{\thethm}
K^\bullet({\bf t}) \iso A[-r].
\end{equation*}
Since $K^\bullet({\bf t}, M)$ is a (bounded) complex of free modules, for every %bounded above
complex $M^\bullet$ we have an isomorphism in $\D({\mathrm{Mod}}_R)$
\stepcounter{thm}
\begin{equation*}\label{iso:kosz-r2}\tag{\thethm}
M^\bullet\otimes_R K^\bullet({\bf t}) \iso M^\bullet{\overset{\bL}{\otimes}}_R(A[-r])
={\overline M}^\bullet{\overset{\bL}{\otimes}}_A(A[-r])
\end{equation*}
where ${\overline{M}}^\bullet=M^\bullet\otimes_RA$.

2) We have $\Hom_R(A,\,M)=\ker{(K^0({\bf t},\, M)\to K^1({\bf t},\,M))}$ where 
$M = K^0({\bf t},\, M)$ and 
$\Hom_R(A,\,M)$ is identified with the submodule of $I$-torsion elements of $M$ 
namely $(0{\<\<\underset{{{}^M}}{\textup{:}}}\<\<I)$ in the usual way (i.e., by ``evaluation at $1$").
We thus have a map of complexes $\Hom_R(A,\,M)[0] \to K^\bullet({\bf t},\,M)$. If $M$ is an
{\em injective} $R$-module then this map is a quasi-isomorphism. It follows that if $M^\bullet$
is a bounded-below complex, and $M^\bullet \to E^\bullet$ is an injective resolution with $E^{\bullet}$
a bounded-below complex, then we have
quasi-isomorphisms $M^\bullet\otimes_RK^\bullet({\bf t}) \to E^\bullet\otimes_RK^\bullet({\bf t})$
and $\Hom_R(A,\,E^\bullet)\to E^\bullet\otimes_RK^\bullet({\bf t})$ so that  
in $\D({\mathrm{Mod}}_R)$ we have an isomorphism
\stepcounter{thm}
\begin{equation*}\label{iso:kosz-rhom}\tag{\thethm}
M^\bullet\otimes_RK^\bullet({\bf t}) \iso \R\Homb_R(A,\, M^\bullet)
\end{equation*}
fitting into a commutative diagram in $\D({\mathrm{Mod}}_R)$ as follows.
\[
\xymatrix{
M^\bullet\otimes_RK^\bullet({\bf t}) 
\ar[d]^{\,\rotatebox{-90}{\makebox[-0.1cm]{\Iso}}}_{\eqref{iso:kosz-rhom}}
\ar[r]^{\Iso} 
& E^\bullet\otimes_RK^\bullet({\bf t})  \\
\R\Homb_R(A,\, M^\bullet)\ar@{=}[r] & \Hom_R(A,\, E^\bullet) 
\ar[u]_{\,\rotatebox{90}{\makebox[-0.1cm]{\Iso}}}
}
\]
In particular we have an isomorphism 
\[\psi_{\bf t}\colon M^\bullet\overset{\bL}{\otimes}_R(A[-r]) \iso \R\Homb_R(A,\,M^\bullet)\]
where $\psi_{\bf t}=\eqref{iso:kosz-rhom}\smcirc\eqref{iso:kosz-r2}^{-1}$.

3) Let  $\frac{\bf 1}{\bf t}$ (or ${\bf 1}/{\bf t}$ for typographical convenience)
be the element of $\wI{A}{I}$ defined in \eqref{def:1/t}.
Then $\wI{A}{I}$ is a free
$A$ module of rank one, with $\frac{\bf 1}{\bf t}$ as a generator. One therefore has an
isomorphism:
\stepcounter{thm}
\begin{equation*}\label{iso:lambda-t}\tag{\thethm}
\lambda_{\bf t}\colon A \iso \wI{A}{I}, 
\end{equation*}
given by $1 \mapsto (-1)^r{\bf 1}/{\bf t}$. The reason for the sign $(-1)^r$ will be clear later.
We thus get an isomorphism, 
\stepcounter{thm}
\begin{equation*}\label{iso:etaRA}\tag{\thethm}
\eta_{R,A}(M^\bullet) \colon M^\bullet\overset{\bL}{\otimes}_R(\wI{A}{I}[-r]) \iso 
\R\Homb_R(A,\,M^\bullet)
\end{equation*}
with $\eta_{R,A} = \psi_{\bf t}\smcirc(\lambda_{\bf t}[-r])^{-1}$. The crucial property here
is that {\em $\eta_{R,A}$ does not depend on ${\bf t}$, even though 
$\psi_{\bf t}$ and $\lambda_{\bf t}$
do.}

The data above fits into the following commutative diagram
\stepcounter{thm}
\[
\begin{aligned}\label{diag:eta-K}
\xymatrix{
M^\bullet\overset{\bL}{\otimes}_R(A[-r]) 
\ar[d]^{\,\rotatebox{-90}{\makebox[-0.1cm]{\Iso}}}_{\text{via $\lambda_{\bf t}$}}
\ar[dr]_{\psi_{\bf t}}
& M^\bullet\otimes_RK^\bullet({\bf t}) \ar[l]_{\Iso}^{\eqref{iso:kosz-r2}}
\ar[d]^{\,\rotatebox{-90}{\makebox[-0.1cm]{\Iso}}}_{\eqref{iso:kosz-rhom}}
\ar[r]^{\Iso} 
& E^\bullet\otimes_RK^\bullet({\bf t})  \\
M^\bullet\overset{\bL}{\otimes}_R(\wI{A}{I}[-r]) \ar[r]^-{\Iso}_-{\eta_{R,A}}
& \R\Homb_R(A,\, M^\bullet)\ar@{=}[r] 
& \Hom_R(A,\, E^\bullet) 
\ar[u]_{\,\rotatebox{90}{\makebox[-0.1cm]{\Iso}}}
}
\end{aligned}\tag{\thethm}
\]
Let $M$ be an $R$-module. Our version of the {\em fundamental local isomorphism} is the isomorphism
\stepcounter{thm}
\begin{equation*}\label{iso:fli}\tag{\thethm}
\phi_{R,A}(M)\colon M\otimes_R\wI{A}{I} \iso \ext^r_R(A,\,M)
\end{equation*}
given by 
\[\phi_{R,A}(M)=\Hr^0(\eta_{R,A}(M[r])).\]

Let us globalize this construction.
%Since $\eta_{R,A}$ is independent of the choice of $\bf t$, the isomorphism in \eqref{iso:fli}
%globalizes. %Standard way-out arguments allow us to globalise this. 
Let $\X$ be a formal scheme,
and $\I$ a coherent ideal sheaf such that the resulting closed immersion
$i: \Z\hookrightarrow \X$ is a regular immersion of codimension $r$, i.e., it is given
locally by a regular sequence of length~$r$. Let us write~$\eN_i$ for the normal bundle
of~$\Z$ in~$\X$, i.e. $\eN_i=(\I/\I^2)^*$ and set
\stepcounter{thm}
\begin{equation*}\label{map:wnor}\tag{\thethm}
\wnor{i}\set \wedge^r\eN_i =\wI{\co_\Z}{\I}.
\end{equation*}
There is a natural isomorphism
\[
\wnor{i} = \wI{\co_\Z}{\I} \iso {\bar i}^*\Ext^r_{\co_{\X}}(\co_{\Z},\, \co_{\X})
%= H^r({\bar i}^*\R\Hom_{\co_{\X}}(\co_{\Z},\, \co_{\X})) 
= H^ri^{\flat}\co_{\X}
\]
obtained by locally gluing the isomorphisms coming from~\eqref{iso:fli} in view of the fact
that~$\eta_{R,A}$ is independent of the choice of~$\bf t$.
Since $i^{\flat}\co_{\X}$ has homology only in degree~$r$ as is obvious locally from~\eqref{iso:etaRA}, 
we obtain a natural isomorphism
\stepcounter{thm} \begin{equation*}\label{iso:fli-global}\tag{\thethm}
 \wnor{i}[-r] \iso i^{\flat}\co_{\X} = {\bar i}^*\R\sHomb_{\X}(\co_{\Z},\, \co_{\X}).
\end{equation*}

%Then for any quasi-coherent module $\F$ on $\X$, we have  
%\stepcounter{thm} \begin{equation*}\label{iso:fli-global}\tag{\thethm}
%\phi_{\X, \Z}(\F)\colon i^*\F\otimes_{\co_{\Z}}\wI{\co_{\Z}}{\I} \iso \Ext^r_R(A,\,M)
%\end{equation*}
%given by \[\phi_{R,A}(M)=\Hr^0(\eta_{R,A}(M[r])).\]

Set
\stepcounter{thm}
\begin{equation*}\label{def:btrg}\tag{\thethm}
i^\btrg \set \bL i^*(\boldsymbol{-})\overset{\bL}{\otimes}_{\co_\Z} (\wnor{i}[-r]).
\end{equation*}
Then for $\eF\in \Dqc(\X)$ we have an isomorphism
\stepcounter{thm}
\begin{equation*}\label{thm:eta-i}\tag{\thethm}
\eta_i(\eF) \colon i^\btrg\eF \iso i^\flat\eF
\end{equation*}
given by the composite
\stepcounter{thm}
\[
\begin{aligned}\label{iso:trg-flat}
i^\btrg \eF = \bL i^*(\eF)\overset{\bL}{\otimes}_{\co_\Z} (\wnor{i}[-r]) 
%\xrightarrow[\eqref{iso:fli-global}]{\Iso} 
%&\iso \bL i^*(\eF)\overset{\bL}{\otimes}_{\co_\Z} i^{\flat} \co_{\X} \\
&\iso \bL i^*\eF \overset{\bL}{\otimes}_{\co_\Z} {\bar i}^*\R\sHomb_{\X}(\co_{\Z},\, \co_{\X}) \\
&\iso {\bar i}^*(\eF \overset{\bL}{\otimes}_{\co_\X} \R\sHomb_{\X}(\co_{\Z},\, \co_{\X})) \\
&\iso {\bar i}^*\R\sHomb_{\X}(\co_{\Z},\, \eF) = i^\flat\eF
\end{aligned}\tag{\thethm}
\]
where the first isomorphism is given by \eqref{iso:fli-global} while the third one results from the fact that $i_*\co_{\Z}$
is coherent and has finite tor dimension over $\co_{\X}$.

For $\eF \in \Dc^+(\X)$, let 
\stepcounter{thm}
\begin{equation*}\label{iso:eta'-i}\tag{\thethm}
\eta'_i(\eF) \colon i^\btrg\eF \iso \ush{i}\eF
\end{equation*}
be the composite $\eta'_i = \eqref{iso:flat-sharp}\smcirc\eta_i$.

\begin{rem}\label{rem:Tr-tensor} In the above, the isomorphism $i^\btrg\co_\X \iso i^\flat\co_\X$ in \eqref{iso:fli-global} 
is what drives the isomorphism \eqref{thm:eta-i}. In slightly greater detail, for $\eF\in\Dc^+(\X)$,
we have (by definition of $i^\btrg$):
 \[i^\btrg\eF=\bL i^*(\eF)\otimes_{\co_\Z} i^\btrg(\co_\X).\]
We also have an isomorphism (whose inverse is the composite of the last two maps
 in \eqref{iso:trg-flat}) 
\[\bL i^*(\eF)\otimes_{\co_\Z} i^\flat(\co_\X) \iso i^\flat(\eF).\] 
Applying $i^\btrg\co_\X \iso i^\flat\co_\X$ (from \eqref{iso:fli-global})
to the two isomorphisms above, we get $\eta_i(\eF)$.

The isomorphism $\bL i^*(\eF)\otimes_{\co_\Z} i^\flat(\co_\X) \iso i^\flat(\eF)$ above is
such that ``evaluation at $1$" is respected. In greater detail if
$\Tr{i}^\flat\colon i_*i^\flat \to {\bf 1}$ is as in \eqref{def:Tr-i-flat}, then the composite
$\eF\otimes_{\co_\X}i_*i^\flat\co_\X \iso i_*(\bL i^*(\eF)\otimes_{\co_\Z} i^\flat(\co_\X))
\iso i_*i^\flat(\eF) \xrightarrow{\Tr{i}^\flat}(\eF)$ is equal to $1\otimes \Tr{i}^\flat(\co_\X)$.
This means that if $\Tr{i}^\btrg\colon i_*i^\btrg \to {\bf 1}$ is defined by the formula 
\[\Tr{i}^\btrg =\Tr{i}^\flat\smcirc i_* \eta_i,\]
then the following diagram commutes
\stepcounter{sth}
\[
\begin{aligned}\label{diag:1-tensor-Tr}
{\xymatrix{
i_*(\bL i^*\eF\otimes_{\co_\Z}i^\btrg\co_\X) \ar@{=}[rr] 
&& i_*i^\btrg \ar[dd]^{\Tr{i}^\btrg(\eF)}\\
\eF\otimes_{\co_\X} i_*i^\btrg\co_{\X} \ar[d]_{{\bf 1}\otimes\Tr{i}^\btrg(\co_\X)} 
\ar[u]_{\>\rotatebox{90}{\makebox[-0.1cm]{\Iso}}}^{{\text{projection formula}}} && \\
\eF\otimes_{\co_\X} \co_\X \ar@{=}[rr] && \eF
}}
\end{aligned}\tag{\thesth}
\]
\end{rem}

% \marginpar{Have to give reference from
%[RD] for fundamental local isomorphism outside the affine case. Don't have the patience to
%look it up now. Also should look up Conrad. Plus we have formal schemes. I am sticking to
%bounded complexes with coherent cohomology for security at the moment.}
\setcounter{subsubsection}{\value{thm}}
\subsubsection{}\label{sss:Tr-concrete} \stepcounter{thm}
If $\X=X$ is an ordinary scheme, so that $\ush{i}=i^!$, then the maps
$\eta_i(\eF)$ and $\eta_i'(\eF)$ above can be extended to
isomorphisms for $\eF\in\Dqc(X)$, without any 
boundedness hypotheses on $\eF$. In greater detail, recall that a complex $\eF$
of $\co_X$-modules is called {\emph{perfect}} if there exist $a, b \in\mathbb{Z}$,
 $a \le b$, and locally $\eF$ is
$\D(X)$-isomorphic to a complex $E$ of finite rank free $\co_X$-modules with $E^n=0$ for
$n\notin [a, b]$. The map $i_*$ takes perfect complexes to perfect complexes (locally use
appropriate Koszul complexes!). In other words $i$ is a {\emph{quasi-perfect map}} (see
\cite[p.\,192, Definition 4.7.2]{notes}). According to a result of Neeman in \cite{bous} and
Bondal and van den Bergh in \cite{bb}, since $i_*$ takes perfect complexes to perfect
complexes, one has a unique isomorphism (with $Z=\Z$)
\[\bL i^*(\eF)\overset{\bL}{\otimes}_{\co_Z}i^!\co_X \iso i^!\eF\]
such that Diagram \eqref{diag:1-tensor-Tr} commutes with $i^\btrg$ replaced by $i^!$,
$\Tr{i}^\btrg$ by $\Tr{i}$, the equality on the top row by $i_*$ of the isomorphism
displayed above, and allowing $\eF$ to vary $\Dqc(X)$ rather than in $\Dc^+(X)$. It is
now clear that one can extend $\eta'_i$ to an isomorphism of functors on $\Dqc(X)$.
As for $\eta_i$, see \cite[p.\,53, (2.5.3)]{conrad}, keeping in mind the differing sign
conventions for $K^\bullet({\bf t})$ as well as the order of the tensor product. In fact
the isomorphism ${\bar i}^*(\eF \overset{\bL}{\otimes}_{\co_X} \R\sHomb_{X}(\co_{Z},\, \co_{X}))
\iso {\bar i}^*\R\sHomb_{X}(\co_{Z},\, \eF) $ in \eqref{iso:trg-flat} works for
$\eF\in\Dqc(X)$ when $X$ is an ordinary scheme.

In view of \eqref{diag:1-tensor-Tr}, in order to understand $\Tr{i}^\btrg$ it is enough to
understand $\Tr{i}^\btrg(\co_X)$. We give an explicit representation of $\Tr{i}^\btrg$ when
$X=\Spec{\,R}$, $Z=\Spec{\,A}$, and the $I=\ker{R\twoheadrightarrow A}$ is generated
by a quasi-regular sequence ${\bf t}=(t_1, \dots, t_r)$i, i.e., the situation we have been
with for most of this section. Let $N=\Gamma(X,\,\wnor{i})$.
In this case, the quasi-isomorphism of
complexes of $R$-modules $K^\bullet({\bf t})\to A[-r]$
in \eqref{iso:kosz-r1}, is the map defined by 
$K^r({\bf t})= R \xrightarrow{{\text{natural}}}\mathrel{\mkern-14mu}\rightarrow R/I = (A[-r])^r$.

Using the isomorphism $A\iso N$ given by $1 \mapsto {\bf 1}/{\bf t}$ we get a quasi-isomorphism
\[\varphi_{\bf t}\colon K^\bullet({\bf t}) \lra N[-r],\]
where $\varphi_{\bf t}$ is defined by 
$\varphi_{\bf t}^r\colon K^r({\bf t}) = R \to N= (N[-r])^r$, the arrow $R\to N$ 
being $1\mapsto {\bf 1}/{\bf t}$. 
For a complex of $R$-modules $M^\bullet$, let
$\Tr{A/R}^\btrg(M^\bullet) \colon M^\bullet\otimes_R N[-r] \to M^\bullet$ and
$\Tr{A/R}^\flat(M^\bullet)\colon \R\Homb_R(A, M^\bullet) \to M^\bullet$ be the maps
corresponding to $\Tr{i}^\btrg(\wit{M}^\bullet)$ and $\Tr{i}^\flat(\wit{M}^\bullet)$.
By definition of \eqref{thm:eta-i}, we have a commutative diagram in the category
$\D({\mathrm{Mod}}_R)$ with {\emph{isomorphisms}} bordering the triangle on the right:
\[
{\xymatrix{
 N[-r] \ar[rrd]_{\eta_i} \ar[d]_{\Tr{A/R}^\btrg(R)} & & K^\bullet({\bf t}) \ar[ll]_{\varphi_{\bf t}} 
\ar[d]^{\>\>\eqref{iso:kosz-rhom}} & \\
 R & & \R\Homb_R(A,\,R) \ar[ll]^-{\Tr{A/R}^\flat(R)}
}}
\]
The composite $\Tr{A/R}^\flat(R)\smcirc\eqref{iso:kosz-rhom}$ is the natural projection
\[\pi_{\bf t}\colon 
K^\bullet({\bf t}) \xrightarrow{\phantom{XX}}\mathrel{\mkern-14mu}\rightarrow K^0({\bf t})=R\]
which is a map of complexes, since $K^\bullet({\bf t})$ has no negative terms. Thus
\stepcounter{sth}
\begin{equation*}\label{map:Tr-concrete}\tag{\thesth}
\Tr{A/R}^\btrg(R) = \pi_{\bf t}\smcirc\varphi_{\bf t}^{-1}.
\end{equation*}

\subsection{Compatibility with completions} In view of the above, \Lref{lem:ij-kappa}
has a useful re-interpretation in the special case where the two
closed immersions of $Z$ into $X$ and $\X$ are regular immersions of codimension $r$. 
In greater detail, suppose as in \Sref{subsec:i-j-k}, we have a commutative diagram
\[
\xymatrix{
Z\phantom{X} \ar@{^(->}[r]^{j} \ar@{_(->}[dr]_{i} & \X \ar[d]^{\kappa} \\
& X
}
\]
with $X$ an ordinary scheme, but with {\em $i$, $j$ regular closed immersions},
$\X=X_{/Z}$ the completion of $X$ along $Z$, $\kappa$ the completion map, and let $\I$ and
$\J=\I\co_{\X}$ be the ideal sheaves for $Z$ in $X$ and $\X$ respectively. Now regarding $\I/\I^2$ and
$\J/\J^2$ as invertible sheaves on $Z$, we have an obvious identification $\I/\I^2=\J/\J^2$, whence
the identification $j^\btrg\kappa^*=i^\btrg$.
Then the following is an easy corollary to \Lref{lem:ij-kappa}.
\begin{lem}\label{lem:eta-kappa} 
The following diagram commutes.
\[
\xymatrix{
j^\btrg \kappa^*  
\ar[rr]^-{\Iso}_-{\eta'_j} && \ush{j}\kappa^* \ar[r]^{\Iso}_{\eqref{eq:gm}}
& \ush{j}\ush{\kappa} 
\ar[d]^{\,\rotatebox{-90}{\makebox[-0.1cm]{\Iso}}}\\
i^\btrg 
\ar@{=}[u]
\ar[rrr]^-{\Iso}_-{\eta'_i}& & & \ush{i} 
}
\]
where the unlabelled isomorphism $\ush{j}\ush{\kappa}\iso \ush{i}$ is the canonical one. 
\end{lem}

%\section{\bf Pull-back}
\subsection{Compatibility between the flat-base-change isomorphisms of~${\boldsymbol{-}}^\btrg$ 
and of~$\ush{-}$}

Suppose we have a cartesian diagram $\mathfrak s$ of formal schemes
\stepcounter{thm}
\[
\begin{aligned}\label{diag:bc0}
\xymatrix{
\W' \ar@{}[dr]|{\square} \ar[d]_{\kappa} \ar[r]^j & \W \ar[d]^{\kappa_{{}_0}}\\
\X' \ar[r]_i & \X
}
\end{aligned}\tag{\thethm}
\]
such that $i$ is a regular immersion (i.e., given locally by the vanishing of a regular sequence)
and $\kappa_{{}_0}$ is the completion of $\X$ with respect to a closed subscheme given by
a coherent ideal.
By \eqref{iso:eta'-i}, for any $\eF \in \Dc^+(\X)$ and $\eG \in \Dc^+(\W)$ there are natural isomorphisms 
\[
i^\btrg\eF \iso \ush{i}\eF, \qquad \qquad j^\btrg\eG \iso \ush{j}\eG.
\]

Now, on one hand we have the flat base-change isomorphism 
\[
\ush{\beta_{\mathfrak s}} \colon \kappa^*\ush{i} \iso \ush{j}\kappa_{{}_0}^*
\]
of \eqref{iso:bc-sharp} while on the other we have an isomorphism
\stepcounter{thm}
\begin{equation*}\label{iso:BCf}\tag{\thethm}
\kappa^*i^\btrg \iso j^\btrg\kappa_{{}_0}^*
\end{equation*}
given by the composite
\begin{align*}
\kappa^*((\bL i^*({\boldsymbol{-}})\overset{\bL}\otimes \wnor{i}[-r])
&\iso (\bL j^*\bL\kappa_{{}_0}^*({\boldsymbol{-}}))\overset{\bL}\otimes\kappa^* \wnor{i}[-r]\\
&\iso (\bL j^*\kappa_{{}_0}^*({\boldsymbol{-}}))\overset{\bL}\otimes \wnor{j}[-r]
\end{align*}
where the second isomorphism is the one that arises from the canonical isomorphism
%$\bL\kappa^*\eN_{\<\<\<\<{}_i}=$
$\kappa^*\eN_{\<\<\<\<{}_i} \iso \eN_{\<\<\<\<{}_j}$.
Fortunately these two flat-base-change isomorphisms are compatible:

\begin{prop}
\label{prop:bc-!f} 
For the diagram $\mathfrak s$ in \eqref{diag:bc0}, for any $\eF \in \Dc^+(\X)$ 
%let $\eF$ be a bounded-above complex of coherent locally free $\co_{\X}$-modules. Then 
the following diagram commutes.
\stepcounter{sth}
\[
\begin{aligned}\label{diag:bc}
\xymatrix{
\kappa^*i^\btrg\eF \ar[d]^{\,\rotatebox{-90}{\makebox[-0.1cm]{\Iso}}}_{\eta'_i} \ar[r]^\Iso_{\eqref{iso:BCf}} & 
j^\btrg\kappa_{{}_0}^*\eF \ar[d]_{\,\rotatebox{-90}{\makebox[-0.1cm]{\Iso}}}^{\eta'_j} \\
\kappa^*\ush{i}\eF \ar[r]_{\ush{\beta_{\mathfrak s}}} & \ush{j}\kappa_{{}_0}^*\eF
}
\end{aligned}\tag{\thesth}
\]
\end{prop}

\begin{proof}
As per the definition of $\eta'$ in \eqref{iso:eta'-i}, the diagram in \eqref{diag:bc} expands as follows
\[
\xymatrix{
\kappa^*i^\btrg\eF \ar[d]^{\,\rotatebox{-90}{\makebox[-0.1cm]{\Iso}}}_{\eta_i} \ar[r]^\Iso_{\eqref{iso:BCf}} 
& j^\btrg\kappa_{{}_0}^*\eF \ar[d]_{\,\rotatebox{-90}{\makebox[-0.1cm]{\Iso}}}^{\eta_j} \\
\kappa^*i^{\flat}\eF \ar[r]_{\beta^\flat} \ar[d]^{\,\rotatebox{-90}{\makebox[-0.1cm]{\Iso}}}_{\eqref{iso:flat-sharp}}
& j^{\flat}\kappa_{{}_0}^*\eF \ar[d]_{\,\rotatebox{-90}{\makebox[-0.1cm]{\Iso}}}^{\eqref{iso:flat-sharp}} \\
\kappa^*\ush{i}\eF \ar[r]_{\ush{\beta_{\mathfrak s}}} & \ush{j}\kappa_{{}_0}^*\eF 
}
\]
where $\beta^\flat$ is induced by the natural isomorphism
\stepcounter{thm}
\begin{equation*}\label{eq:explicit-basech}\tag{\thethm}
\kappa_{{}_0}^*\R\sHomb_{\co_\X}(i_*\co_{\X'}, \eF) \iso \R\sHomb_{\co_\W}(j_*\co_{\W'}, \kappa_{{}_0}^*\eF).
\end{equation*}
It is straightforward to check that the top rectangle commutes. For the bottom one, using the 
adjointness property of~$\ush{j}$, it suffices to check that the outer border of the 
following diagram commutes where, as before, $\iGp{\X} = \R\iGp{\X}$ etc..
\begin{comment} %%%%%%%%%%
\[
\xymatrix{  
& j_*\iGp{\W'}\kappa^*i^{\flat} \ar@{=}[ld] \ar[rd] \ar[d]_{a_1} \ar[rr]^{\text{via }\beta^\flat} 
& & j_*\iGp{\W}j^\flat\kappa_{{}_0}^* \ar[dd]  \\
j_*\iGp{\W'}\kappa^*i^{\flat} \ar[r]_{a_1} \ar[d]_{b_1} 
& \kappa_{{}_0}^*i_*\iGp{\X'}i^\flat \ar[d]_{b_2}  \ar[rd]  
& j_*\kappa^*i^{\flat} \ar[d] \ar[rd]^{\text{via }\beta^\flat}  \\
%& \kappa_{{}_0}^*i_*i^\flat  \ar[r] & j_*j^\flat\kappa_{{}_0}^*  \\
j_*\iGp{\W'}\kappa^*\ush{i}i_*\iGp{\X'}i^{\flat}  \ar[r]_{a_2} \ar[d]  
& \kappa_{{}_0}^*i_*\iGp{\X'}\ush{i}i_*\iGp{\X'}i^\flat  \ar[d]^{\qquad \qquad\boxminus} & 
\kappa_{{}_0}^*i_*i^\flat \ar[rdd]_{\text{via }\Tr{i}^\flat} \ar[r]_{\eqref{eq:explicit-basech}} 
& j_*j^\flat\kappa_{{}_0}^* \ar[dd]^{\text{via }\Tr{j}^\flat} \\
j_*\iGp{\W'}\kappa^*\ush{i}i_*i^{\flat} \ar[r]_{a_3} \ar[d]_{\text{via }\Tr{i}^\flat}  
& \kappa_{{}_0}^*i_*\iGp{\X'}\ush{i}i_*i^\flat \ar[d]_{\text{via }\Tr{i}^\flat}  \\
j_*\iGp{\W'}\kappa^*\ush{i} \ar[r]_{a_4} & \kappa_{{}_0}^*i_*\iGp{\X'}\ush{i} \ar[rr]_{\text{via } \Tr{i}}  
& & \kappa_{{}_0}^*
}  
\]  
\end{comment} 
%%%%%%%%%%%%
\[
\xymatrix{  
j_*\iGp{\W'}\kappa^*i^{\flat}  \ar[dddd]_{i^{\flat} \cong \ush{i}} \ar[rrd]  \ar[rrr]^{\text{via }\beta^\flat} \ar[rdd]_{a_1} 
& & & j_*\iGp{\W}j^\flat\kappa_{{}_0}^* \ar[dd]  \\
%j_*\iGp{\W'}\kappa^*i^{\flat} \ar[r]_{a_1}  
& & j_*\kappa^*i^{\flat} \ar[d] \ar[rd]^{\text{via }\beta^\flat}  \\
%j_*\iGp{\W'}\kappa^*\ush{i}i_*\iGp{\X'}i^{\flat}  \ar[r]_{a_2} \ar[d]  
%& \kappa_{{}_0}^*i_*\iGp{\X'}\ush{i}i_*\iGp{\X'}i^\flat  \ar[d]^{\qquad \qquad\boxminus} & 
& \kappa_{{}_0}^*i_*\iGp{\X'}i^\flat \ar[dd]_{i^{\flat} \cong \ush{i}}^{\hspace{3em}\boxminus}  \ar[r]  
& \kappa_{{}_0}^*i_*i^\flat \ar[rdd]_{\text{via }\Tr{i}^\flat} \ar[r]_{\eqref{eq:explicit-basech}} 
& j_*j^\flat\kappa_{{}_0}^* \ar[dd]^{\text{via }\Tr{j}^\flat} \\
%j_*\iGp{\W'}\kappa^*\ush{i}i_*i^{\flat} \ar[r]_{a_3} \ar[d]_{\text{via }\Tr{i}^\flat}  
%& \kappa_{{}_0}^*i_*\iGp{\X'}\ush{i}i_*i^\flat \ar[d]_{\text{via }\Tr{i}^\flat}  \\
\\
j_*\iGp{\W'}\kappa^*\ush{i} \ar[r]_{a_2} & \kappa_{{}_0}^*i_*\iGp{\X'}\ush{i} \ar[rr]_{\text{via } \Tr{i}}  
& & \kappa_{{}_0}^*
}  
\]  
Here the maps $a_i$ are induced by the composite of natural maps 
\[
j_*\R\iGp{\W'}\kappa^* \to j_*\kappa^*\R\iGp{\X'} \iso \kappa_{{}_0}^*i_*\iGp{\X'}.
\]
%while the $b_i$'s are induced by the natural map $i^{\flat} \to \ush{i}$.
%cotrace ${{\bf 1} \to  \ush{i}i_*\iGp{\X'}}$.
The unlabelled maps are the obvious natural ones. The diagram $\scriptstyle{\boxminus}$ 
commutes by definition of the map $i^{\flat} \to \ush{i}$ in~${\eqref{iso:flat-sharp}}$.
Commutativity of the remaining parts is easy to check.
\end{proof}

\subsection{Stable Koszul complexes and generalized fractions}\label{ss:stabkoz} 
Let $R$, $I$, $A$ be
as above, and let ${\bf t}=(t_1,\ldots, t_d)$ be generators for $I$. {\em Note that, for now, we are not
requiring ${\bf t}$ to be a quasi-regular sequence.}
%\marginpar{Too much effort to drop the requirement that ${\bf t}$ is an $R$-sequence, and
%anyway need it in the Lemma.}
We now recall the relationship between $K^\bullet({\bf t},\,M)$ and the
local cohomology of $M$ and relate the above discussion to generalized fractions leading
to the explicit formula in \Lref{lem:gen-frac} below.
For an $r$-tuple of positive integers $\boldsymbol{\alpha}=(\alpha_1,\dots,\,\alpha_r)$, let 
${\bf t}^{\boldsymbol{\alpha}}= (t_1^{\alpha_1}, \dots, t_r^{\alpha_r})$. Let
\stepcounter{thm}
\[
\begin{aligned}\label{def:kos-cech}
K^\bullet_\infty({\bf t}) \set & \dirlm{{\boldsymbol{\alpha}}}K^\bullet({\bf t}^{\boldsymbol{\alpha}}) \\
K^\bullet_\infty({\bf t},\,M) \set &\dirlm{{\boldsymbol{\alpha}}}K^\bullet({\bf t}^{\boldsymbol{\alpha}},\,M)
=M\otimes_R K_\infty^\bullet({\bf t}).
\end{aligned}\tag{\thethm}
\]
%Since each ${\bf t}^{\boldsymbol{\alpha}}$ is an $R$-sequence,  
%$\Hr^i(K^\bullet({\bf t}^{\boldsymbol{\alpha}}))=0$ for $i\neq 0$, and 
%$\Hr^r(K^\bullet({\bf t}^{\boldsymbol{\alpha}}))=R/{\bf t}^{\boldsymbol{\alpha}}R$. Define
The complex $K^\bullet_\infty({\bf t},\,M)$ is called the {\em stable Koszul complex of $M$
associated to~${\bf t}$} and it has a well known relationship with the \v{C}ech complex
${\mathrm C}^\bullet={{\mathrm C}}^\bullet(\mathfrak{U},\,\wit{M})$ associated
with the open cover $\mathfrak{U}= \{\{t_i\neq 0\}\mid i=1,\dots, r\}$ of the scheme
$U\set \Spec{\,R}\smallsetminus V(I)$. The relationship is that 
${\mathrm C}^i=K^{i+1}_\infty({\bf t},\,M)$
for $i\ge 0$ and in this range
and the coboundary maps 
${\mathrm C}^i\to {\mathrm C}^{i+1}$ and $K^{i+1}_\infty({\bf t},\,M) \to K^{i+2}_\infty({\bf t},\,M)$
are equal. We also note that the natural map $K^r({\bf t},\,M) \to K^r_\infty({\bf t}, M)={\mathrm C}^{r-1}$, is the
map $M\to M_{t_1\dots t_r}$ given by $m\mapsto m/t_1\dots t_r$.

We point out that there is an obvious commutative diagram
\[
\xymatrix{
\dirlm{{\boldsymbol{\alpha}}}\Hom_R(R/{\bf t}^{\boldsymbol{\alpha}},\,M) \ar[r] \ar@{=}[d] & 
K^0_\infty({\bf t},\,M)  \ar@{=}[d] \\
\Gamma_I(M) \, \ar@{^(->}[r] & M 
}
\]
where the horizontal arrow in the top row is the one obtained by applying a direct limit to the
map of direct systems 
$\Hom_R(R/{\bf t}^{\boldsymbol{\alpha}},\,M) \to K^0({\bf t}^{\boldsymbol{\alpha}},\,M)$ and the
horizontal arrow in the bottom row is the natural inclusion. If, as before, $M\to E^\bullet$ is an
injective resolution of $M$, we have 
$\Hom_R(R/{\bf t}^{\boldsymbol{\alpha}}R,\, E^\bullet)\iso 
E^\bullet\otimes_RK^\bullet({\bf t}^{\boldsymbol{\alpha}})$ whence an isomorphism
\[
\dirlm{{\boldsymbol{\alpha}}}\Hom_R(R/{\bf t}^{\boldsymbol{\alpha}}R,\, E^\bullet)
\iso E^\bullet\otimes_R K^\bullet_\infty({\bf t}).
\]
We then we have a diagram of isomorphisms in $\D({\mathrm{Mod}}_R)$:
\[
\xymatrix{
K^\bullet_\infty({\bf t},\,M) \ar@{=}[r]
\ar@{.>}[d]^{\,\rotatebox{-90}{\makebox[-0.1cm]{\Iso}}}
& M^\bullet\otimes_RK^\bullet_\infty({\bf t})
 \ar[d]^{\,\rotatebox{-90}{\makebox[-0.1cm]{\Iso}}} \\ 
 \R\Gamma_I(M) \ar@{=}[d]& E^\bullet\otimes_R K^\bullet_\infty({\bf t}) \\
 \Gamma_I(E^\bullet) \ar@{=}[r]
 & \dirlm{{\boldsymbol{\alpha}}}\Hom_R(R/{\bf t}^{\boldsymbol{\alpha}}R,\, E^\bullet)
 \ar[u]_{\,\rotatebox{90}{\makebox[-0.1cm]{\Iso}}}
 }
\]
Since all solid arrows in this diagram are isomorphisms, we can fill the dotted arrow, i.e.,
we have a unique isomorphism
\stepcounter{thm}
\begin{equation*}\label{iso:k-infty-gam}\tag{\thethm}
K^\bullet_\infty({\bf t},\,M) \iso \R\Gamma_I(M)
\end{equation*}
which fills the dotted arrow to make the diagram commute. 
Since $K^j_\infty({\bf t},\,M)=0$
for $j > r$, we have a surjective map $M_{t_1\dots t_r}=K^r_\infty({\bf t},\,M) \to \Hr^r_I(M)$.
The image of $m/t_1^{\alpha_1}\dots t_r^{\alpha_r}\in M_{t_1\dots t_r}$ is denoted by
the generalised fraction 
$\bigl \{\begin{smallmatrix} m\\ t^{\alpha_1}_1,\,\dots,\,t^{\alpha_r}_r \end{smallmatrix}\bigr \}$.

Now, standard excision arguments give us a map $\Hr^{r-1}(U,\,\wit{M})\to H^r_I(M)$ which
is an isomorphism when $r\ge 2$ and surjective when $r=1$.
%\marginpar{Have to give reference for all this stable Koszul, Cech, generalised fractions stuff.} 
In the \v{C}ech complex ${\mathrm C}^\bullet$, we have ${\mathrm C}^j=0$ for $j\ge r$.
We thus have a composition of surjective maps
\[ M_{t_1\dots t_r}={\mathrm C}^{r-1} \twoheadrightarrow \Hr^{r-1}(\mathfrak{U},\,\wit{M}) \iso
\Hr^{r-1}(U,\,\wit{M}) \twoheadrightarrow \Hr^r_I(M).\]
The image of $\frac{m}{t_1^{\alpha_1}\dots t_r^{\alpha_r}} \in M_{t_1\dots t_r}$ is denoted
by the generalized fraction 
$\bigl [\begin{smallmatrix} m\\ t^{\alpha_1}_1,\,\dots,\,t^{\alpha_r}_r \end{smallmatrix}\bigr ]$.
The two generalized fractions are related by the formula 
\stepcounter{thm}
\begin{equation*}\label{eq:2gen-fracs}\tag{\thethm}
\begin{bmatrix} m\\ t^{\alpha_1}_1,\,\dots,\,t^{\alpha_r}_r \end{bmatrix}
= (-1)^r \begin{Bmatrix} m\\ t^{\alpha_1}_1,\,\dots,\,t^{\alpha_r}_r \end{Bmatrix}
\end{equation*}
(see \cite[p.47, Lemma\,4.1.1]{lns}).

\begin{lem}\label{lem:gen-frac} Suppose the sequence ${\bf t}$ above is a quasi-regular sequence in $R$.
Let $M$ be an $R$-module. Then the composite map (with $\phi_{R,A}(M)$ as in \eqref{iso:fli})
\stepcounter{sth}
\begin{equation*}\label{foo}\tag{\thesth}
M\otimes_R\wI{A}{I} \xrightarrow[\phi_{R,A}(M)]{\Iso} \ext^r_R(A,\,M) \longrightarrow \Hr^r_I(M)
\end{equation*}
is given by
\[
m\otimes \frac{\bf 1}{\bf t} \mapsto \begin{bmatrix} m\\ t_1,\,\dots,\,t_r \end{bmatrix} \qquad (m\in M)
\]
Where $\frac{\bf 1}{\bf t}$ is as in \eqref{def:1/t}.
\end{lem}

\proof For an arbitrary bounded complex $M^\bullet$  consider the following commutative diagram
%\marginpar{$M^\bullet$ bounded?}
\stepcounter{sth}
\stepcounter{subsubsection}
\[
\begin{aligned}\label{diag:foo-M}
\xymatrix{
M^\bullet\overset{\bL}{\otimes}_R(A[-r])  
\ar[ddd]^{\,\rotatebox{-90}{\makebox[-0.1cm]{\Iso}}}_{\lambda_{\bf t}[-r]} 
& M^\bullet\otimes_RK^\bullet({\bf t}) \ar[l]_{\Iso}  \ar[r]
\ar[d]^{\,\rotatebox{-90}{\makebox[-0.1cm]{\Iso}}} 
& M^\bullet\otimes_RK^\bullet_\infty({\bf t})
 \ar[d]^{\,\rotatebox{-90}{\makebox[-0.1cm]{\Iso}}} \\
 & E^\bullet\otimes_RK^\bullet({\bf t}) \ar[r] & E^\bullet\otimes_RK^\bullet_\infty({\bf t})\\
& \Hom_R(A,\,E^\bullet)  \ar[u]_{\,\rotatebox{90}{\makebox[-0.1cm]{\Iso}}} \ar[r]  \ar@{=}[d]
& \Gamma_I(E^\bullet) \ar[u]_{\,\rotatebox{90}{\makebox[-0.1cm]{\Iso}}} \ar@{=}[d] \\
M^\bullet\overset{\bL}{\otimes}_R(\wI{A}{I}[-r]) \ar[r]^-{\Iso}_-{\eta_{R,A}} 
& \R\Homb_R(A,\,M^\bullet) \ar[r] & \R\Gamma_I(M^\bullet)
}
\end{aligned}\tag{\thesth}
\]
Set $M^\bullet=M[r]$ in the above and apply the cohomology functor $\Hr^0({\boldsymbol{-}})$.
We get a commutative diagram
\[
\xymatrix{
M\otimes_RA \ar[dd]^{\,\rotatebox{-90}{\makebox[-0.1cm]{\Iso}}}_{\lambda_{\bf t}} \ar@{=}[rr]
&& M/IM \ar[rr] 
& &\Hr^r(M\otimes_R\K^\bullet_\infty({\bf t})) 
\ar[dd]_{\,\rotatebox{-90}{\makebox[-0.1cm]{\Iso}}} \\
&&&&\\
M\otimes_R\wI{A}{I} \ar[rr]_{\phi_{R,A}(M)}^{\Iso} 
&& \ext^r_R(A,\,M) \ar[rr]^{\text{natural}}
&& \Hr^r_I(M)
}
\]
Let us write $[x]$ for the image of $x\in M_{t_1\dots t_r}=M\otimes_RK^r_\infty({\bf t})$ in the
module $\Hr^r(M\otimes_RK^\bullet_\infty({\bf t}))$. Then chasing an element 
$m\otimes({\bf 1}/{\bf t}) \in M\otimes_R\wI{A}{I}$ by first going north (via $\lambda_{\bf t}^{-1}$)
and then east along the above rectangle, we arrive at the element
$(-1)^r[m/t_1\dots t_r]\in \Hr^r(M\otimes_RK^\bullet_\infty({\bf t}))$. The assertion follows from
\eqref{eq:2gen-fracs}.
\qed

\subsection{Duality for composite for closed immersions} Let  $R$ be a noetherian ring, 
$I\subset R$ an ideal, $A=R/I$ and $i\colon \Spec{\,A}\hookrightarrow \Spec{\,R}$ 
the closed immersion corresponding to the natural surjection $R\twoheadrightarrow R/I=A$.
Let $M^\bullet$ be a bounded below complex of $A$-modules. Consider the 
``evaluation at $1$" map:
\[{\bf ev}_{\<\<{}_I}\colon \R\Homb_R(A,\,M^\bullet) \lra M^\bullet.\]
As is well-known (and easy to verify from the definitions) the following diagram commutes
\[
{\xymatrix{
i_*i^\flat\wit{M}^\bullet \ar[rr]_{\eqref{iso:flat-sharp}}^{\Iso}\ar@/_.75pc/[drr]_{{\bf ev}_{\<\<{}_I}}
&& i_*i^!\wit{M}^\bullet \ar[d]^{\Tr{i}} \\
&& \wit{M}^\bullet
}}
\]
Now suppose $\bar{L} \subset A$ is an ideal, and $L\subset R$ the unique $R$-ideal such that
$L\supset I$ and $L/I=\bar{L}$. We then have the standard isomorphism
\stepcounter{thm}\stepcounter{subsubsection}
\begin{equation*}\label{iso:I-L-it}\tag{\thethm}
\R\Homb_A(B,\,\R\Homb_R(A,\,M^\bullet)) \iso \R\Homb_R(B,\,M^\bullet)
\end{equation*}
which, after replacing $M^\bullet$ by a complex of injective modules if necessary, amounts to
the observation that elements in an $R$-module which are killed by $I$ and also by $\bar{L}$ are
the exactly elements which are killed by $L$. The following diagram clearly commutes
\[
{\xymatrix{
\R\Homb_A(B,\,\R\Homb_R(A,\,M^\bullet)) \ar[r]^-{\Iso} \ar[d]_{{\bf ev}_{\<\<{}_{\bar{L}}}} 
& \R\Homb_R(B,\,M^\bullet) \ar[d]^{{\bf ev}_{\<\<{}_L}} \\
\R\Homb_R(A,\,M^\bullet) \ar[r]_-{{\bf ev}_{\<\<{}_I}} & M^\bullet
}}
\]
This means that the following diagram commutes (with $j\colon \Spec{\,B}\to \Spec{\,A}$
the natural inclusion):
\stepcounter{thm}\stepcounter{subsubsection}
\[
\begin{aligned}\label{diag:ij-it}
{\xymatrix{
j^\flat i^\flat\wit{M}^\bullet \ar[rr]^{\Iso}_{\eqref{iso:I-L-it}} \ar[d]^{\,\rotatebox{-90}{\makebox[-0.1cm]{\Iso}}}_{\eqref{iso:flat-sharp}} 
&& (ij)^\flat\wit{M}^\bullet  \ar[d]_{\,\rotatebox{-90}{\makebox[-0.1cm]{\Iso}}}^{\eqref{iso:flat-sharp}}\\
j^!i^!\wit{M}^\bullet \ar[rr]^{\Iso}_{\text{natural}} && (ij)^!\wit{M}^\bullet
}}
\end{aligned}\tag{\thethm}
\]

Suppose $I$ is generated by ${\bf t}=(t_1,\ldots,t_d)$, $\bar{L}$ is generated by
$\bar{\bf u}= (\bar{u}_1,\,\ldots,\,\bar{u}_e)$, and $u_i\in L$ are lifts of ${\bar u}_i$ for
$i=1,\ldots, e$. Set ${\bf u}=(u_1,\,\ldots,\,u_e)$. Suppose $({\bf t},\,{\bf u})$ is a quasi-regular
sequence in $R$ (so that $\bar{\bf u}$ is quasi-regular in $A$). The map
\stepcounter{thm}
\begin{equation*}\label{map:e-t}\tag{\thethm}
{\bf e_t}\colon M^\bullet\otimes_R K^\bullet({\bf t}) \lra M^\bullet
\end{equation*}
corresponding to ${\bf ev}_{\<\<{}_I}$ under the isomorphism 
$ M^\bullet\otimes K^\bullet({\bf t}) \iso \R\Homb_R(A,\,M^\bullet)$ 
(cf.\,\eqref{iso:kosz-rhom})) is the map which in degree $n$ is
\stepcounter{thm}
\begin{equation*}\label{map:e-t2}\tag{\thethm}
(M^\bullet\otimes_RK^\bullet({\bf t}))^n =\bigoplus_{{p+q}=n}M^p\otimes_RK^q({\bf t})
\xrightarrow{\text{projection}} M^n\otimes_RK^0({\bf t}) = M^n.
\end{equation*}
(Note that the map ${\bf e_t}$ is defined even if ${\bf t}$ is not quasi-regular, so
that in particular ${\bf e_u}$ makes sense.) The following diagram clearly commutes
\[
{\xymatrix{
M^\bullet\otimes_RK^\bullet({\bf t})\otimes_RK^\bullet({\bf u}) \ar@{=}[rr] \ar[d]_{\bf e_u} 
&& M^\bullet\otimes_RK^\bullet({\bf t,u}) \ar[d]^{\bf e_{(t,u)}} \\
M^\bullet\otimes_RK^\bullet({\bf t}) \ar[rr]_{{\bf e_u}} && M^\bullet
}}
\]
An obvious re-interpretation of this, in our case, is that the following diagram commutes
(we are implicitly using the fact that if $N$ is an $A$-module, then ${\bf e_u}={\bf e_{\bar u}}$
on $N\otimes_RK^\bullet({\bf u})=N\otimes_AK^\bullet(\bar{\bf u})$):
\stepcounter{thm}\stepcounter{subsubsection}
\[
\begin{aligned}\label{diag:koz-flat-it}
{\xymatrix{
M^\bullet\otimes_RK^\bullet({\bf t,u}) \ar@{=}[d] \ar[rr]^{\Iso}_{\eqref{iso:kosz-rhom}} 
&& \R\Homb_R(B,\,M^\bullet) \\
M^\bullet\otimes_RK^\bullet({\bf t})\otimes_RK^\bullet({\bf u}) 
\ar[d]^{\,\rotatebox{-90}{\makebox[-0.1cm]{\Iso}}}_{\eqref{iso:kosz-rhom}}  & & \\
\R\Homb_R(A,\,M^\bullet)\otimes_RK^\bullet({\bf u}) \ar@{=}[d] && \\
\R\Homb_R(A,\,M^\bullet)\otimes_AK^\bullet(\bar{\bf u}) \ar[rr]^{\Iso}_{\eqref{iso:kosz-rhom}}
&& \R\Homb_A(B,\,\R\Homb_R(A,\,M^\bullet)) 
\ar[uuu]^{\,\rotatebox{-90}{\makebox[-0.1cm]{\Iso}}}_{\eqref{iso:I-L-it}}
}}
\end{aligned}\tag{\thethm}
\]

Setting $n=d+e$ we have
an isomorphism of rank one free $A$-modules
\[\alpha\colon (\wedge^d_R I/I^2)^*\otimes_R(\wedge_A^e{\bar L}/{\bar L}^2)^* \iso 
 (\wedge^n_AL/L^2)^*\]
given by $1/{\bf t}\otimes 1/\bar{\bf u}\mapsto 1/({\bf t,u})$.

The role of the hyptheses on the ${\mathrm{Tor}}(\boldsymbol{-}, \bullet)$ functors in the
statement of \Pref{prop:trg-trans} below is the following: 
Suppose $S$ is a ring, $P$, $Q$ $S$-modules such that
${\mathrm{Tor}}^S_i(P,\,Q)=0$ for $i\neq 0$. Then $P\overset{\bL}{\otimes}_SQ$ is canonically
isomorphic to $P\otimes_SQ$ and we treat this as an identity, i.e., in this case we write
$P\overset{\bL}{\otimes}_SQ=P\otimes_SQ$. In particular if $J$ is an $S$-ideal generated by
a quasi-regular sequence, and ${\mathrm{Tor}}^S_i(P,\,S/J)=0$,
then  we have 
$P\overset{\bL}{\otimes}_S(\wedge^m_{S/J} J/J^2)^*=P\otimes_S(\wedge^m_{S/J} J/J^2)^*
=(P\otimes_SS/J)\otimes_{S/J}(\wedge^m_{S/J}J/J^2)^*$. In other words, if $\eG=\wit{P}$, the quasi-coherent sheaf on $W={\Spec{\,S}}$ corresponding to $P$,
and $u\colon Z=\Spec{\,S/J}\hookrightarrow W$ the natural closed immersion, we have
\begin{align*}
u^\btrg\eG[m] & =\bL  u^*\eG[m]\otimes_{\co_Z}(\eN^m[-m)] \\
& = \bL u^*\eG\otimes_{\co_Z}(\eN^m[-m])[m] \\
&=(\bL  u^*\eG\otimes_{\co_Z}\eN^m)[0]\\
&= (u^*\eG\otimes_{\co_Z}\eN^m_u)[0].
\end{align*}

\begin{prop}\label{prop:trg-trans} Let $R$, $A$, $B$, $I$, $L$, ${\bar L}$, ${\bf t}$, ${\bf u}$,
$i$, $j$, be as above with $({\bf t,u})$ being a quasi-regular sequence in $R$. Let $M$ be
an $R$-module,   $\eF=\wit{M}$, the quasi-coherent $\co_{\Spec{\,R}}$-module corresponding to $M$. Suppose we have
 ${\mathrm{Tor}}^R_i(M,A)={\mathrm{Tor}}^R_i(M,B)={\mathrm{Tor}}^A_i(M/IM, B) =0$ for
 $i\neq 0$.  Then the following diagram commutes
\[
 {\xymatrix{
 (j^*i^*(\eF)\otimes\eN^n_{ij})[0] 
  \ar@{=}[r] & (ij)^\btrg(\eF[n])  \ar[r]^{\Iso}_{\eta'_{ij}} & (ij)^!(\eF[n])  \\
 (j^*i^*(\eF)\otimes j^*\eN^d_i\otimes\eN^e_j)[0] \ar@{=}[r] 
 \ar[u]^{{\mathrm{via}}\,\alpha}_-{\,\rotatebox{-90}{\makebox[-0.1cm]{\Iso}}}
 & j^{\btrg}(i^{\btrg}(\eF[d])[e])  
 \ar[d]^{\,\rotatebox{-90}{\makebox[-0.1cm]{\Iso}}}_{{\text{{\em via $\eta'_j$ and $\eta'_i$}}}} & \\
 & j^!(i^!(\eF[d])[e]) \ar[r]^{\Iso} & j^!i^!(\eF[n]) \ar[uu]_{\,\rotatebox{-90}{\makebox[-0.1cm]{\Iso}}}
 }}
\]
\end{prop}
\stepcounter{subsubsection}
\proof For any noetherian ring $S$ and $S$-ideal $J$ generated by a quasi-regular sequence
${\bf v}=(v_1,\ldots, v_m)$, and every $S$-module $P$, we have, with ${\overline{S}}=S/J$, a map
of complexes 
\[w_{S, {\bf v}}=w_{S,{\bf v},P} \colon P[m]\otimes_S K^\bullet({\bf v}) \to 
P\otimes_S\wedge^m_{{}_{\overline{S}}}(J/J^2)^*[0]\]
defined on $0$-cochains by 
\[\phantom{XXXXX} x\mapsto (-1)^mx\otimes 1/{\bf v}, \qquad\,\,x\in P=(P[m]\otimes_SK^\bullet({\bf v}))^0.\]
Since $P[m]\otimes_SK^\bullet({\bf v})$ is a complex which is zero in positive degrees
and the complex $P\otimes_S\wedge^m_{{}_{\overline{S}}}(J/J^2)^*[0]$ is concentrated in degree 0,
the above recipe defines $w_{S,{\bf v}}$. Moreover, if ${\mathrm{Tor}}^i_S(P, \overline{S})=0$ for
$j\neq 0$, then in $\D({\mathrm{Mod}}_S)$ the image of the map $w_{S,{\bf v}}$ under
the localization functor is the composite
\begin{align*}
P[m]\otimes_S K^\bullet({\bf v}) \iso P[m]\overset{\bL}{\otimes}_S \overline{S}[-m]
 & \iso P[m]\overset{\bL}{\otimes}_S 
 (\wedge^m_{{}_{\overline{S}}}(J/J^2)^*[-m]) \\
 & \,\,\,\, = \,\,\, 
 P[m]\otimes_S  (\wedge^m_{{}_{\overline{S}}}(J/J^2)^*[-m])  \\
 & \,\,\,\, = \,\,\,
 P\otimes_S (\wedge^m_{{}_{\overline{S}}}(J/J^2)^*)[0]
\end{align*}
where the first arrow is {\eqref{iso:kosz-r1}} and the second arrow $1\otimes \lambda_{\bf v}[-m]$.
In other words the following diagram in $\D({\mathrm{Mod}}_S)$, consisting of isomorphisms, 
commutes :
\[
{\xymatrix{
& P[m]\otimes_SK^\bullet({\bf v})\ar[ddl]_{w_{S,{\bf v}}}^{\rotatebox{35}{\makebox[-0.1cm]{\Iso}}} \ar[ddr]^{\eqref{iso:kosz-rhom}}_{\rotatebox{-35}{\makebox[-0.1cm]{\Iso}}} & \\
& & \\
P\otimes_S (\wedge^m_{{}_{\overline{S}}}(J/J^2)^*)[0] \ar[rr]_{\eta_{S,\overline{S}}}^{\Iso} && 
\R\Hom_S(\overline{S},\,P[m])
}}
\]
(see \eqref{iso:etaRA} for the definition of  $\eta_{S,\overline{S}}$).
In view of these observations, as well the commutativity of \eqref{diag:ij-it} and 
\eqref{diag:koz-flat-it}, 
we are done if we show that the following diagram commutes where for convenience we
use $N$ to denote $\wedge^d_A(I/I^2)^*$.
\[
{\xymatrix{
M[n]\otimes_R K^\bullet({\bf t}, {\bf u}) \ar@{=}[d] \ar[rr]^{w_{R,({\bf t,u})}} 
& & M\otimes_R(\wedge^n_BL/L^2)^*[0]  \\
(M[d]\otimes_RK^\bullet({\bf t}))[e]\otimes_RK^\bullet({\bf u}) 
\ar[d]_{w_{R,{\bf t}}}^{\,\rotatebox{-90}{\makebox[-0.1cm]{\Iso}}} && \\
(M\otimes_R N)[e]\otimes_AK^\bullet(\bar{\bf u}) \ar[rr]_{w_{A,\bar{\bf{u}}}} 
&& M\otimes_R N\otimes_A(\wedge_B^e{\bar L}/{\bar L}^2)^*[0]
\ar[uu]_{(1\otimes\alpha)[0]}^{\,\rotatebox{-90}{\makebox[-0.1cm]{\Iso}}}
}}
\]
Indeed, we only have to check on $0$-cochains as we argued earlier. Let $m\in M$ be an
element. Regard it as a $0$-cochain of the complex on the northwest corner. Its image in 
$M\otimes_R(\wedge^d_A I/I^2)^*\otimes_A(\wedge_B^e{\bar L}/{\bar L}^2)^*$ in the southeast
corner under the composite $w_{A,\,\bar{\bf u}}\smcirc w_{R,\,{\bf t}}$ is
 $(-1)^nm\otimes 1/{\bf t}\otimes 1/\bar{\bf u}$, and its image in $M\otimes_R(\wedge^n_BL/L^2)^*$ 
 in the northeast corner (via $w_{R,({\bf t,u})}$) is $(-1)^nm\otimes 1/(\bf{t,u})$. This
 proves our assertion.
\qed

\subsection{}
This is a slightly different but related exploration of duality for compositions of closed immersions. 
So, as before, suppose $R$ is a noetherian ring, $I, J$ ideals in $R$, $I\subset J$, $I$ (resp.~$J$)
generated by a regular sequence $\{t_1,\dots, t_r\}$ (resp.~$\{g_1,\dots,g_r\}$). Note that the 
number of $t$'s equals the number of $g$'s. 

In this set-up, let $t_i=\sum_j u_{ij}g_j$, $A=R/I$, $B=R/J = A/{\overline{J}}$, where $\overline{J}=JA$.
Let $i\colon \Spec{\,A}\hookrightarrow \Spec{\,R}$, $j\colon \Spec{\,B}\hookrightarrow\Spec{\,R}$,
and $h\colon \Spec{\,B}\hookrightarrow \Spec{\,A}$ be the closed immersions corresponding
respectively to the surjections $R\twoheadrightarrow A$, $R\twoheadrightarrow B$, and
$A\twoheadrightarrow B$. Then $i\smcirc h=j$. We have a composite
\[\phi^!_h\colon h_*j^!\iso h_*h^!i^! \xrightarrow{\Tr{h}} i^!.\]
Using \eqref{diag:ij-it} and \eqref{iso:eta'-i} (the latter for $i$ and $j$), this corresponds to a map
\[\phi_h^\btrg\colon h_*j^\btrg \to i^\btrg.\]
In particular, for an $R$-module $M$, $H^n(\phi^\btrg_h(M))$ gives us a map
\stepcounter{thm}
\begin{equation*}\label{eq:dir-im}\tag{\thethm}
\phi_h \colon M\otimes_R\wI{B}{J} \lra M\otimes_R\wI{A}{I}.
\end{equation*}
From the definition of $U^\bullet(M)$ in \eqref{map:KU}, it is straightforward that for an
$R$-module $M$, the following diagram commutes.
\[
{\xymatrix{
\R\Homb_R(A,\,M)  \ar[d]_{\eqref{iso:kosz-rhom}} & \R\Homb_R(B,\,M)
\ar[d]^{\eqref{iso:kosz-rhom}} \ar[l] \\
M\otimes_RK^\bullet({\bf t})& M\otimes_RK^\bullet({\bf g}) \ar[l]^{U^\bullet} 
}}
\]
commutes. The unlabelled arrow is the one arising from the map $A\twoheadrightarrow B$. 
Unwinding all the definitions, we see that $\phi_h(m\otimes {\bf 1}/{\bf g}) = 
\det(u_{ij})m\otimes{\bf 1}/{\bf t}$. Moreover if $\Hr^r_J(M) \to \Hr^r_I(M)$ is the natural map
arising from the inclusion $\Gamma_J\hookrightarrow \Gamma_I$, the element
$\bigl[\begin{smallmatrix}m \\ g_1,\dots,g_r \end{smallmatrix}\bigr]$ maps to
$\bigl[\begin{smallmatrix}\det(u_{ij})m \\ t_1,\dots,t_r \end{smallmatrix}\bigr]$.

These are a well-known results (see, e.g., \cite[\S3,\,pp.71--72]{hk1} or 
\cite[Chap.\,III,\,\S7, pp.59--60]{ast117}). We record them for completeness.
It should be pointed out that  (ii) and (iii) in \Tref{thm:dir-im} do not need
${\bf g}$ or ${\bf t}$ to be regular sequences.

\begin{thm}\label{thm:dir-im} Let $R$, $I$, $J$, ${\bf t}$, ${\bf g}$ and $u_{ij}$ be as above.
\begin{enumerate}
\item Let $\phi_h\colon M\otimes \wI{B}{J} \to M\otimes \wI{A}{I}$ be as in \eqref{eq:dir-im}. Then
\[\phi_h\biggl(m\otimes\frac{\bf 1}{\bf g}\biggr) = \det(u_{ij})\,m\otimes\frac{\bf 1}{\bf t}.\]
\item If $\psi\colon\Hr^r_J(M) \to \Hr^r_I(M)$ is the natural map arising from the inclusion 
$\Gamma_J\hookrightarrow \Gamma_I$, then
\[
\psi\Biggl(\begin{bmatrix} m\\ g_1,\dots, g_r\end{bmatrix}\Biggr) =
\begin{bmatrix}\det(u_{ij})m\\t_1,\dots,t_r\end{bmatrix}
\]
\item If $\sqrt{I}=\sqrt{J}$, so that $\Hr^r_I(M)=\Hr^r_J(M)$, then
\[
\begin{bmatrix} m\\ g_1,\dots, g_r\end{bmatrix} =
\begin{bmatrix}\det(u_{ij})m\\t_1,\dots,t_r\end{bmatrix}
\]
\end{enumerate}
\end{thm}

\proof Part\,(i) has been established. We elaborate a bit on (ii) and (ii). 
Let $Z= \Spec{\,A}$ and $W=\Spec{\,B}$. The natural
maps $i_*i^\flat\to \R\Gamma_Z$ and $j_*j^\flat\to \R\Gamma_W$ fit into a commutative
diagram, as the reader can readily verify:
\[
{\xymatrix{
j_*j^\flat \ar[dd] \ar[r]^{\Iso} & i_*h_*h^\flat i^\flat \ar[dr]^{{\text{via}}\,\Tr{h}^\flat} & \\
& & i_*i^\flat \ar[dl] \\
\R\Gamma_W \ar@{^(->}[r] & \R\Gamma_Z &
}}
\]
Here $\Tr{i}^\flat$ is the map in \eqref{def:Tr-i-flat}. Parts (ii) and (iii) then follow from (i).
\qed

\begin{ack} We owe a huge debt of gratitude to Joseph Lipman. He is the driving force behind
getting us to write this paper and its sequel \cite{fub-verd}. We benefited enormously
from his detailed comments on many previous versions of this paper and its sequel.
We also benefitted from stimulating conversations with him on these
results.
\end{ack}
%{\,\rotatebox{-90}{\makebox[-0.1cm]{\Iso}}}

%\bibliographystyle{plain}
%\printindex

\end{document}